\documentclass{amsart}

 \newtheorem{theorem}{Theorem}[section]
 \newtheorem{corollary}[theorem]{Corollary}
 \newtheorem{lemma}[theorem]{Lemma}
 \newtheorem{proposition}[theorem]{Proposition}
 \theoremstyle{definition}
 \newtheorem{definition}[theorem]{Definition}
 \theoremstyle{remark}
 \newtheorem{remark}[theorem]{Remark}
 \newtheorem*{example}{Example}

 \numberwithin{equation}{section}

\usepackage{amssymb}
\newcommand{\1}{{1_\A}}
\newcommand{\A}{\mathsf{A}}
\newcommand{\ad}{{\rm ad}}
\newcommand{\alg}{\mathsf{A}}
\newcommand{\Alg}{\mathcal{A}}
\newcommand{\blg}{\mathsf{B}}
\newcommand{\B}{\mathcal{B}}
\newcommand{\C}{\mathbb{C}}
\newcommand{\bc}{\mathbb{C}}
\newcommand{\D}{\mathcal{D}}
\newcommand{\e}{\varepsilon}
\newcommand{\F}{\mathcal{F}}
\newcommand{\HA}{\mathcal{H}}
\newcommand{\hil}{\mathsf{h}}
\newcommand{\id}{\mbox{id}}
\newcommand{\m}{\cdot}
\newcommand{\N}{\mathbb{N}}
\newcommand{\pol}{{\rm Pol}\,}
\newcommand{\R}{\mathbb{R}}
\newcommand{\Pol}{\textup{Pol}}
\newcommand{\Irr}{\textup{Irr}}
\newcommand{\QG}{\mathbb{G}}
\newcommand{\QH}{\mathbb{H }}
\newcommand{\Ldag}{\mathcal{L}^{\dagger}}

\newcommand{\RepG}{\tu{Rep}_{\QG}}
\newcommand{\IrrG}{\tu{Irr}_{\QG}}
\newcommand{\PolG}{\tu{Pol}(\QG)}

\newcommand{\tu}{\textup}
\newcommand{\ot}{\otimes}

\newcommand{\Com}{\Delta}

\newcommand{\la}{\langle}
\newcommand{\ra}{\rangle}

\begin{document}

\title[L\'evy Processes on Quantum Permutation Groups]
 {L\'evy Processes on Quantum Permutation Groups}

\author{Uwe Franz}

\address{%
D\'epartement de math\'ematiques de Besan\c{c}on \\
Universit\'e de Franche-Comt\'e \\
16, route de Gray \\
25 030 Besan\c{c}on cedex, France}

\email{uwe.franz@univ-fcomte.fr}

\thanks{UF was supported by the ANR Project OSQPI (ANR-11-BS01-0008), by the Alfried Krupp Wissenschaftskolleg in Greifswald, and by MAEDI/MENESR and DAAD through the PROCOPE programme.}

\author{Anna Kula}
\address{%
Instytut Matematyczny \\
Uniwersytet Wroc{\l}awski, \\
pl. Grunwaldzki 2/4,
50-384 Wroc{\l}aw, Poland}
\email{Anna.Kula@math.uni.wroc.pl}

\thanks{\noindent AK was supported by the FUGA grant 2012/04/S/ST1/00102 of the Polish National Science Center.}

\author{Adam Skalski}
\address{%
Institute of Mathematics of the Polish Academy of Sciences \\
ul.~\'Sniadeckich 8, 00--656 Warszawa, Poland}

\thanks{\noindent AS is partially supported by the HARMONIA NCN grant
2012/06/M/ST1/00169}
\email{a.skalski@impan.pl}


\begin{abstract}
We describe basic motivations behind quantum or noncommutative probability, introduce quantum L\'evy processes on compact quantum groups, and discuss several aspects of  the study of the latter in the example of quantum permutation groups. The first half of this paper is a survey on quantum probability, compact quantum groups, and L\'evy processes on compact quantum groups. In the second half the theory is applied to quantum permutations groups. Explicit examples are constructed and certain classes of such L\'evy processes are classified.
\end{abstract}

\maketitle
\tableofcontents

\section{Introduction}
\label{sec-intro}

Classical probability, which concentrates on properties of random variables
with values in $\R$, $\C$ or $\R^n$, can be generalized to more abstract
algebraic structures  in -- at least -- two ways. On one hand, we can consider
random variables with values in ``non-classical'' structures, such as
infinite-dimensional vector spaces, groups (abelian or not),  Banach algebras,
etc. On the other hand, we can consider the usual distributions on $\R$
or $\C$  but replace the random variables by elements of some algebra, as it
is done for example in random matrix theory.

Several different motivations for studying such constructions exist. First of all, some specific physical systems require the use of the random as well as the algebraic descriptions to characterize them. As an example, the behavior of a sphere with fixed centre, kept in a liquid {\it with random flow}, can be analyzed in terms of probability distributions on $SO(3)$, the group of three dimensional rotations. Also, if we consider a system in which a given set of data is transformed sequentially by some machines, ruled by linear transformations, but committing errors from time to time, then the final result corresponds to the multiplication of an appropriate number of random matrices. Secondly, the probabilistic approach can also be helpful in studying deterministic problems. This is, for instance, the case with the method of constructing Banach spaces with specific properties as limits of randomly generated finite dimensional spaces
(see \cite{mankiewicz+tomczak03}). More recently, this idea was used to provide
counterexamples for some matricial inequalities related to entropy (see \cite{hayden+winter08}), solving a problem that has for many years resisted deterministic techniques. Last but not least, consideration of standard problems in the wider context yields a better understanding of the advantages and limitations of the classical theory.

Generalizations of the results which are valid for $\R$-valued random variables usually involve some technical problems. The basic tools, such as Fourier transform, need to be reformulated and often lose properties which make them easily applicable. It happens that some assumptions, trivially satisfied in classical case, have to be added to generalizations of well-known theorems. However, the methods of functional analysis remain useful in many cases. In fact, only in this context, all their power and beauty can appear. Much more information on the theme of probability on algebraic structures  can be found in relatively accessible books \cite{grenander63} and \cite{heyer04} as well as in a complete monograph \cite{heyer77}.

In this survey we combine both generalisations, i.e., we will consider probability laws on noncommutative algebras equiped with an additional structure that plays the r\^ole of a group multiplication. This allows us to define ``quantum L\'evy processes'', i.e., quantum stochastic processes with independent and stationary increments.

L\'evy processes on groups are group-valued stochastic processes with stationary and independent increments. They are also characterized as time- and space-homogeneous Markov processes, where time- and space-homogeneity means that the transition probabilities are invariant under time translation and space translation, with the latter defined by left multiplication by a group element (see, e.g.,  \cite{applebaum14,cecherini-silberstein+al08,diaconis88,heyer77,saloff-coste04}).

In the generalisation from classical to quantum (or noncommutative) probability, stochastic processes with values in a classical space are replaced by families of $*$-homomorphisms from a ``quantum space''  into a noncommutative probability space. Both the ``quantum space'' and the noncommutative probability space are certain $*$-algebras. For such a family one can define a notion of independence, and if the ``quantum space'' is actually a ``quantum group'' then there exists also a natural notion of increment. This leads to L\'evy processes on quantum groups. 

It was observed, cf.\ \cite{schurmann93}, that any L\'evy process corresponds to a functional on the ($*$-algebraic version of the) quantum group, called the generating functional,  and -- through a GNS-type construction -- to a triple consisting of a representation of the quantum group algebra, a cocycle and the functional itself. This allows to study probabilistic problems,  e.g.\ the classification of the L\'evy processes or the description of their symmetries, via objects of a more algebraic nature.

In this paper, we shall focus on processes on (``with values in'') ``free'' or ``quantum'' permutation groups, and more generally, quantum symmetry groups. Quantum permutation group is a noncommutative analogue of the algebra of functions on the group of permutations of $n$ elements. The latter can be seen as the universal algebra generated by $n^2$ commuting orthogonal projections that -- presented in a square matrix -- sum up to 1 in each row and each column. The quantum version is obtained by the liberation precedure, which bases on omitting the commutativity assumption in the relations generating the universal algebra. L\'evy processes on this quantum group are the central issue in this paper. 

The paper is organized as follows. Sections \ref{cqg-chap}-\ref{sec-markov} provide a short introduction to the theory of compact quantum groups and L\'evy processes on them. Section \ref{sec-q-perm} gives an introduction to Wang's quantum permutations groups. We describe their representation theory (i.e., the \emph{corepresentations} of the Hopf algebra ${\rm Pol}(S_n^+)$) and also introduce the more general notion of quantum symmetry groups.

Starting from Section \ref{sec-class}, we present different types of problems that appear in the theory of L\'evy processes on quantum groups and, as far as possible, discuss their solutions in the special case of $S_n^+$. In Section \ref{sec-class} we show that L\'evy processes on $S_n^+$ are completely characterised by their cocycle, and that these cocycles are determined by a family of $n$ vectors satisfying certain relations.

In Section \ref{sec-cohomology}, we describe all cocycles associated to certain representations of ${\rm Pol}(S_n^+)$, e.g., those coming from a classical permutation or from a Fourier-Hadamard matrix. In Section \ref{sec-sym-ad-inv}, we characterise the generating functionals of L\'evy processes on $S_n^+$ with symmetric (i.e., self-adjoint) Markov semigroup). We also give a classification of ``central'' L\'evy processes on $S_n^+$, i.e., those whose marginal distributions are invariant under the adjoint action.

Finally, in Section \ref{classicalstoch} we give an explicit (but informal) description of L\'evy processes on $S_n^+$ that can be constructed from a classical permutation. We also describe their Sch{\"u}rmann triples.

The survey part of this article (roughly corresponding to the content of Section \ref{cqg-chap} - \ref{sec-schurmann}) has appeared (in Polish) in a survey on probability on topological quantum groups written by the third named author (\cite{WM}). 

\section{Compact quantum groups}
\label{cqg-chap}
\subsection{Noncommutative (quantum) mathematics}
\label{noncomm_math_sec}

Although the notion of \emph{noncommutative mathematics} has no precise meaning, it has become a popular and useful name since 1994 (when the book ``Noncommutative geometry'' by Alain Connes, \cite{connes94} appeared). It does not denote any specific area of mathematics, but rather represents a research program, originally inspired  by quantum mechanics. The starting point is the following observation: the properties of a given space are often reflected by (and thus can be observed through) families of $\C$-valued functions on this space.

The choice of the family depends on the properties in question. For instance, if the space $X$ is equipped with a $\sigma$-algebra, then the natural candidate to study is the family of functions measurable with respect to the $\sigma$-algebra. In our case, $X$ will be endowed with a topological structure and the related family of functions will be $C(X)$, the space of all functions from $X$ to $\C$, continuous with respect to the given topology of $X$. Moreover, we shall assume that $X$ is compact, in which case $C(X)$ becomes  a normed space with the natural norm $\|f\|=\sup_{x\in X} |f(x)|$.

A crucial feature of passing from the set $X$ to the family $C(X)$ is that the
latter has the additional structure of an \emph{involutive unital Banach algebra} with respect to the poinwise
multiplication of functions, the involution $f^*(x)=\overline{f(x)}$ for $f\in
C(X)$ and $x\in X$, and the constant function $f(x)=1$ for all $x\in X$ as the
unit. Being an involutive Banach algebra means that $C(X)$ is complete with
respect to the norm $\|.\|$, that the multiplication is continuous in both
variables: $\|fg\| \leq \|f|\|g\|$, for $f,g\in C(X)$, and that the involution
is isometric: $\|f^*\|=\|f\|$ for $ f\in C(X)$.

Moreover, $C(X)$ satisfies an important condition, which connects its algebraic (the multiplication and the involution)
and analytic (the norm) aspects:
$$ \|f^* f \|=\|f\|^2, \quad f\in C(X). $$

We shall specify the class of algebras with this property in the following
definition. The motivation for considering this class is explained by next theorem.
\begin{definition}
 A (unital) involutive Banach algebra $\A$  is called a \emph{(unital) $C^*$-algebra  } if
$$ \|a^* a \|=\|a\|^2, \quad a\in \A. $$
\end{definition}

\begin{theorem}[Gelfand-Najmark, 1943]
Every commutative unital $C^*$-algebra $\mathsf{A}$ is isometrically isomorphic to the
algebra $C(X_\A)$ of continuous complex function on  some compact {\it (topological)} space $X_\A$. Moreover, if $Y$ is a compact {\it (topological)} space, then $\A$ and $C(Y)$ are isometrically isomorphic if and only if $X_\A$ and $Y$ are homeomorphic.
\end{theorem}

The space $X_\A$ mentioned in the Theorem can be explicitly described as the space of characters of $\A$, and the isomorphism between $\A$ and $C(X_\A)$ is given by the so-called \emph{Gelfand transform}, which can be seen as a generalization of the Fourier transform.

The Gelfand-Najmark Theorem states that there is a correspondence between objects in the category of compact topological spaces and commutative unital $C^*$-algebras. We still however need a closer look at the relation between morphisms of these categories.

If $X$ and $Y$ are compact spaces and if $T:X\to Y$ is continuous, then $T$
induces a map $j_T: C(Y)\to C(X)$ by the formula
$$j_T(f)=f\circ T, \quad f\in C(Y). $$
Notice that the arrows are ``reversed'':
\begin{eqnarray*}
 &X\stackrel{T}{\longrightarrow} Y \\
 & C(X) \stackrel{j_T}{\longleftarrow} C(Y)
\end{eqnarray*}
The map $j_T$ is a $*$-homomorphism which preserves the unit. Moreover, it turns out that each such map $\phi: C(Y)\to C(X)$ comes from a unique continuous transformation from $X$ to $Y$. We see that the study of the usual (``classical'') topology on compact sets corresponds to the theory of commutative unital $C^*$-algebras. So what can we call the ``noncommutative'' (more precisely, ``not necessarily commutative'') topology? The answer should be clear: ``noncommutative'' topology is the theory of ``not necessarily commutative'' unital $C^*$-algebras.

Considering generalizations of this kind is not only a purely mathematical habit. With the birth of quantum mechanics in the 1920s, it became clear that in the micro-world models functions (observables) describing physical quantities should be replaced by operators on a Hilbert space.
This makes a fundamental difference. First of all, even though such operators are still called \emph{observables} and we can still work with them (running calculations, doing experiments), there is no such thing as a phase space on which the observables are defined (the space $X$ disappears). What is more, we can still multiply operators acting on a Hilbert space, but this operation is no longer commutative.

The relation between the theory of $C^*$-algebras and the dicussion above is
explained by another Gelfand-Najmark Theorem: each $C^*$-algebra can be
realized as a closed $*$-subalgebra of $B({\sf h})$, the algebra of all
bounded operators on some Hilbert space ${\sf h}$. For example, the
commutative $C^*$-algebra of continuous functions on the interval $[0,1]$ can be identified with the algebra of the operators of multiplication by continuous functions acting on the Hilbert space $L^2([0,1])$. The (unital) $*$-homomorphisms from a $C^*$-algebra $\A$
to $B({\sf h})$ are called \emph{representations}. It is obvious that the possibility to realize each $C^*$-algebra as a sub-algebra of
$B({\sf h})$ has a strong influence on the methods used in developing ``noncommutative'' or ``quantum'' topology.

If we endow the $C^*$-algebra in question with a supplementary structure (a dense subalgebra of ``$C^\infty$-functions'', for instance), we can consider in a similar way noncommutative geometry or noncommutative measure theory. In the next section we shall show how to define the noncommutative analogue of group theory and in particular the notion of quantum (semi)groups -- for obvious reasons, we do not use here the term ``noncommutative groups'', reserved for classical, non-abelian groups.

\subsection{The group structure in the noncommutative setting}
\label{sec-cqg}

Before we give the definition of a compact quantum group, we need to understand
how classical groups can be described in the $C^*$-language. We start with the case where $G$ is a semigroup with the additional assumption that it is compact and has a neutral element -- this setup is sufficient for the introduction of the convolution of measures.

Let $G$ be a topological group, i.e.\ a Hausdorff space endowed with the continuous operations $\m : G \times G \to G$ and ${}^{-1}: G \to G$ which make $(G, \m)$ a group. We shall denote by $e$ the neutral element of $G$. In fact, each group can be seen as a topological group with respect to the discrete topology. In this paper, we assume that the group $G$ is \emph{compact} (as a topological space). In fact, it would be more natural to work with general locally compact groups. However, especially when passing to the noncommutative context, the compact framework yields many simplifications, while it still allows to consider a wide class of examples, containing, in particular, all finite groups and classical compact Lie groups.

Let $G$ be a compact semigroup with a neutral element -- it is a compact space equipped with an associative and continuous multiplication $\m : G \times G \to G$ and with a distinguished element $e\in G$, being the neutral element for $\m$. As we explained in the previous section, maps between the  spaces induce maps between the function algebras that go in the opposite direction. In this way we obtain maps
$$\Delta: C(G) \to C(G\times G), \quad \e : C(G) \to \C, $$
called \emph{comultiplication} and \emph{counit} (note that $\C$ is the algebra of function on a space with only one element\footnote{A distinguished element in the set $X$ can be treated as a transformation from the 1-point set $\{\mbox{pt}\}$ to $X$. The algebra of the continuous functions on $\{\mbox{pt}\}$ with values in $\C$ is just $\C$ itself, so we obtain the transformation $\:C(G) \to C($\{\mbox{pt}\}$)\approx \C$}).
They are defined by the formulas:
\begin{equation} \label{comm_comultipliaction}
 \Delta (f)(s,t)=f(s\m t), \quad f\in C(G), \; s,t \in G,
\end{equation}
and
$$\e (f)=f(e), \quad f\in C(G).$$
The comultiplication and the counit are $*$-homomorphisms preserving the unit
(i.e.\ the constant function $f(g)=1$ for all $g\in G$).

Keeping in mind that we want to replace $C(G)$ by a $C^*$-algebra $\A$, which is not related to any concrete set $G$, we need to find an alternative description of $C(G\times G)$. Here, the following observation can help us: if the group $G$ is finite, then the $*$-algebra $C(G\times G)$ is isomorphic to the (algebraic) tensor product $C(G)\odot C(G)$ (all tensor products appearing in the paper are considered over $\C$). In case $G$ is an arbitrary compact group, we have $C(G\times G) \approx C(G)\otimes C(G)$, where $\otimes$ denotes the \emph{minimal} tensor product in the category of $C^*$-algebras (see \cite{Mur}). Intuitively, this statement says that each continuous function on $G\times G$ can be uniformly approximated  by the finite sums of products of functions from $C(G)$, which is a direct consequence of the Stone-Weierstrass theorem.

The next challenge is to express, in the language of $\Delta$ and $\e$, the associativity of the multiplication in $G$ and the unit property of the distinguished element $e$. The associativity can be written down in the following symbolic way:
$$ \m \circ (\id_G \times \m)=\m \circ (\m \times \id_G),$$
where both sides are mappings from $G\times G$ to $G$. By analysing the definition of the comultiplication, we see that the relation above corresponds to the condition:
$$ (\id_{C(G)} \otimes \Delta) \circ \Delta = (\Delta \otimes \id_{C(G)}) \circ \Delta, $$
where on both sides we have $*$-homomorphisms from $C(G)$ to $C(G \times G \times G) \approx C(G) \otimes C(G) \otimes C(G)$. The formal proof is again an easy application of the Stone-Weierstrass Theorem.

In the same manner, the defining condition of the neutral element can be rewritten in terms of the counit as
$$ (\id_{C(G)} \otimes \e) \circ \Delta = (\e \otimes \id_{C(G)}) \circ \Delta =\id_{C(G)}.$$

In this way we obtained a full description of the structure of a compact semigroup
$G$ with a neutral element, in terms of the algebra $C(G)$ and the maps
$\Delta$ and $\e$ defined on it. Motivated by the previous section we arrive at the following definition.

\begin{definition}
An \emph{algebra of continuous functions on a compact quantum semigroup} is a unital $C^*$-algebra $\A$ with a $*$-homomorphism $\Delta: \A \to \A \otimes \A$ (called \emph{comultiplication}) which preserves the unit $1_\A$ and satisfies
$$ (\id \otimes \Delta) \circ \Delta = (\Delta \otimes \id) \circ \Delta. $$
If, moreover, $\A$ is equipped with a unital $*$-homomorphism
$\e:\A \to \C$ (called \emph{counit}) such that
$$ (\id \otimes \e) \circ \Delta = (\e \otimes \id) \circ \Delta =\id,$$
then we call $\A$ the \emph{algebra of continuous functions on a compact quantum semigroup with a neutral element}.
\end{definition}

If $\A$ is an algebra of continuous functions on a compact quantum semigroup which is commutative as a $C^*$-algebra, then there exists a unique compact semigroup $G$ such that $\A$ is isomorphic to $C(G)$ (in the sense that the isomorphism preserves the respective comultiplications, and also the counits, if the latter are defined on $\A$). We will often use the notation $\A=C(\mathbb{S})$ to emphasise that $\A$ should be considered as the algebra of functions on a
compact quantum semigroup $\mathbb{S}$. This approach has its advantages, but
requires some care. If $C(\mathbb{S})$ is not abelian, the ``semigroup''
$\mathbb{S}$ is not formally defined, it is rather an abstract concept
which helps to build the analogy with the classical theory. Still we will sometimes simply say: `let $\mathbb{S}$ be a compact quantum semigroup' and work with the algebra $C(\mathbb{S})$. Finally let us also remark that sometimes it is the pair $(\A, \Com)$ which is called a compact quantum semigroup.

\medskip

Now the question is how to define an algebra of continuous functions on a compact quantum \emph{group}. This
problem turns out to be much more difficult. A good definition should -- first of all --
generalize the definition in the classical setting (i.e.\ commutative
$C^*$-algebras in the new class of objects should correspond to algebras of
continuous functions on usual groups). Furthermore, it should be precise enough to become the foundation of a non-trivial theory. And finally, it should be sufficiently flexible to provide a large class of examples.

The first idea is that the theory of compact quantum groups can be developed just by considering the operation of taking the inverse and then by repeating the same reasoning as in the beginning of this section.  The aim would be to define a mapping on the algebraic level (a \emph{coinverse} or an \emph{antipode}) which would reflect the properties of the inverse operation (treated as a continuous transformation). Such a path leads to an interesting theory of \emph{(compact) Kac algebras} (first appearing in an early form in \cite{Kac}, see also the book \cite{Enock}), which was intensively studied starting form the 1970s. However, in the mid 1980s Woronowicz provided in \cite{woronowicz87} an example of a deformation of the classical algebra of continuous functions on the group $SU(2)$ for which the coinverse in only densely defined and can not be extended to the whole algebra.

In the same paper and in later lecture notes \cite{woronowicz98},
Woronowicz suggested a definition of a compact quantum group, which was soon
accepted as the one satisfying all the conditions mentioned above. It is based
on the following observation.

\begin{lemma}
 Let $(G,\cdot)$ be a compact semigroup (possibly without neutral element) for which the \emph{cancellation law} holds, i.e.\ for any $g,g',g''\in G$ either of the relations $g\cdot g'=g\cdot g''$ and $g' \cdot g = g'' \cdot g$ implies $g'=g''$. Then $G$ is a compact group.
\end{lemma}

The cancellation law can be easily translated into the languge of
algebras. Moreover, it turns out that in the Woronowicz theory -- like in the
lemma above -- the counit plays only a secondary role and often can be
defined only on the dense subset (like the coinverse). This is reflected in the following definition.

\begin{definition}[Woronowicz, 1987] \label{def_cqg}
A pair $(\A,\Delta)$, consisting of a $C^*$-algebra $\A$ and a unital $*$-homomorphism $\Delta: \A \to \A \otimes \A$, is called the \emph{algebra of continuous functions on a compact quantum group} if the comultiplication $\Delta$ is \emph{coassociative}:
 $$ (\id_A \otimes \Delta) \circ \Delta = (\Delta \otimes \id_A) \circ \Delta $$
and the \emph{quantum cancellation laws} hold:
$$ \overline{\rm Lin} \{\Delta(a)(1\otimes b): a,b\in \A \} = \overline{\rm Lin}\{\Delta(a)(b\otimes 1): a,b\in \A \} = \A \otimes \A.$$
(Here $\overline{\rm Lin}\, F$ denotes the closure of the set of all linear combinations of elements of $F$.) Again, we will usually write $\A=C(\mathbb{G})$ and call $\QG$ a compact quantum group.
\end{definition}

We shall also use the name the algebra of functions on a \emph{finite quantum group}
for the pair  $(\A,\Delta)$ which satisfies Definition \ref{def_cqg}, if $\A$
is finite-dimensional. Finite quantum groups can also be described in an axiomatic way with no use of the cancellation law, see \cite{vandaele97}.

\medskip

Before discussing benefits of the definition of Woronowicz, let us have a look on two examples.

\begin{example}
It follows from Pontriagin's theory of duality that each abelian (locally) compact group $G$ is isomorphic (in a canonical way) to its second dual $\hat{\hat{G}}$. By the dual group $\hat{G}$ we mean the set of characters on $G$, i.e. the homomorphisms from $G$ to the unit circle $\mathbb{T}$, furnished with the natural operations and topology (the pointwise multiplication of characters, the complex conjugate of a character as the inverse and the topology of uniform convergence on compact sets). Moreover, the theory establishes the duality between compact and discrete abelian groups. Namely, an abelian group $G$ is discrete if and only if $\hat{G}$ is compact. On the other hand, we already know that each commutative algebra of functions on a compact quantum group  is of the form $(C(G), \Delta)$, where $G$ is a compact group and $\Delta$ reflects the multiplication in $G$, according to formula \eqref{comm_comultipliaction}.
This suggest that a discrete group $\Gamma$ should also correspond to a
$C^*$-algebra, which leads to  a compact quantum group. In the special case where
$\Gamma$ is abelian, the resulting object should be the algebra of continuous functions on the dual group of $\Gamma$. We shall show now how to establish such a correspondence.

Let $\Gamma$ be a discrete group and let $\ell^2(\Gamma)$ denotes the Hilbert space of all square-summable functions on $\Gamma$. The space $\ell^2(\Gamma)$ is spanned by the orthonormal basis $\{\delta_g\}_{g\in \Gamma}$, where as usual $\delta_g(h)=1$ if $g=h$ and $\delta_g(h)=0$ otherwise. Then each element $g \in \Gamma$ defines the operator $\lambda_g : \ell^2(\Gamma) \to \ell^2(\Gamma)$ given by the linear extension of the formula
$$ \lambda_g(\delta_h) = \delta_{gh}, \quad h\in \Gamma.$$
Each $\lambda_g$ ($g\in \Gamma$) is a unitary operator and $\lambda_e=\id_{\ell^2(\Gamma)}$.
The closure of the $*$-algebra generated by $\{\lambda_g: g\in \Gamma\}$ in
$B(\ell^2(\Gamma))$ is denoted by $C^*_r(\Gamma)$ and called the \emph{reduced
  $C^*$-algebra}\footnote{One can also define the \emph{universal}
  $C^*$-algebra of the group, denoted by $C^*(\Gamma)$. The two algebras are
  isomorphic if and only if $\Gamma$ is amenable, cf. \cite{pedersen79}.} or
the \emph{group $C^*$-algebra} of $\Gamma$.

The mapping $\Delta$ defined by $\Delta(\lambda_g)=\lambda_g \otimes
\lambda_g$ extends (in a unique way) to a $*$-homomorphism from
$C^*_r(\Gamma)$ to $C^*_r(\Gamma)\otimes C^*_r(\Gamma)$ which preserves the
unit. The pair $(C^*_r(\Gamma), \Delta)$ is the algebra of functions on a compact quantum group. It is cocommutative, i.e.\ the comultiplication is invariant under the flip $\sigma:\A \otimes \A\to \A \otimes \A$,
$\sigma(a\otimes b)=b\otimes a$. One can show that each algebra of continuous functions on a compact quantum group which is cocommutative is essentially\footnote{Formally, this can
also be for example the universal $C^*$-algebra, see the previous footnote.} of this form. To follow the notational analogy with the classical Pontriagin duality, we sometimes write $C^*_r(\Gamma):=C(\hat{\Gamma})$, even if $\Gamma$ is not abelian.
\end{example}

\begin{example}
We shall explain here in few words the famous construction of Woronowicz from
\cite{woronowicz87}. Let $SU(2)$ denotes the classical Lie group of all
unitary complex $2\times 2$ matrices of determinant 1, equipped with the
topology induced by the natural embedding of $SU(2)$ into the topological
vector space $M_2(\C)\cong \C^4$. It is easy to check that
$$SU(2)=\left\{ \left (\begin{array}{cc} \alpha & -\bar{\gamma} \\ \gamma &\bar{\alpha} \end {array} \right):
\alpha, \gamma \in \C, |\alpha|^2+|\gamma|^2=1 \right \}.$$
This suggest that we can consider the algebra $C(SU(2))$ as the commutative
$C^*$-algebra generated by two elements $\alpha$ and $\gamma$, understood as
the functions which associate to a matrix from $SU(2)$ the values of the
coefficients in the left upper corner (in case of $\alpha$) and the left
bottom corner (for $\gamma$). These functions satisfy the relation
\begin{equation}\label{su2_determinant}
\alpha \alpha^*+\gamma \gamma^*=1.
\end{equation}

Since
$$ \left (\begin{array}{cc}
\alpha_1 & -\bar{\gamma}_1 \\ \gamma_1 &\bar{\alpha}_1
\end {array} \right) \cdot \left (\begin{array}{cc}
\alpha_2 & -\bar{\gamma}_2 \\ \gamma_2 &\bar{\alpha}_2
\end {array} \right)
=
\left (\begin{array}{cc} \alpha_1\alpha_2-\bar{\gamma}_1 \gamma_2 & \ldots \\ \gamma_1 \alpha_2+\bar{\alpha}_1 \gamma_2 & \ldots  \end {array} \right),$$
the comultiplication in $C(SU(2))$ is determined by the conditions
\begin{equation}\label{su2_comultiplication}
  \Delta(\alpha)=\alpha \otimes \alpha - \gamma^* \otimes \gamma, \quad
\Delta(\gamma)=\gamma \otimes \alpha + \alpha^* \otimes \gamma.
\end{equation}
The algebra of functions on the compact quantum group $SU_q(2)$ arises as a deformation of the algebra $C(SU(2))$ in the sense that it is a $C^*$-algebra generated by two operators satisfying a modification of the conditions \eqref{su2_determinant} and \eqref{su2_comultiplication} depending on a parameter $q\in [-1,0)\cup (0,1]$. More precisely, $C(SU_q(2))$ is the universal unital $C^*$-algebra generated by $\alpha$ and $\gamma$ subject to the following relations
\begin{eqnarray*}
 & \alpha^* \alpha+\gamma^* \gamma=1, \quad \alpha \alpha^*+q^2\gamma \gamma^*=1, \\
 & \gamma^* \gamma = \gamma \gamma^*, \quad \alpha\gamma = q \gamma \alpha, \quad  \alpha\gamma^* = q \gamma^* \alpha
\end{eqnarray*}
with the comultiplication extended uniquely to a unital homomorphism from the formulas
\begin{equation*}\label{suq2_comultiplication}
  \Delta(\alpha)=\alpha \otimes \alpha - q\gamma^* \otimes \gamma, \quad
\Delta(\gamma)=\gamma \otimes \alpha + \alpha^* \otimes \gamma.
\end{equation*}
Note that the universal $C^*$-algebra generated by a family of generators and relations is the completion of the corresponding universal $*$-algebra (say $\Alg$). This completion is taken with respect to the norm defined as the supremum over the norms of all representations $\pi:\Alg \to B(\hil)$, where $\hil$ is a Hilbert space.  I.e., for an element $a$ of the universal $*$-algebra associated to the given family of generators and relations, we set
\[
\|a\| = \sup_{(\pi,\hil)} \|\pi(a)\|_{B(\hil)}.
\]
It is not difficult to see that this supremumm is always finite for the relations given above, because the relations imply that $\alpha$ and $\gamma$ are cofficients of a unitary matrix.

It turns out that the pair $(C(SU_q(2)),\Delta)$ is a compact quantum group in
the sense of Definition \ref{def_cqg}, and (of course) if $q=1$, then we
recover the algebra of continuous functions on $SU(2)$ group.

After the Woronowicz's work appeared, it was shown that similar quantum
deformations can be constructed for every compact simple Lie group. We refer
the reader to the books \cite{klimyk+schmudgen97} and
\cite{korogodski+soibelman98} for more details on the relation of these
constructions with classical Poisson geometry and on deformations of the
universal enveloping algebras of simple Lie groups \`{a} la Drinfeld and Jimbo.
\end{example}

\section{Convolution}

\subsection{Probability measures and convolution on (classical) groups}
\label{sec-prob-cl-gr}

The notion of convolution of measures on the real line $\R$ usually appears very early and plays an important role in any course on classical probability. It exploits obviously the additive structure of $\R$ and can be generalized to ``decent'' measures defined on an arbitrary (locally compact) semigroup.

For a compact group  $G$, let us denote by $M(G)$ the set of all regular Borel measures on $G$, and by
$C(G)$ the algebra of complex continuous functions on $G$. Thanks to the Riesz Theorem, we can identify measures from $M(G)$ with continuous functionals on $C(G)$; the correspondence $\mu \mapsto \psi_\mu$ is given by the formula:
\[
\psi_\mu(f)=\int_G f(s) \; {\rm d}\mu(s), \quad f\in C(G).
\]
In particular, if $\mu$ is a probability measure, then $\psi_\mu$ is a positive functional of norm 1.
Moreover, the identification $M(G) \approx (C(G))^*$ induces the natural norm and the weak (more precisely, weak$^*$) topology on $M(G)$. We shall say that $(\mu_i)_{i\in I}$ \emph{tends weakly} to $\mu$  if for any function $f\in C(G)$ we have $\lim_{i\in I} \mu_i(f) = \mu(f)$.

\medskip

The convolution of two measures $\mu, \nu\in M(G)$ is defined, via the identification above, by the formula
\begin{equation} \label{clasical_conv}
(\mu \star \nu)(f):= \int_{G} \int_G f(s\cdot t) \;{\rm d} \mu (s)\;{\rm d} \nu(t),
\quad f\in C(G).
\end{equation}
It is easy to check (ignoring for the moment the fact thar $\R$ is not a compact group) that if $(G, \cdot)=(\R, +)$, the definition above
corresponds to the formula for the classical convolution of probability
measures on the real line. In general, the convolution of two (probability) measures from $M(G)$ is again a (probability) measure
and the associativity of the multiplication on $G$ implies that the convolution is associative, too.
So $(M(G), \star)$ is a Banach algebra.

The notion of the convolution is closely related to
stochastic independence. Denote by $\mathcal{L}(X)$ the \emph{law} or
\emph{distribution} of a random variable $X$, i.e.\ the measure induced by $X$
on its range. If two  $\R$-valued random variables $X$ and $Y$ are
independent, then the distribution of their sum is given by the formula $\mathcal{L}(X+Y) =\mathcal{L}(X) \star \mathcal{L}(Y)$.
This will be important in Section \ref{quantum_levy_sec}.

Let us observe that, in the discussion above, we did not make any use of the
operation of taking the inverse in the group -- everything we have said about the convolution until
now remains true if $G$ is a \emph{compact semigroup}.

Once the convolution defined, we can ask about the existence of invariant measures. The following classical theorem due to Haar and Weyl answers this question.
\begin{theorem}
 Let $G$ be a compact group. Then there exists a unique regular Borel probability measure $\mu_h$, which is left and right invariant:
$$ \mu_h (gS) = \mu_h (S) = \mu_h (Sg), \quad g\in G, S\in {\rm Bor} (G).$$
We shall call it the \emph{Haar measure}.
\end{theorem}
The invariance condition can be rewritten in terms of the convolution as
$$    \mu_h \star \nu = \mu_h = \nu \star \mu_h \quad \mbox{for any} \quad \nu\in M(G).$$
Indeed, the last displayed formula implies the one before simply by considering the Dirac mass supported at a point $g \in G$ as the measure $\nu$.

Let us finally mention that the existence of the Haar measure on the group $G$ opens the door to the development of harmonic analysis on $G$ and to the study of its representation theory (for more details, see for example \cite{HR}).

\subsection{Convolution on compact quantum (semi-)groups and Haar state}
\label{sec-conv-haar}

From the discussion in the previous subection,
we know that positive measures on a compact group $G$ (or a semigroup $S$) correspond, via the Riesz Theorem, to continuous positive functionals on $C(G)$ (resp.\ $C(S)$), and that these functionals are of norm 1 for probability measures. Therefore, natural counterparts of probability measures in the framework of unital $C^*$-algebras
are \emph{states}, i.e.\ continuous functionals $\omega:\A\to\C$ which satisfy
$$\omega(\1)=1 \quad \mbox{and} \quad \omega(a^*a) \geq 0, \; a\in \A. $$

Let us denote by $\A^*$ the space of all continuous functionals on a $C^*$-algebra $\A$. A careful look at the formulas \eqref{clasical_conv} and \eqref{comm_comultipliaction} suggests the following generalization of the classical convolution to the quantum semigroups:
\begin{definition}
 Let $\mathbb{S}$ be a compact quantum semigroup, put $\A=C(\mathbb{S})$ and let $\mu, \nu\in \A^*$. The \emph{convolution} of $\mu$ and $\nu$, denoted by $\mu \star \nu$, is defined by
\begin{equation} \label{conv_quantum}
 \mu \star \nu = (\mu \otimes \nu) \circ \Delta.
\end{equation}
\end{definition}

The pair $(\A^*, \star)$ is a Banach algebra, and the counit $\e$ of $\A$ (if exists) is the unit of $(\A^*, \star)$, i.e. $\mu \star \e = \e \star \mu = \mu$ for all $\mu \in \A^*$. Also, it is easy to check that the convolution of two states on $\A$ is a state on $\A$.

Given the definition of the convolution of states, we can introduce the invariance condition in the quantum framework and, consequently, we may ask about the existence of an analogue of the Haar measure. The positive answer to this question, provided by Woronowicz in \cite{woronowicz98}, is a crucial observation in the theory of compact quantum groups.

\begin{theorem}
 Let $\QG$ be a compact quantum group. Then there exists a unique state $h\in C(\QG)^*$ such that
$$ h\star \omega = \omega \star h = h$$
for any state $\omega \in C(\QG)^*$. The state $h$ is called the \emph{Haar state}.
\end{theorem}

As in the classical case, the existence of the Haar state is of fundamental importance. In particular it is a starting point of the quantum version of the classical Peter-Weyl theory (see \cite{HR}). To make it more precise we introduce a notion of a (finite-dimensional, unitary)  representation of a compact quantum group.

\begin{definition} \label{fdurep}
A \emph{unitary representation of a quantum group} $\QG$ (of dimension $n$) is a unitary matrix $u=(u_{ij})_{i,j=1}^n\in M_n(C(\QG))$ such that
\[ \Com(u_{ij}) = \sum_{k=1}^n u_{ik} \ot u_{kj}, \;\;\;i,j=1,\ldots,k.\]
If the set $\{u_{ij}:i,j=1,\ldots,k\}$ generates $C(\QG)$ as a $C^*$-algebra, we say that $u$ is a \emph{fundamental unitary representation} of $\QG$.
\end{definition}

We write $\RepG$ to denote the set of all finite-dimensional representations of $\QG$. If $u \in \RepG$, then each element in the linear span of $\{u_{ij}:i,j =1, \ldots,n\}$ is called a \emph{coefficient} of $u$.  The key consequence of the quantum Peter-Weyl theorem proved in \cite{woronowicz98} is that $\PolG$, the space spanned inside $C(\QG)$ by the coefficients of all representations in $\RepG$, is a dense unital $*$-subalgebra of $C(\QG)$, equipped with the natural Hopf $*$-algebra structure (with the coproduct inherited from $C(\QG)$). In particular it admits an antipode $S:\PolG \to \PolG$; if the antipode satisfies the condition $S^2 =\tu{id}_{\PolG}$ (equivalently, the Haar state is tracial), we say that $\QG$ is \emph{of Kac type}.

Two representations $u_1,u_2 \in \RepG$ of dimension $n$ are said to be \emph{unitarily equivalent} if there exists a scalar unitary matrix $V \in M_n$ such that $u_1 = V u_2 V^*$. Further given $u_1,u_2 \in \RepG$ one can construct in an obvious way their direct sum $u_1 \oplus u_2 \in \RepG$. We say that $u \in \RepG$ is \emph{irreducible} if it is not unitarily equivalent to a direct sum of two non-zero representations. The set of all equivalence classes of irreducible representations of $\QG$ will be denoted by $\IrrG$; if we choose for each $\alpha \in \IrrG$ a representative $u^{(\alpha)} \in M_{n_{\alpha}} (C(\QG))$ then the set $\{u^{(\alpha)}_{ij}:i, j =1, \ldots, n_{\alpha}, \alpha \in \IrrG\}$ is a linear basis in $\PolG$. Moreover the quantum Peter-Weyl theorem says that any representation in $\RepG$ decomposes as a finite direct sum of irreducible representations and that the Haar state is \emph{faithful} on $\PolG$ (i.e.\ if $a \in \PolG$ and $h(a^*a)=0$, then $a=0$). The algebra $\PolG$ carries in a sense all the information about $\QG$; in particular it can be described abstractly as a particular type of a Hopf $*$-algebra, and it can be always completed at least in two, potentially different ways to the algebra of the type $C(\QG)$. The first completion is simply the universal completion of $\PolG$, as discussed when the example of $SU_q(2)$ was mentioned, and the second is the GNS completion (\cite{Mur}) of $\PolG$ with respect to the Haar state. When the two completions yield isomorphic $C^*$-algebras, the quantum group $\QG$ is said to be \emph{coamenable}. Thus for example if $\Gamma$ is a discrete group, $\hat{\Gamma}$ is coamenable if and only if $\Gamma$ is amenable.

The existence of the Haar state also allows  us to define the (say left) regular representation of $\QG$ (which is in general infinite-dimensional, so one needs to modify suitably Definition \ref{fdurep}) and to generalize the Tannaka-Krein duality, which leads to the construction and further a characterization of the dual of a compact group in terms of a certain tensor-type category (in fact given by the representation category of $\QG$). Tannaka-Krein duality for compact quantum groups was established in \cite{woronowicz88}. The starting point for the development of this concept is the observation that given $U,V \in \RepG$ one can construct the tensor product $U \ot V \in \RepG$. If one tensorizes two irreducible representations, one in general gets a reducible representation (so a non-trivial direct sum of irreducibles). The description of this operation on the level of a map from $\IrrG \times \IrrG$ into the ring $\mathbb{Z}[\IrrG]$ is often called the description of \emph{fusion rules} for $\QG$; we will see an example of that later on.

Let us finally remark that sometimes the name \emph{corepresentation of the pair $(C(\QG), \Com)$} is used instead of that of a representation of $\QG$.

\subsection{Convolution semigroup of states}
\label{sec-conv-sg}

In this section we shall assume that $\mathbb{S}$ is a compact quantum semigroup with a neutral element and put $\A=C(\mathbb{S})$.

\begin{definition} \label{conv_semigroup}
 A \emph{convolution semigroup of states} on $\A$ is a family of states $(\omega_t)_{t\geq 0}$ which satisfies:
\begin{enumerate}
 \item $\omega_{s+t}=\omega_s \star \omega_t$ for $s,t\geq 0$,
 \item $\omega_0 (a) =\e(a)$ for $a\in \A$,
 \item $\omega_t(a) \stackrel{t\to 0^+}{\longrightarrow} \omega_0(a)$ for $a\in \A$.
\end{enumerate}
\end{definition}

The main tool in the study of convolution semigroups on quantum semigroups is --
like in case of semigroup of operators on Banach spaces -- the analysis of the
generator.

\begin{definition} \label{generating_functional_def}
Let $(\omega_t)_{t\geq 0}$ be a convolution semigroup of states on $\A$.
The functional $L: \A \supset \D_L \to \C$ given by
\begin{eqnarray*}
L(a) &:=& \lim_{t\to 0^+} \frac{\omega_t(a)-\e(a)}{t}, \\
a\in \D_L &:=& \left\{ a\in \A: \mbox{ there exists } \lim_{t\to 0^+} \frac{\omega_t(a)-\e(a)}{t} \right\},
\end{eqnarray*}
is called the \emph{generating functional} of $(\omega_t)_{t\geq 0}$.
\end{definition}

The classical L\'evy-Khinchin theorem states that each convolution semigroup of probability measures $(\mu_t)_{t\geq 0}$ on $\R^n$ is given by the following equation (in fact the equation describes the characteristic functions of $\mu_t$ for $t\geq 0$): for $\vec{u}\in \R^n$ and $t\geq 0$ we have
\begin{eqnarray*}
\lefteqn{\phi_t(\vec{u}) := \int_{\R^n} \exp (i\vec{u}\cdot \vec{x}) \,{\rm d}\mu_t (\vec{x})} \\
&=& \exp \left( t \big (
i\vec{b}\cdot \vec{u} - \frac12 \vec{u}\cdot A\vec{u}
+ \int_{\R^n \setminus \{0\}} \hspace{-.5cm}[\exp (i\vec{u}\cdot \vec{y})-1-i\vec{u}\cdot \vec{y} \mathbf{1}_{B_n}(\vec{y})] \,{\rm d}\nu (\vec{y})
\big )\right ),
\end{eqnarray*}
where $\vec{b}\in \R^n$, $A\in M_n(\R)$ is a symmetric, positive definite matrix, $\nu$ is a \emph{L\'evy measure} on $\R^n \setminus \{0\}$ (that is, $\int_{\R^n} (\|y\|^2\wedge 1) \,{\rm d}\nu (\vec{y})<\infty$), $B_n$ denotes the unit ball $\{\vec{x}\in \R^n : \|\vec{x}\|\leq 1\}$, and $\mathbf{1}_{B_n}$ its indicator function. Conversely, every function of the form as above gives rise to a convolution semigroup of measures. In this sense, the L\'{e}vy-Khinchin formula provides a parametrization of all such semigroups on $\R^n$.

The L\'evy-Khinchin formula was later extended by Hunt from $\R^n$ to all Lie
groups (possibly non-compact). Next, using the fact that each (locally)
compact group is an inverse limit of a sequence of Lie groups, one can obtain the
corresponding characterization for all (locally) compact groups. Hunt's
formula uses the L\'evy-Khinchin expression  rewritten in terms of the
generating functional of the semigroup $(\mu_t)_{t\geq 0}$ -- the partial derivatives appearing therein in the case $\R^n$ are replaced by the respective vector fields on the Lie group $G$. The precise formulations and proofs of the relevant theorems, as well as many interesting associated facts can be found in the book \cite{heyer77}.

\medskip

In the quantum group framework we would also like to characterize convolution semigroups of states on quantum (semi-)groups.
 The problem becomes more tractable, if the
continuity condition in Definition \ref{conv_semigroup} is strengthened.

\begin{definition} \label{regular_def}
 A convolution semigroup of states $(\omega_t)_{t\geq 0}$ is called \emph{regular} if $\omega_t \stackrel{t\to 0^+}{\longrightarrow} \omega_0=\e$ uniformly.
\end{definition}

The following result was obtained in \cite{lindsay+skalski08}. It can be seen as a noncommutative generalization of the classical \emph{Sch\"{o}nberg correspondence}, describing the relation between conditionally positive definite functions and their positive definite exponentials.
\begin{theorem} \label{regular_tw}
Each regular convolution semigroup of states is of \emph{Poisson type}, which means that there exists a bounded functional $L\in \A^*$ such that for any $t\geq 0$ we have
\begin{equation}\label{semigroup_exp}
 \omega_t = \exp_\star (tL) := \sum_{n=0}^\infty \frac{t^n}{n!} L^{\star n} .
\end{equation}
The functional $L$ satisfies $L(1_\A)=0$ and $L|_{{\rm Ker}\, \e} \geq 0$.

Conversely, for each such functional $L$ the family $(\omega_t)_{t\geq 0}$ defined by the formula \eqref{semigroup_exp} is a regular convolution semigroup of states.
\end{theorem}

The proof of this theorem is based on earlier results by Sch\"urmann, cf.\ \cite{schurmann93}. We shall comment more on it in Section \ref{quantum_levy_sec}, and especially in Section \ref{sec-schurmann}. The theorem shows also that many interesting convolution semigroups are not regular, e.g., the heat semigroup (i.e.\ the convolution semigroup of a Brownian motion), since their generating functionals are unbounded.

The theorem is also valid for locally compact quantum groups. This is particularly important in the case of discrete quantum groups (these are the quantum groups ``dual'' to compact quantum groups, see \cite{lindsay+skalskiMZ} for the exact definition) as explains the main result from \cite{lindsay+skalskiMZ} which  we cite below.

\begin{theorem}
Each convolution semigroup of states on discrete quantum group is regular and thus of Poisson type.
\end{theorem}

Theorem \ref{regular_tw} states that if $(\omega_t)_{t\geq 0}$ is regular, then generating functional introduced in Definition \ref{generating_functional_def} is exactly the bounded functional appearing in Theorem \ref{regular_tw}, and furthermore its domain $\D_L$ is the whole algebra $\A$.

In general, we have the following theorem (cf. \cite{lindsay+skalskiMZ}).

\begin{theorem}
 Let $\mathbb{S}$ be a compact quantum semigroup with a neutral element and let  $(\omega_t)_{t\geq 0}$ be a convolution semigroup of states on $\A=C(\mathbb{S})$. Then
\begin{enumerate}
 \item $\D_L$ is dense in $\A$,
 \item $1 \in \D_L$ and $a^*\in \D_L$ provided $a\in \D_L$,
 \item $L(1)=0$ and $L(a)=\overline{L(a^*)}$ for $a\in \D_L$,
 \item $L(x) \geq 0$ if $x \in \mathsf{A}_+ \cap \textup{Ker}(\epsilon)\cap \mathcal{D}_L$.
\end{enumerate}
The functional $L$ determines uniquely the semigroup $(\omega_t)_{t\geq 0}$. Moreover, if $\D_L=\A$, then $(\omega_t)_{t\geq 0}$ is regular.
\end{theorem}

For arbitrary compact quantum semigroups we do not know the full characterization of (densely defined) functionals $L: \D_L \to \C$ which generate a
convolution semigroup of states. One should expect that, in addition to the
algebraic conditions above, the generating functionals need to satisfy some
analytic conditions (like in the Hille-Yosida
theorem
for operator case; note however that these might be better seen on the level of the associated \emph{Markov semigroup}, which we will introduce in Section \ref{sec-markov}). A satisfactory characterization is known if  we deal with the \emph{full} version of the algebra $C(\QG)$, where $\QG$ is a compact quantum group.
\emph{Full} here means that the $C^*$-algebra $C(\QG)$ is the universal $C^*$-algebra completion of $\PolG$.
In this case the problem becomes purely algebraic and can be
solved with the methods from \cite{schurmann93}. In particular, the domain $\D_L$ always contains the (Hopf) $*$-algebra $\PolG$. In fact starting from Section \ref{sec-schurmann} we will pass to such an algebraic context.

\section{L\'evy Processes on Compact Quantum Groups}
\label{quantum_levy_sec}

A L\'evy process on a group $G$ is a stochastic process indexed by the
nonnegative real numbers whose increments are independent and identically distributed.
Such processes play an essential role in the probabilistic description of physical systems whose phase-space has some natural symmetry property.
In particular, the king of all stochastic processes, Brownian motion in $\R^n$, is a L\'evy process.

The exact definition  can be stated in the following way (note that although in general we treat in our article only the compact case, the locally compact version looks exactly the same).
\begin{definition} \label{Levy_class}
 A L\'evy process on a compact group $G$ is a family $X=\{X_t: t\in \R_+\}$ of random variables on some probability space $(\Omega, \mu)$ with values in $G$ which satisfies the following conditions:
\begin{enumerate}
\item $X_{r,t}=X_{r,s}X_{s,t}$ ($0\leq r\leq s \leq t$), where $X_{s,t}=X_s^{-1}X_t$ denotes the increment of the process from time $s$ to time $t$;
 \item $X_0=X_{t,t}\equiv e$ (almost surely) for all $t\geq 0$;
 \item the increments on the intervals of the same length have identical distributions (the distribution of $X_{s,t}$ depends only on $t-s$);
 \item the increments corresponding to the intervals whose interiors are  disjoint are independent;
 \item the distribution of $X_t$ tends weakly to $\delta_e$ (Dirac measure at point $e$) when $t\to 0^+$, which means that
$$ \int_G f(g){\rm d}_{X_t} (g) \stackrel{t\to 0^+}{\longrightarrow} f(e) \quad \mbox{for all} \;f\in C(G). $$
\end{enumerate}
\end{definition}

The conditions (1)-(2) in the definition above are trivially satisfied, we
note them only to emphasize the important role of the increments of the process
$(X_t)_{t\geq 0}$, and in view of the  generalization of this notion to the quantum
case, to be discussed below. Let us note that the other three conditions can easily be adopted to the
case when $G$ is only a semigroup with unit. Indeed, if we denote by
$\mu_{s,t}$ the distribution of the increment $X_{s,t}$, then condition (3)
says that $\mu_{s,t}=\mu_{0,t-s}$ for all $t\geq s\geq 0$. Conditions (4)
and (5) state that the family $(\mu_t)_{t\geq 0} :=(\mu_{0,t})_{t\geq 0}$ is a
convolution semigroup of measures on $G$. It also follows from them that all finite dimensional distributions of
the process $X$ are determined by the convolution semigroup $(\mu_t)_{t\geq 0}$.
This explains why we are so interested in convolution semigroups. Moreover, the correspondence between L\'evy processes and their convolution semigroups implies that the L\'evy-Khinchin formula classifies all L\'evy processes on $\R$. For more on this subject the reader is referred to the books \cite{applebaum04b} and \cite{bertoin96}.

In order to define quantum L\'evy processes, we need to translate the notion
of random variables into noncommutative language. Recall that a random variable
is a measurable map from some probability space $(\Omega, \mathcal{F},\mu)$ to
a measurable space $(G,\mathcal{X})$. The noncommutative analogue of probability space (the \emph{noncommutative probability space}, or rather the space of functions on it) is a pair $(\B, \Phi)$, where $\B$ is a $*$-algebra with unit and $\Phi$ is a state on $\B$.
Let us note that, since the topological structure is not important at the
moment, we do not assume that $\B$ is a $C^*$-algebra.\footnote{However, it often happens in the noncommutative probability that properties of $\B$ as the ``measurable space'' are considered. In such case, one usually assumes that $\B$ is a \emph{von Neumann algebra} and that $\Phi$ is a \emph{normal} state on $\B$.}
We follow again the idea of ``reversing arrows'', i.e.\ we take
as the quantum analogue of a random variable the unital
$*$-homomorphism induced by it. This means that a quantum random variable over a
noncommutative probability space  $(\B, \Phi)$ with
values in a compact quantum semigroup $\mathbb{S}$ is a unit-preserving
$*$-homomorphism from the algebra $C(\mathbb{S})$ to the noncommutative probability space $(\B, \Phi)$.

In what follows we will again write $\A:=C(\mathbb{S})$ to simplify the notation. In order to define a quantum L\'evy process it will be more convenient to use explicitly the increments. That is why we consider as a stochastic process a double-parameter family $(j_{s,t})_{0\leq s\leq t}$ of unital $*$-homomorphisms $j_{s,t}: \A \to \B$, where $(\B,\Phi)$ is a noncommutative probability space. The aim now is to rephrase the conditions (1)-(5) from Definition \ref{Levy_class}. It turns out that in the $C^*$-algebraic context already the first of them is problematic.

Let us consider the classical situation: let $X,Y:\Omega \to  G$ be
random variables and let $j_X, j_Y:C(G) \to F(\Omega)$ be the homomorphisms
induced by $X$ and $Y$ respectively, according to the idea from Section
\ref{noncomm_math_sec}. By $F(\Omega)$ we denote the family of all functions
on $\Omega$. How to describe then the homomorphism corresponding to the
function $XY:\Omega \to G$ defined by $XY(p)=X(p)Y(p)$ ($p\in \Omega$)? In
general, there is no satisfactory answer to such question. However, if $X$ and $Y$ are
independent, then we expect that $j_{XY}$ will be the ``convolution`` of $j_X$ and $j_Y$, which -- in the quantum case -- can be defined for $*$-homomorphisms as $j_X \star j_Y := (j_X \otimes j_Y)\circ \Delta$. But then the next two problems appear. First, it is not clear how to extend the definition of $j_X \otimes j_Y$ to the $C^*$-tensor product $\A\otimes \A$ when $\B$ is not a $C^*$-algebra. And even if it was a $C^*$-algebra, the convolution $j_X \star j_Y$ would take values in $\B \otimes \B$, not in $\B$. In the  algebraic theory of quantum L\'evy processes as introduced in \cite{accardi+schurmann+waldenfels88} (and developed in \cite{schurmann93} and \cite{franz06}, for example), $\A$ is assumed to be a $*$-bialgebra,
which we will denote by $\mathcal{A}$ here.
Then the comultiplication of this ``algebra of functions on an algebraic quantum group'' $\mathcal{A}$ takes values in the algebraic tensor product $\mathcal{A} \odot\mathcal{A}$, and so $(j_X \otimes j_Y)\circ \Delta$ goes from $\mathcal{A}$ to $\B \odot \B$. Now it is enough to apply the natural multiplication $m_\B: \B \odot \B \to \B$, $m_\B(a\otimes b)=ab$,  to get a well defined mapping
$j_X \star^a j_Y:=m_\B \circ (j_X \otimes j_Y)\circ \Delta :\mathcal{A} \to \B$. In the topological case we treat in this section it is impossible -- for a $C^*$-algebra $\mathsf{C}$ the multiplication $m_\mathsf{C} : \mathsf{C} \odot \mathsf{C} \to \mathsf{C}$ almost never (except in the finite dimensional case) extends continuously to the $C^*$-algebraic tensor product $\mathsf{C} \otimes \mathsf{C}$.

To solve this problem, it remains to refer to the afore-mentioned fact that the complete information about a classical L\'evy process is encoded in the corresponding convolution semigroup. This leads to the following definition, suggested in \cite{lindsay+skalski08}.
\begin{definition} \label{Levy_process_on_A}
 Let $\A=C(\mathbb{S})$ be the algebra of continuous functions on a compact quantum semigroup with a neutral element. A \emph{L\'evy process} (\emph{in the weak sense}) on $\mathbb{S}$ over a noncommutative probability space $(\B, \Phi)$ is a family
$j:=(j_{s,t})_{0\leq s\leq t}$ of unital $*$-homomorphisms $j_{s,t}: \A \to \B$ such that the functionals $\omega_{s,t}:= \Phi \circ j_{s,t}$ (which describe the distributions of the increments from time $s$ to $t$, $0\leq s\leq t$) satisfy:
\begin{enumerate}
 \item $\omega_{r,t}=\omega_{r,s} \star \omega_{s,t}$ for $0\leq r\leq s\leq t$,
 \item $\omega_{t,t}=\e$ for $t\geq 0$,
 \item $\omega_{s,t}=\omega_{0,t-s}$ for $0\leq s\leq t$,
 \item for any $n\in \mathbb{N}$, $a_1,\ldots, a_n \in \A$ and pairwise disjoint intervals \\$[s_1,t_1),\ldots,[s_n,t_n)$
 $$ \Phi(\prod_{i=1}^n j_{s_i,t_i} (a_i)) = \prod_{i=1}^n \omega_{s_i,t_i} (a_i),$$
 \item $\omega_{0,t} (a) \to \e(a)$ when $t\to 0^+$, for all $a\in \A$.
\end{enumerate}
\end{definition}

We say that the process is defined in the weak sense, beacause the conditions stated in the definition describe only the properties of the \emph{distribution} of the random variables $j_{s,t}$. In particular, since the operators of the form $j_{s,t}$ need not  commute, we cannot reconstruct  all finitely dimensional distributions of the process $j$ just from the knowledge of the functionals $\omega_{s,t}$. Let us note that condition (4) reflects the classical independence on the disjoint time-intervals. In fact, what we have in the definition is the so-called \emph{tensor independence} -- it turns out that in the noncommutative world also other notions of independence are possible (\cite{franz06}).

The functionals $\omega_{s,t}=\Phi\circ j_{s,t}$ are called the \emph{marginal distributions} of the process $j$. The family $(\omega_t)_{t\geq 0} := (\omega_{0,t})_{t\geq 0}$ is a convolution semigroup, which we shall call the \emph{convolution semigroup}
of the process $j$. Two L\'evy processes on $\mathbb{S}$ are called \emph{equivalent} if they have the same convolution semigroups. Finally, a process $j$ is called \emph{regular} if its convolution semigroup is regular (see Definition \ref{regular_def}).

\medskip

How to construct and study quantum L\'evy processes? One of the possible tools
is the theory of quantum stochastic integration, founded in mid-eighties of the 20th century by Hudson and Parthasarathy (cf. \cite{hudson+parthasarathy84}), which generalizes the theory of stochastic It\^o integrals. Below we present some
very basic aspects of H-P theory, which will be essential for quantum L\'evy
processes. More details and motivations on this subject can be found in
\cite{parthasarathy92} and \cite{meyer93} (the latter book is written for an
audience with a good background in classical probability).

Let $\hil$ be a fixed Hilbert space (which describes the number of independent \emph{noises}\footnote{For example, if $\hil=\C^n$, then $\mathcal{F}_\hil$ is isomorphic to $L^2(\mathbb{P})$, where $\mathbb{P}$ is the Wiener measure corresponding to $n$-dimensional Brownian motion.}) and let $\F:=\F_{\hil}$ denotes the \emph{symmetric Fock space over $L^2(\R_+,\hil)$}. For our needs (more information may be found in Subsection \ref{subsec-real-Fock-space}), it is enough to know that $\F$ is a Hilbert space equipped with the natural \emph{filtration}: for any $0\leq s< t < \infty$ we have the decomposition
\begin{equation} \label{F_decomp}
 \F \cong \F_{0,s} \otimes \F_{s,t} \otimes \F_{t,\infty}
\end{equation}
and a family $\{\sigma_t : t\geq 0\}$ of endomorphisms of the algebra $B(\F)$ which plays the role of shifts with respect to the decomposition \eqref{F_decomp}:
$$ \sigma_t(B(\F_{0,s})) = B(\F_{t,t+s}).$$
We shall often identify the algebra $B(\F_{0,s})$ with a subalgebra of $B(\F)$ by assuming that operators from $B(\F_{0,s})$ act trivially ``after time $s$'' -- they have the form $T\otimes I_{\F_{s,\infty}}$ with respect to the decomposition \eqref{F_decomp}.

The Fock space $\F$ contains a distinguished vector $\Omega$, called the \emph{vacuum vector} (the terminology comes from the quantum field theory). This vector induces the state $\omega_\Omega$ on $B(\F)$ defined by $\omega_\Omega (T):= \langle \Omega, T\Omega\rangle_\F$, which is called the \emph{vacuum state} and which has particularly nice factorization properties with respect to the decomposition \eqref{F_decomp}.

\begin{definition}
 An \emph{adapted process} on a $C^*$-algebra $\A$ with values in $B(\F)$ is a family $\{j_t:\A \to B(\F_{0,t}): t\geq 0\}$ of unital $*$-homomorphisms.
\end{definition}

The fact that a process $j$ is adapted means intuitively that the operators $j_t(a)$ act non-trivially only ``before time $t$''.

\begin{definition} \label{LP_on_Fock}
Let $\mathbb{S}$ be a compact quantum semigroup with a neutral element and let
$\A:=C(\mathbb{S})$. We call a \emph{L\'evy process}
on $\mathbb{S}$ over the Fock space an adapted process $\{j_t:\A \to B(\F_{0,t}): t\geq 0\}$ for which
$$ j_{s+t}(a) = (j_s \otimes (\sigma_s\circ j_t)) \circ \Delta (a), \quad j_0(a)=\e(a) I_{\F} \quad \mbox{for all} \; a \in \A.$$
\end{definition}

Let us observe that $j_s \otimes (\sigma_s\circ j_t)$ takes values in the
algebra $B(\F_{0,s}) \otimes B(\F_{s,t})$, which -- by the decomposition
\eqref{F_decomp} can be identified with a subalgebra of $B(\F_{0,s+t})$, so the problem
mentioned before Definition \ref{Levy_process_on_A} disapears. This
observation allowed to define L\'evy
processes in the strong sense in \cite{lindsay+skalski08}, using the notion of Arveson's product system of Hilbert spaces, cf. \cite{arveson03}.

It turns out that all regular L\'evy processes
can be realized in the form described in Definition \ref{LP_on_Fock}. An
algebraic version of this result was proved by Sch\"{u}rmann
(cf. \cite{schurmann93}, see also the following section). The analytic version presented below comes from \cite{lindsay+skalski08}.

\begin{theorem} \label{ReconstructionThm - Analytic}
 Every regular L\'evy process on a compact quantum semigroup with a neutral element has an equivalent realization on a Fock space.
\end{theorem}

\begin{proof} (a sketch of the proof) Let $\A:=C(\mathbb{S})$ and let $j:=(j_{s,t}:\A \to \B)_{0\leq s\leq t}$ be a regular L\'evy process on $\mathbb{S}$. Since the semigroup of this L\'evy process is regular by definition, it has a generator  $L:\A \to \C$ by Theorem \ref{regular_tw}. The properties of $L$ imply that we can define a sesquilinear form (linear with respect to the second and anti-linear with respect to the first argument) on the quotient space ${\rm Ker}\, \e/\mathcal{N}_L$, where $\mathcal{N}_L=\{a\in {\rm Ker}\, \e;L(a^*a)=0\}$, by the formula
$$
\langle [a],[b] \rangle := L\big((a-\e(a)\1)^*(b-\e(b)\1)\big), \quad a,b \in \A,
$$
 and this form will be positive. Via a standard construction (a modification of the Gelfand-Najmark-Segal construction), the form $ \langle \cdot,\cdot \rangle$ leads to the Hilbert space $\mathsf{k}$ with a (unital) representation $\rho:\A \to B(\hil)$ and a $(\rho,\e)$-derivation $\eta:\A \to \hil$:
$$ \eta(ab)=\rho(a)\eta(b)+\eta(a) \e(b), \quad a,b\in \A.$$
The triple $(\rho,\eta,L)$, is the so-called \emph{Sch\"{u}rmann triple} (cf.\ Definition \ref{def-triple} in the following section) of the L\'evy process. It defines a map $\varphi:\A \to B(\C \oplus \hil)$ by the matrix-type formula
$$
\varphi(a)=\left [
\begin{array}{cc} L(a) & |\eta(a)\rangle^\dagger \\ |\eta(a)\rangle & \rho(a)-\e(a)I_{\hil} \end{array}
\right ], \quad a\in \A.
$$

Thanks to the algebraic properties of the triple $(\rho,\eta,L)$ we can prove that $\varphi$ is completely bounded. This, in turn, implies that the quantum stochastic differential equation
\begin{equation} dk_t=k_t \star d\Lambda_\varphi(t); \quad k_0(a)=\e(a)I_\mathcal{F}, \quad a \in \A\end{equation}
has a unique solution $k:=\{k_t^\varphi: t\geq 0\}$, which is an adapted process on $\A$ with values in $B(\mathcal{F})$.

Then we prove that $k$ is in fact a L\'evy process on the Fock space $\mathcal{F}$ (parametrized with one parameter). Putting
$$k_{s,t} := \sigma_s \circ k_{t-s}, \quad 0\leq s \leq t,$$
we get a L\'evy process on $\A$ over the quantum probability space
$(B(\mathcal{F}),\omega_\Omega)$, which is equivalent to the process $j$.
\end{proof}

The sketch we presented above, being far from the full proof, does not explain (in particular) what is a \emph{quantum stochastic differential equation}. In the literature the latter is often referred to by the acronym \emph{QSDE}, in analogy to \emph{SDE}, the stochastic differential equations. The theory of QSDEs which studies the existence and the properties of solutions of quantum stochastic differential equations is explained, for example, in the afore-mentioned books \cite{parthasarathy92,meyer93}. A modern approach to this problem, which is based on the theory of operator spaces (cf. \cite{effros+ruan00})  and which emphasises the natural role of complete boundedness in the sketch of the proof of the previous theorem, can be found in \cite{lindsay05}. Finally, let us note that the construction we described in the second part of the sketch plays a crucial role in the proof of Theorem \ref{regular_tw}.

Theorem \ref{LP_on_Fock} is the starting point for the theory of \emph{quantum stochastic convolution cocycles}, that is families of completely bounded linear operators $\{l_t:\A \to B(\mathcal{F}_{0,t}):t\geq 0\}$ which, for each $a\in \A$, satisfy
$$ l_{s+t}(a)=(l_s \otimes (\sigma\circ l_t))\circ \Delta(a), \quad l_0(a)=\e(a)I_\mathcal{F}.$$
For more information on this subject -- the description of completely positive dilations of quantum stochastic convolution cocycles of quantum L\'evy processes and the proof of an existence of the approximation of quantum L\'evy processes by the quantum random walks -- we refer to \cite{lindsay+skalski08} and the references there.

A short interesting introduction to classical and quantum L\'evy processes can be found in  \cite{applebaum04a}, see also \cite[Lecture 6: Two L\'evy Paths to Quantum Stochastics]{applebaum05}.

\section{Sch\"urmann triples}
\label{sec-schurmann}

In this section we pass to the algebraic framework of quantum L\'evy processes, as introduced by Accardi, Sch\"urmann and von Waldenfels in \cite{accardi+schurmann+waldenfels88} and later developed by Sch\"urmann in \cite{schurmann93}. Here the basic object is a $*$-bialgebra $\Alg$ (i.e.\ a unital $*$-algebra equipped with the coproduct $\Com:\Alg \to\Alg \odot \Alg$ and the counit $\e: \Alg \to \bc$ satisfying the natural axioms). Our key example will be $\Alg=\PolG$ for $\QG$ being a compact quantum group; this in fact allows some further simplifications. Note also that in what follows we will often use Sweedler's notation, writing $\Delta(a)= a_{(1)} \otimes a_{(2)}$ for $a \in \Alg$, and $(id \otimes \Delta)(\Delta(a))= a_{(1)} \otimes a_{(2)}\otimes a_{(3)}$ for $a \in \Alg$.

Note that we can define \emph{states} on $\Alg$ simply as functionals on $\Alg$ which map $1$ to $1$ and take non-negative values on $a^*a$ for $a \in \Alg$ -- and then define convolution semigroup of states exactly as it was done in Section \ref{sec-conv-sg}.

The key result of Sch\"urmann is the establishment of the correspondence between L\'evy processes on $\Alg$ and \emph{generating functionals} on $\Alg$. By the former we mean families of unital $*$-homomorphisms $j_{s,t}:\Alg \to \B$, $0\geq s\geq t$, satisfying the conditions from Definition \ref {Levy_process_on_A}, whereas by the latter we mean linear functionals $L:\Alg \to \bc$ which vanish at $1$ ($L(1)=0$), are hermitian ($L(a^*) = \overline{L(a)}$, $a \in \Alg$) and positive on the kernel of the counit ($L(a^*a) \geq 0$ if $a\in \Alg$, $\e(a) =0$). Of course this definition should be compared to Definition \ref{generating_functional_def}, and an important part of this comparison is the following, algebraic version of Theorem \ref{regular_tw}.

\begin{theorem}
Let $(\omega_t)_{t \geq 0}$ be a convolution semigroup of states on a $*$-bi\-al\-ge\-bra $\Alg$. Then  the formula
\[
L(a) = \lim_{t\to 0^+} \frac{\omega_t(a)-\e(a)}{t}, \;\;\; a \in \Alg,
\]
defines a generating functional on $\Alg$. Conversely, if $L$ is a generating functional then putting
\[\omega_t (a)= \exp_\star (tL) (a):= \sum_{n=0}^\infty \frac{t^n}{n!} L^{\star n}(a),\;\;\; a \in \Alg, t \geq 0,\]
defines a convolution semigroup of states $(\omega_t)_{t \geq 0}$ on $\Alg$.
\end{theorem}

Note that above we talk only about the pointwise convergence, and the existence of the limit in the definition of $L$ is a consequence of the Fundamental Theorem of Coalgebra (see for example Section 4.4 of \cite{schurmann93}).

The hard part of the proof of the one-to-one correspondence between generating functionals and L\'evy processes on a $*$-bialgebra is the construction of a L\'evy process from a generating functional. Sch\"urmann's original reconstruction theorem uses quantum stochastic processes and kernel calculus. The general scheme of the proof follows that of the proof of Theorem \ref{ReconstructionThm - Analytic} (which of course was itself modelled on Sch\"urmann's result). An important part of the proof is related to a formation of the so-called Sch\"urmann triple associated to $L$, which we define next (but which was already mentioned in the proof of Theorem \ref{ReconstructionThm - Analytic}). For a pre-Hilbert space $D$ we define  $\Ldag(D)$  to be the collection of all linear operators mapping $D$ to $D$ which admit adjoints whose restrictions to $D$ leave $D$ invariant. It is easy to see that $\Ldag(D)$ is a unital $*$-algebra.

\begin{definition}\label{def-triple}
The triple $(\rho,\eta, L)$ is a \emph{Sch\"{u}rmann triple} on a $*$-bialgebra $\Alg$ if $\rho:\Alg \to \Ldag(D)$ is a unital $*$-homomorphism,  $\eta: \Alg\to D$ is a linear map which is a $\rho$-$\e$-cocycle, i.e.\
\[ \eta(ab) = \rho (a) \eta(b) + \eta(a) \e(b), \;\;\; a, b \in \Alg,\]
and $L:\Alg \to \bc$ is a hermitian functional which is a coboundary of $\eta$, i.e.\
\[ L(a^*b) = \la \eta(a), \eta(b) \ra + \e(a^*) L(b) + L(a^*) \e(b).\]
\end{definition}

Note that both the concept of a Sch\"urmann triple and a generating functional do not involve the coproduct, so can be considered also for a general unital $^*$-algebra with a character. Sch\"urmann triples are in one-to-one correspondence with generating functionals and therefore with L\'evy processes (modulo the natural notions of equivalence). From a given generating functional $L$ one can construct a Sch\"urmann triple $(\rho,\eta,L)$ by the construction outlined in the proof of Theorem \ref{ReconstructionThm - Analytic}. Conversely, given a Sch\"urmann triple $(\rho,\eta,L)$, $L$ is conditionally positive, since $L(a^*a) = \langle\eta(a),\eta(a)\rangle$ for $a\in {\rm ker}\varepsilon$, and therefore a generating functional. Uniqueness up to unitary equivalence of the Sch\"urmann triple associated to a given generating functional follows as for the  Gelfand-Najmark-Segal construction.

If $\mathcal{A}={\rm Pol}(\mathbb{G})$ is the Hopf $*$-algebra associated to a compact
quantum group $\mathbb{G}$, then $\rho$ takes bounded operators as values. Therefore we can extend the operators to the Hilbert space closure $\hil=\overline{D}$ of $D$, and
replace $D$ by $\hil$. In the following we shall assume that $\eta$ is
surjective (i.e.\ it is defined with values in a pre-Hilbert space); this guarantees that $\rho$ and $\eta$ are unique up to bijection of the underlying pre-Hilbert space which preserves the inner product. When we pass to the completion and consider $\eta$ as taking values in a Hilbert space, then it will in general no longer be surjective. In this case we shall assume that its range is dense, this assures again that  $\rho$ and $\eta$ are unique (up to unitary equivalence).

\medskip
 We now introduce two important classes of L\'evy processes.

\begin{definition}\label{def-poisson}
A L\'{e}vy process is called \emph{Poisson process} if the restriction of its generator $L$ to $K_1 = {\rm ker}\,\e $
coincides with a positive functional,
i.e.\ if there exist a state $\omega:\mathcal{A}\to\mathbb{C}$ and
$\lambda>0$ such that
\[
L(a) = \lambda\big(\omega(a) - \varepsilon(a)\big).
\]
\end{definition}
We often say that $L$ is \emph{of Poisson type}, meaning that it is the generating functional of a Poisson process. It is important to remark that, if $(\rho,\eta,L)$ is a Sch\"{u}rmann triple, then if $L$ is of Poisson type then the associated cocycle $\eta$ is a \label{def_coboundary}
\emph{1-coboundary} (in other terminology \emph{inner}), i.e., there exists a vector $v\not=0$ such that
\[
\eta(a) = \big(\rho(a)-\e(a){\rm id}_\hil\big) v, \quad a\in \mathcal{A}.
\]

The following theorem is a combination of the results of \cite{lindsay+skalski08}. We sketch the proof for the convenience of the reader.

\begin{theorem}\label{thm-unbounded}
Let $\QG$ be a compact quantum group. A generating functional $L:\PolG\to \C$ is of Poisson type if and only if it is bounded with respect to the universal (enveloping) $C^*$--norm on $\PolG$.
\end{theorem}
\begin{proof}
The forward implication is easy, as the so-called CQG property of the Hopf $*$-algebra $\PolG$ (see \cite{DiK})) implies that states on $\PolG$ admit continuous extensions to the states on the universal completion (which we will denote $C^u(\QG)$), as can be deduced from example from Lemma 8.7 in \cite{lindsay+skalski08}.

Assume then that $L$ admits a bounded extension to a functional on $C^u(\QG)$. Then the same GNS-type construction as in the algebraic case leads to a Sch\"urmann triple
$(\rho, \eta, L)$, with all the maps defined on $C^u(\QG)$ (note that $\e$ extends to a character on $C^u(\QG)$). In particular $\eta:\alg \to \hil$ is a $\rho$-$\e$-derivation. By a standard `matrix-corner trick' we can view $\eta$ as a corner of a derivation with respect to a direct sum representation $\rho\oplus \e$; then an application of the theorem of Sakai (\cite{Sakai}, see also \cite{Ringrose}) shows that $\eta$ is bounded (say its norm is equal to $M$). Let now $\mathsf{J}\subset C^u(\QG)$ denote the kernel of $\e$ treated as a character on $C^u(\QG)$ and let $(e_{i})_{i \in \mathcal{I}}$ be a contractive approximate unit of $\mathsf{J}$. Consider the projection $P:C^u(\QG) \to \mathsf{J}$ given by $P(a) = a - \e(a) 1$, $a \in C^u(\QG)$, and let $\tilde{\eta}= \eta|_{\mathsf{J}}$. Finally let $n \in \N$ and let $(a_{j,k})_{j,k=1}^n \in M_n(\mathsf{J})$. We have then
\begin{align*} \tilde{\eta}^{(n)} \left((a_{j,k})_{j,k=1}^n\right) &= (\tilde{\eta}(a_{j,k}))_{j,k=1}^n = \lim_{i \in I} (\tilde{\eta}(a_{j,k} e_i))_{j,k=1}^n
\\&=  \lim_{i \in I} (\pi(a_{j,k})\tilde{\eta} (e_i))_{j,k=1}^n = \lim_{i \in I} (\pi(a_{j,k}))_{j,k=1}^n \textup{diag}( \eta(e_i)) \\&=  \lim_{i \in I} \pi^{(n)}\left((a_{j,k})_{j,k=1}^n\right)\textup{diag}( \eta(e_i)),
\end{align*}
where the upper indices $^{(n)}$ denote respective matrix liftings. As $\pi^{(n)}$ is a $*$-homomorphism, hence contraction, we obtain that the norm of $\tilde{\eta}^{(n)}: M_n (C^u(\QG)) \to M_n (\hil) \approx B(\bc^n;\hil^{\oplus n})$ is not greater than $M$; in other words $\tilde{\eta}$ is \emph{completely} bounded. So is $P$, and thus further $\eta = \tilde{\eta} \circ P$. The main result of \cite{Der} implies that $\eta$ is a 1-coboundary, i.e.\ there exists $v \in \hil$ such that
$\eta(a) = \rho(a)v - \e(a) v, \;\; a \in C^u(\QG)$. Using the fact that 
$\{a^*b:a, b \in ker\,  \epsilon\}$ is dense in $ker\, \epsilon$ one can show that 
$L = \langle v, \rho(a)v - \epsilon(a) v \rangle$, so $L$ is of Poisson type.
\end{proof}

Note that the above proof is not using the comultiplication structure on $\PolG$ (nor on $C^u(\QG)$) at all, it could be cast completely in the context of a $C^*$-algebra with a character, as was done in the Appendix of \cite{lindsay+skalski08}.

\medskip

\begin{definition}\label{def-gaussian}
A L\'{e}vy process is called \emph{Gaussian process} if the related generating functional $L$ vanishes on all triple products of elements from $\ker \e$.
\end{definition}
The definition can be rephrased in terms of other elements of the Sch\"{u}\-rmann triple. With the notation,
$$K_m:={\rm Lin} \, \{a_1\cdot \ldots \cdot a_m; a_j\in \ker \e\},$$
one can easily show (see \cite{schurmann93}) that the following conditions are equivalent:
\begin{enumerate}
\item $L|_{K_3}=0$,
\item $\eta|_{K_2}=0$,
\item $\rho(a)=\e(a)1$ for $a\in \alg$.
\end{enumerate}
A special case of Gaussian process is a \emph{drift}, for which $L|_{K_2}=0$ or, equivalently, $\eta=0$.

The generalization of the notion of Gaussian process to processes on bialgebras has been given in \cite{schurmann93}. It is crucial if one looks for an analogue of the L\'{e}vy-Khinchin formula in the quantum group framework. Indeed, the classical version of formula can be regarded as a decomposition result, which reveals that every L\'{e}vy process breaks up into one component related to Brownian motion (with continuous paths) and the remaining component, which contains the 'jumps'.

In our framework, an analogous idea leads to the question whether any generating functional on a given compact quantum group admits the extraction of the maximal Gaussian part ---  such functionals are said to admit a \emph{L\'{e}vy-Khinchin decomposition}. The positive answer to this decomposition problem in the case of $SU_q(2)$ was provided in \cite{schurmann+skeide98}. Recently, the question has again received some attention \cite{FKLS,DFKS,FGT}, it turned that such a decomposition is not always possible.

\medskip

We will now show that Gaussian L\'evy processes with tracial marginal distributions are commutative, see Proposition \ref{prop-gauss-trace} below. Actually, we will see later that the quantum permutation groups do not have any Gaussian L\'evy processes anyway, see Proposition \ref{prop-no-Gauss}.

Let us start with a lemma that characterizes L\'evy processes whose marginal
distributions are tracial.
\begin{lemma}\label{lem-trace}
Let $(\omega_t)_{t\ge 0}$ be a convolution semigroup of states on a
$*$-bialgebra $\mathcal{A}$ with Sch\"urmann triple $(\rho,\eta,L)$. Then the
following are equivalent.

\begin{itemize}
\item[(i)]
$\omega_t$ is tracial for all $t\ge 0$.
\item[(ii)]
$L$ is tracial.
\item[(iii)]
There exists an anti-unitary operator $J:D\to D$ on $D=\overline{\eta(\mathcal{A})}$ such that
\[
\eta(a^*) = J\eta(a)
\]
for all $a\in\mathcal{A}$.
\end{itemize}
\end{lemma}
\begin{proof}
(i)$\Leftrightarrow$(ii)

 Let $L$ be a tracial generating functional.

Let us first recall that the convolution of two
tracial functionals is tracial,
\[
(f\star g)(ab) = f(a_{(1)}b_{(1)})g(a_{(2)}b_{(2)}) =
f(b_{(1)}a_{(1)})g(b_{(2)}a_{(2)})=(f\star g)(ba),
\]
for $f,g$ two tracial linear functionals on $\mathcal{A}$,
$a,b\in\mathcal{A}$.

The counit is also tracial, since it is a homomorphism with values in the
commutative algebra $\C$. Therefore
\[
\omega_t(ab) = \sum_{n=0}^\infty \frac{t^n}{n!} L^{\star n}(ab) =
\sum_{n=0}^\infty \frac{t^n}{n!} L^{\star n}(ba) = \omega_t(ba)
\]
for all $a,b\in\mathcal{A}$ and all $t\ge 0$.

For the converse, suppose that $\omega_t$ is tracial for all $t\ge 0$. Then
\[
L(ab) = \left.\frac{{\rm d}}{{\rm d}t}\right|_{t=0} \omega_t(ab)
=\left.\frac{{\rm d}}{{\rm d}t}\right|_{t=0} \omega_t(ba) = L(ba)
\]
for all $a,b\in\mathcal{A}$

(ii)$\Leftrightarrow$(iii)

 Let $L$ be tracial. Then we have
\begin{eqnarray*}
\langle \eta(a),\eta(a)\rangle &=& - \e(a^*)L(a)+ L(a^*a) - \e(a)L(a^*) \\
&=&  -
\e(a)L(a^*)+ L(aa^*) - \e(a^*)L(a) = \langle \eta(a^*),\eta(a^*)\rangle
\end{eqnarray*}
and therefore the formula  $J\eta(a)=\eta(a^*)$, defines correctly a conjugate linear map $J:\eta(\mathcal{A})\to
\eta(\mathcal{B})$ which extends to an anti-unitary operator $J:D\to D$.

Conversely, if there exists an anti-unitary operator $J:D\to D$ with
$J\eta(a)=\eta(a^*)$, then
\begin{eqnarray*}
L(ab) &=& \e(a)L(b) + \langle\eta(a^*),\eta(b)\rangle + L(a)\e(b) \\
&=& \e(b)L(a) +
\langle J\eta(b),J \eta(a^*)\rangle + L(b)\e(a) \\
&=& e(b)L(a) +
\langle \eta(b^*),\eta(a)\rangle + L(b)\e(a) = L(ba)
\end{eqnarray*}
for all $a,b\in\mathcal{A}$.
\end{proof}

\begin{proposition}\label{prop-gauss-trace}
Let $(j_t)_{t\ge 0}$ be a Gaussian L\'evy process on a $*$-bialgebra
$\mathcal{A}$ whose marginal
distributions are tracial. Then its marginal distributions vanish on the
commutator ideal and the process is therefore equivalent to a commutative
L\'evy process on $\mathcal{A}$.
\end{proposition}
\begin{proof}
Let $\mathcal{I}$ be the $*$-ideal generated by the commutators of elements of
$\mathcal{A}$, i.e.\
\[
\mathcal{I} = {\rm Lin}\{a(bc-cb)d:a,b,c,d\in\mathcal{A}\}.
\]
The set $\mathcal{I}$ is also a coideal (i.e.\ $\Delta(\mathcal{I}) \subset \mathcal{I} \otimes \mathcal{I}$) and the quotient $\mathcal{A}/\mathcal{I}$ is a
commutative $*$-bialgebra. If the generating functional $L$ of $(j_t)_{t\ge 0}$
vanishes on $\mathcal{I}$, then the marginal distributions also vanish on
$\mathcal{I}$ and $(j_t)_{t\ge 0}$ is equivalent to a L\'evy
process on commutative $*$-bialgebra $\mathcal{A}/\mathcal{I}$.

Let $L$ be tracial and Gaussian, with cocycle $\eta:\mathcal{A}\to D$. By Gaussianity, we have
\[
\eta(ab)= \e(a)\eta(b) + \eta(a)\e(b) = \eta(ba)
\]
for all $a,b\in\mathcal{A}$,
i.e.\ $\eta$ is also tracial. Let $a,b,c\in\mathcal{A}$, then we have
\begin{eqnarray*}
L(abc) &=& \e(a) L(bc) + \langle\eta(a^*),\eta(bc)\rangle +
L(a)\e(bc) \\
&=&\e(a) L(cb) + \langle\eta(a^*),\eta(cb)\rangle +
L(a)\e(cb) = L(acb)
\end{eqnarray*}
Combined with the traciality of $L$, it is now easy to see that the value of
the generating functional on a product of any number of elements of
$\mathcal{A}$ does not depend on the order of the factors, and therefore $L$
vanishes on the commutator ideal.
\end{proof}

In the following sections we shall study the L\'{e}vy processes on the free permutation group $S_n^+$. It will be useful to know when a Sch\"{u}rmann triple passes from a quantum group of a particular type to its quantum subgroup; this will be the use of the following lemma.

\begin{lemma}
\label{quotient}
Let $\mathcal{A}$ be an algebra generated by a collection of elements, $a_1,\ldots,a_n$, let $\e$ be a character on $\mathcal{A}$, and let $(\rho,\eta,L)$ be a Sch\"{u}rmann triple on $\mathcal{A}$. Let $\B$ be the quotient of $\mathcal{A}$ by the two-sided ideal generated by the selfadjoint relations $r_1(a_1,\ldots,a_n)=0$, $\ldots$, $r_k(a_1,\ldots,a_n)=0$.

If $\rho$, $\e$ and $\eta$ vanish on $r_1,\ldots,r_k$, then $\rho$ is a representation of $\B$ and $\eta$ is a $\rho$-$\e$-cocycle on $\B$. If, moreover, $L$ vanishes on $r_1,\ldots,r_k$, then $(\rho, \eta, L)$ is a Sch\"{u}rmann triple on $\B$.
\end{lemma}

\begin{proof}
Let us denote by $\mathcal{J}$ the two-sided ideal generated by the relations $r_1,\ldots,r_k$. Then $\B=\mathcal{A}/\mathcal{J}$. To show that all elements of the Sch\"{u}rmann triple on $\B$ are well-defined, it is enough to check that $\rho$, $\eta$ and $L$ vanish on $\mathcal{J}$. For that, take $a\in \mathcal{A}$ and $r\in \mathcal{J}$. By the assumption, $\rho(r)=0$, $\e(r)=0$, $\eta(r)=0$, $L(r)=0$, and thus immediately, $\rho(ar)=\rho(a)\rho(r)=0=\rho(ra)$, $\e(ar)=\e(a)\e(a)=0=\e(ra)$. Moreover, from the cocycle property, we see that $\eta(ar)=\rho(a)\eta(r)+\eta(a)\e(r)=0$ and similarily, $\eta(ra)=0$. Finally, from the relation $L(ab)=\langle \eta(a^*),\eta(b)\rangle$ for $a,b\in \ker \e$, we get
$$L(ar)=L([a-\e(a)1]r)+\e(a)L(r)= \langle \eta([a-\e(a)1]^*),\eta(r)\rangle=0$$
and in the same way we show that $L(ra)=0$.
\end{proof}

\section{The Markov semigroup of a L\'evy process}
\label{sec-markov}

In this short section we discuss Markov semigroups on the algebras related to compact quantum groups and their basic symmetry properties.

Functionals on a Hopf algebra act on it by convolution, i.e.\ for $f\in \mathcal{A}'$, we define $T_f:\mathcal{A}\to\mathcal{A}$ by $T_f(a) = ({\rm id}_\mathcal{A}\otimes f)\circ\Delta$ for $a\in\mathcal{A}$, which we shall also write as $T_f(a) = f\star a$. The map $f\mapsto T_f$ is homomorphism, i.e., we have
\[
T_f\circ T_g = T_{f\star g}
\]
for $f,g\in\mathcal{A}'$, as one checks by the calculation
\[
T_f\circ T_g (a) = \sum a_{(1)} f(a_{(2)}) g(a_{(3)}) = T_{f\star g}(a)
\]
using Sweedler's notation.

Further if $\mathbb{G}$ is a compact quantum group and $\omega$ is a state on ${\rm Pol}(\mathbb{G})$, then $T_\omega=({\rm id}\otimes \omega)\circ\Delta:{\rm Pol}(\mathbb{G})\to {\rm Pol}(\mathbb{G})$ preserves the unit, i.e., $T_\omega(1)=1$, and preserves positivity, since
\[
T_\omega(a^*a) = \sum a^*_{(1)i} a_{(1)j} \omega(a^*_{(2)i} a_{(2)j})
\]
for $a\in{\rm Pol}(\mathbb{G})$ with $\Delta(a) = \sum a_{(1)i} \otimes a_{(2)i}$. In fact, $T_\omega$ is even completely positive, i.e., its extension $T^{(n)}={\id}_{M_n}\otimes T_\omega:M_n\big({\rm Pol}(\mathbb{G})\big)\to M_n\big({\rm Pol}(\mathbb{G})\big)$ also preserves positivity for all $n\in\mathbb{N}$. It turns out that in fact these convolution operators extend also to completely positive maps between both types of completed $C^*$-algebras mentioned earlier, $C^r(\mathbb{G})$ and $C^u(\mathbb{G})$ (see for example \cite{CFK14}).

We can therefore associate to a convolution semigroup of states $(\omega_t)_{t\ge 0}$ a composition semigroup of completely positive unital maps $(T_t)_{t\ge0}$ by setting $T_t=T_{\omega_t}$. $(T_t)_{t\ge 0}$ is called the \emph{Markov semigroup} of $(\omega_t)_{t\ge 0}$ or of the associated L\'evy process $(j_t)_{t\ge 0}$. The terminology `Markov semigroup' arrives from noncommutative probability, where it usually denotes a semigroup of unital (or contractive) completely positive maps; its source lies of course in the classical theory of stochastic processes.

We call a Markov semigroup $(T_t)_{t\ge0}$ on a compact quantum group of Kac type \emph{symmetric}, if for all $t\ge 0$ $T_t$ is hermitian w.r.t.\ the inner product induced by the Haar state $h$, i.e.,
\[
h\big(T_t(a)^* b\big) = h(a^*T_t(b)\big)
\]
for $a,b\in {\rm Pol}(\mathbb{G})$, $t\ge 0$. For compact quantum groups whose Haar state is not a trace one may also consider a second, weaker, symmetry property called \emph{KMS-symmetry}, cf. \cite{CFK14}, but we shall not need it here since the permutation quantum groups we will study are of Kac type.

We have the following equivalent characterizations of symmetry of a Markov semigroup.

\begin{theorem}\cite[Theorem 4.11]{CFK14}\label{thm-sym}
Let $(j_t)_{t\ge 0}$ be a L\'evy process on a Kac-type compact quantum group $\mathbb{G}$ with Markov semigroup $(T_t)_{t\ge 0}:{\rm Pol}(\mathbb{G})\to{\rm Pol}(\mathbb{G})$ and generating functional $L:{\rm Pol}(\mathbb{G})\to \mathbb{C}$. Then the following conditions are equivalent.
\begin{description}
\item[(a)]
$(T_t)_{t\ge 0}$ is symmetric;
\item[(b)]
$T_L:{\rm Pol}(\mathbb{G})\to{\rm Pol}(\mathbb{G})$ is symmetric, i.e.
\[
h\big(T_L(a)^* b\big) = h(a^*T_L(b)\big)
\]
for $a,b\in {\rm Pol}(\mathbb{G})$;
\item[(c)]
$L$ is invariant under the antipode, i.e., $L\circ S=L$.
\end{description}
\end{theorem}

The antipode on a Kac-type compact quantum group is involutive. It furthermore preserves positivity and the counit. Therefore we can produce a symmetric functional starting from an arbitrary generating functional $L$ by setting
\[
L_{\rm sym} = \frac{1}{2}(L+L\circ S).
\]

\medskip
The \emph{adjoint action} of a Hopf algebra is the map $\ad :\Alg\to \Alg\otimes \Alg$ defined by the formula
$$\ad(a) = a_{(1)}S(a_{(3)}) \otimes a_{(2)}$$
for $a\in \Alg$, see, e.g., \cite[Section 1.3.4]{klimyk+schmudgen97}. We say that a linear functional $\phi\in \Alg'$ is \emph{\ad-invariant} if
\begin{equation}
(\id \otimes \phi)\circ \ad = \phi 1_\Alg.
\end{equation}
If $\Alg=\Pol(\mathbb{G})$ for a compact quantum group $\mathbb{G}$, this happens if and only if $L$ commutes (with respect to the convolution product) with all elements of  the space of linear functionals on $\Alg$ of the form $h_a: b \mapsto h(ab)$ for $a\in \Alg$.

It is easy to see that the counit $\e$ and the Haar state $h$ are always \ad-invariant, and if the quantum group in question is cocommutative then all the  functionals in $\Pol(\mathbb{G})'$ are \ad-invariant.

The ad-invariant functionals are a natural generalization of central measures on classical Lie groups. Recall that a measure $\mu$ on a topological group $G$ is called \emph{central}, if it commutes with all other measures (w.r.t.\ to the convolution). This is the case if
$$ \int_G f(gxg^{-1})d\mu(x) = \int_G f(x) d\mu(x)$$
for all $g\in G$ and $f\in C(G)$. Such measures play an important role in harmonic analysis and the study of L\'evy processes.

For arbitrary quantum groups we can produce $\ad$-invariant functionals using the Haar state. More precisely, let us denote by $\ad_h\in \mathcal{L}(\Pol(\mathbb{G}))$ the linear map given by
$$\ad_h=(h\otimes{\rm id})\circ \ad.$$
Then $\phi_{\ad} :=\phi\circ \ad_h$ is $\ad$-invariant for all $\phi\in\Pol(\mathbb{G})'$ and, in fact, $\phi\in\Pol(\mathbb{G})'$ is $\ad$-invariant if and only if $\phi= \phi\circ\ad_h$ (see \cite[Propositions 6.6 and 6.7]{CFK14}).

Given a set of irreducible unitary (pairwise inequivalent) representations $(u_{jk}^{(s)})_{s\in \mathcal \Irr(\QG)}$, an \ad-invariant functional $\phi$ must be of the form $ \phi(u_{jk}^{(s)})=c_s\delta_{jk}$ for $s\in \Irr(\QG)$. Hence $\phi$ is uniquely determined by its value on the characters $\chi_s := \sum_{j=1}^{n_s} u^{(s)}_{jj}$, $s\in \Irr(\QG)$, which suggest the relation to the \emph{algebra of central functions} on $\mathbb{G}$
\begin{equation} \label{eq_central_functions}
{\Pol}_0 (\mathbb G):= {\rm Lin}\left\{ \chi_s = \sum_{j=1}^{n_s} u^{(s)}_{jj}:
  s\in\Irr(\QG)\right\}.
\end{equation}

\medskip
In general, the mapping ${\rm ad}_h^*:\phi \mapsto \phi_{\rm ad}$ preserves neither hermitianity nor positivity, so it will not preserve the generating functionals of L\'{e}vy processes. However, this problem disappears if the quantum group $\mathbb G$ is of Kac type. In fact, we can then reduce the problem of classifying ad-invariant
generating functionals to the classification of
generating functionals on the subalgebra of central functions.
\begin{theorem}[see \cite{CFK14}, Corollary 6.11] \label{thm-ad-inv}
If $\mathbb{G}$ is of Kac type, then
\begin{enumerate}
\item $\ad_h$ is a conditional expectation onto ${\Pol}_0 (\mathbb G)$;
\item the mapping $$\ad_h^*: L \mapsto L\circ \ad_h$$
defines a bijection between generating functionals on ${\Pol}_0 (\mathbb G)$ and
{\rm ad}-invariant generating functionals on ${\Pol} (\mathbb G)$.
\end{enumerate}
\end{theorem}

Finally we note that if $(j_t)_{t\geq 0}$ is a L\'evy process on $\mathbb{G}$, and $(\omega_t)_{t\geq0}$ its associated convolution semigroup of states, then the Markov semigroup $(T_{\omega_t})_{t\geq0}$ will be called the \emph{Markov semigroup of the L\'evy process} $(j_t)_{t\geq 0}$.


\section{Quantum Permutation Groups -- Definition and Basic Properties}
\label{sec-q-perm}

An important family of quantum groups arises through the \emph{liberation} procedure, which can be thought of as a procedure based on replacing algebras of commuting coordinates on algebraic groups by the respective `free' coordinates. More precisely, for a fixed subgroup $G$ of the matrix unitary group $U_n$,  the algebra $C(G)$ can be described as a universal $C^*$-algebra generated by some set of relations $R$ satisfied by the coordinate functions $u_{jk}: U_n \ni [a_{jk}]_{j,k=1}^n \to a_{jk} \in \C$. These functions obviously commute. But using the same relations, one can define the universal algebra generated by abstract elements $u_{jk}$ subject to $R$, without assuming the commutativity of $u_{jk}$'s. This gives a noncommutative algebra, interpreted as the algebra of functions on the \emph{free version of the classical group} $G$ and denoted $C(G^+)$.
This idea, which should be attributed to Brown \cite{brown}, see\cite{banica+bichon+collins07}, has been successfully used to define the free version of the unitary group as well as the orthogonal group, \cite{wang95}.

The free analogue of the permutation group is due to S.\ Wang, \cite{wang98}. In this section we recall the definition and the main information about $S_n^+$. We refer the reader to surveys  \cite{banica+bichon+collins07}, \cite{Hadamard} and \cite{Metabief} for more information on this group.

\begin{definition} \label{def_sn+}
Let $C(S_n^+)$ denote the universal $C^*$-algebra generated by the family $\{p_{ij};i,j=1,\ldots,n\}$ of orthogonal projections satisfying the following condition:
\[ \sum_{j=1}^n p_{ij} = \sum_{j=1}^n p_{ji} =1, \;\;\; i=1,\ldots,n.\]
Then the formula
\[ \Com(p_{ij})= \sum_{k=1}^n p_{ik} \otimes p_{kj},\;\;\;i,j=1,\ldots,n,\]
extends to a coassociative unital $*$-homomorphism from $C(S_n^+)$ to $C(S_n^+)\otimes C(S_n^+)$ and the pair $(C(S_n^+), \Delta)$ is the algebra of continuous functions on a compact quantum group, called the \emph{free  permutation group} of $n$ elements, and denoted $A_s(n)$ or $C(S_n^+)$.
\end{definition}

The algebra $\pol(S_n^+)$ is generated - as a $*$-algebra - by $p_{jk}$'s and it is endowed with the standard Hopf $*$-algebra structure: $\e(p_{jk})=\delta_{jk}$, $S(p_{jk})=p_{kj}$.

For $n=1,2,3$, we have $C(S_n^+)=C(S_n)$, the free permutation group algebra is commutative. For $n\geq 4$, $C(S_n^+)$ is noncommutative and infinite dimensional, hence different from $C(S_n)$ (see the two-blocks examples, Section \ref{ssec_two_blocks}). In fact, it is known that $S_4^+$ may be viewed as a deformation of the classical $SO(3)$ group.

It is immediate to see that $S^2=\id$ on $\pol(S_n^+)$, hence the free permutation group is of Kac type.

We call a matrix $P=[P_{jk}]_{j,k=1}^n\in M_n(B(\hil))$ \emph{magic unitary} on $\hil$ if it is a unitary matrix whose entries are orthogonal projections. This is the same as to say that every $P_{jk}$ is a projection and all rows and columns form partitions of the unity. Representation of the algebra  $C(S_n^+)$ are in one-to-one correspondence with magic unitaries on Hilbert spaces, with the relation $\rho:C(S_n^+) \to B(\hil)$, $\rho(p_{jk})=P_{jk}$.

It has been shown in \cite{{Ban99}} that the fusion rules for $S_n^+$ are the same as for $SO(3)$. In particular, $\mathrm{Irr}_{S_n^+}=\mathbb{N}$.

\begin{theorem}[Fusion rules for $S_n^+$, \cite{{Ban99}}]
There exists a family of pairwise inequivalent irreducible representations
$(v^{(s)})_{s\in \N}$ of $S_n^+$ such that for all $s \in \N$
$$ v^{(0)} = 1_{C(S_n^+)}, \quad  [p_{jk}]_{j,k} = 1 \oplus v^{(1)}, \quad \overline{v^{(s)}}\equiv v^{(s)}$$
and for all $r, s \in \N$
$$ v^{(r)}\otimes v^{(s)} = \bigoplus_{l=0}^{2\min\{r,s\}} v^{(r+s-l)}.$$
Furthermore, every irreducible representation of $S_n^+$ is equivalent to
$v^{(s)}$ for some $s \in \N$.
The dimensions $d_s$, of the corepresentations $v^{(s)}$ are given recursively by
$$ d_0 = 1, \quad d_1 = n - 1 \quad \mbox{and}\quad  d_1d_s = d_{s+1} + d_s + d_{s-1}, \quad s\in \N.$$
The irreducible \emph{characters} $\chi_s := \sum_{j=1}^{d_n} (v^{(s)})_{jj}$
satisfy the recursive relations
\begin{equation}\label{eq-rec-rel}
\chi_1\chi_s= \chi_{s+1} + \chi_s + \chi_{s-1},\quad s\in \N.
\end{equation}
\end{theorem}

The above result can be used to show that the free permutation group is co-amenable if and only if $n \leq 4$.

\begin{definition}
Let $\QG$, $\QH$ be compact quantum groups. We say that $\QH$ is a quantum subgroup of $\QG$ if there exists a surjective unital $*$-homomorphism $\pi:C(\QG) \to C(\QH)$ intertwining the respective coproducts:
\[ (\pi \ot \pi)  \circ \Com_{\QG} = \Com_{\QH} \circ \pi.\]
Further we say that $\QH$ is a \emph{quantum permutation group} (or a quantum group of permutations) if it is a quantum subgroup of $S_n^+$ for some $n \in \N$.
\end{definition}

A word of warning is in place: it is neither true that every quantum permutation group is a finite quantum group, nor that every finite quantum group is a quantum permutation group (the latter result is much more difficult, see \cite{BBN}).

\subsection{Quantum symmetry groups}
\label{sec-q-symm}
In this short section we present the concept of \emph{quantum symmetry groups}, of which the free permutation groups are special examples.

Classical groups first appeared in mathematics as collections of symmetries of a given object (a finite set, a figure on the plane, a manifold, a space of solutions of an equation) and only later were axiomatized as sets equipped with particular operations. On the other hand, the description of compact quantum groups presented in the earlier sections is clearly algebraic. Hence it is natural to ask whether we can speak about compact quantum groups as `quantum symmetry groups' of some object. It turns out that the answer to this question is positive and in fact the relevant concept attracted a lot of activity within the last few years. In this section we will mention some aspects of it and refer to the lecture notes \cite{Metabief} for a more thorough treatment.

Naively speaking the symmetry group of a given object $X$ is the family of `all possible transformations' of $X$ preserving some internal structure of $X$. In the modern mathematical language, we can formalize this by saying that the symmetry group of $X$, $\tu{Sym}(X)$, is the {\em universal} object in the category of all groups acting on $X$: we require that for any group $G$ acting on $X$ there exists a unique homomorphism from $G$ to $\tu{Sym}(X)$ intertwining the respective actions. Recall that the action of a group $G$ on a set $X$ can be described as a map
$T:G \times X \to X$ satisfying certain natural conditions, for example $T(g, T(h,x))= T(gh,x)$ for all $g,h\in G$, $x \in X$. Of course, if we work with topological spaces, the map $T$ is assumed to be continuous.

In the quantum case we `invert the arrows' (as usual). This leads to the following definition of an action of a compact quantum group on a unital $C^*$-algebra, introduced in the papers \cite{Pod} and \cite{wang99}.

\begin{definition}
Let $\QG$ be a compact quantum group and let $\blg$ be a unital $C^*$-algebra. A map
\[ \alpha:\blg \to C(\QG) \ot \blg\]
is called a (left, continuous) \emph{action} of $\QG$ on $\blg$ if $\alpha$ is a unital $*$-homomorphism,
\[ (\Com \ot \id_{\blg}) \circ \alpha = (\id_{C(\QG)} \ot \alpha) \circ
\alpha\] and additionally $\alpha(\blg) (C(\QG) \ot 1_{\blg})$ is dense in $C(\QG)\ot \blg$ (\emph{Podle\'s/continuity condition}).
\end{definition}

Unless stated otherwise when we say that $\alpha:\blg \to C(\QG)\ot \blg$ is an action we mean that all the above are satisfied.
Note that each compact quantum group acts on its algebra of continuous functions via the coproduct.

If $G$ is a compact group, $X$ is a compact space and $G\curvearrowright X$ is a continuous action, then in the
quantum picture we define
\[\alpha:C(X) \to C(G) \otimes C(X)\approx C(G \times X)\]
via
\[ \alpha(f) (g,x) = f(gx), \;\;\; g \in G, x \in X, f \in C(X).\]
Conversely given an action of $\alpha:C(X) \to C(X) \ot C(G)$ we define for each $g \in G$ first
\[ \alpha_g = (\tu{ev}_g \ot \id_{C(X)}) \circ \alpha\]
and then note that the `action relation' implies that  $\alpha_g \circ \alpha_h = \alpha_{g h}$ for $g,h \in G$.
The condition `$\alpha(\blg) (C(G) \ot 1_{\blg})$ is dense in $C(G)\ot \blg$' is a kind of a nondegeneracy  or continuity property for
the action. In the case above, it guarantees that each $\alpha_g$ is a surjection, excluding for example $\alpha = \rho(\cdot)
1_{C(G)} \ot 1_{\blg}$ if $\mathsf{B}\not=\mathbb{C}$, where $\rho$ is a character on $\blg$.

Now, we want to define quantum symmetry groups as universal objects in the relevant categories. Consider then the category $\mathfrak{C} (\blg):=\{(\QG, \alpha)\}$ whose objects are compact quantum groups acting on a given unital $C^*$-algebra $\blg$.
A morphism in the category  $\mathfrak{C} (\blg)$:
\[ \gamma:(\QG_1, \alpha_1) \to (\QG_2, \alpha_2)\]
is a unital $*$-homomorphism $\gamma:C(\QG_2) \to C(\QG_1)$ such that
\[ (\gamma \ot \gamma) \circ \Com_{C(\QG_2)} = \Com_{C(\QG_1)} \circ \gamma, \;\; \alpha_1 = (\gamma \ot \id_{\blg})\circ \alpha_2.\]
(Note that once again we invert the arrows).

\begin{definition}
We say that the category  $\mathfrak{C} (\blg)$ admits a final object, if there is $(\QG_u, \alpha_u)$ in
$\mathfrak{C} (\blg)$ such that for all $(\QG, \alpha)$ in $\mathfrak{C} (\blg)$ there exists a unique morphism $\gamma:(\QG,
\alpha) \to (\QG_u, \alpha_u)$. If such a final object exists, it is unique, and we call it the \emph{quantum symmetry group} of $\blg$.
\end{definition}

In fact the final object in this category exists in very few cases. An important case for which the existence of the quantum symmetry group can be established is related to free permutation groups, see Definition \ref{def_sn+}. Before we formulate the relevant result we note that to deal with the non-existence one often considers the category of compact quantum group actions on $\blg$ preserving some more structure. A typical example comes from the requirement that the action preserves a given state on $\blg$.
\begin{definition}
Let $\QG$ be a compact quantum group, $\blg$ a unital $C^*$-algebra and  $\omega\in \blg^*$ be a state. We say that  the action $\alpha$ of $\QG$ on $\blg$ preserves $\omega$ if for all $b\in \blg$
\[ (\id_{\blg} \ot \omega) (\alpha(b)) = \omega(b)1.\]
The category of all compact quantum group actions acting on $\blg$ and preserving $\omega$ (with the morphisms defined as above) will be denoted $\mathfrak{C} (\blg, \omega)$.
\end{definition}
It can be checked that if both $C^*$-algebras are commutative, i.e.\ $\alg=C(G)$ and $\blg =C(X)$, the condition above corresponds to the fact that the action of $G$ on $X$ preserves the probability measure on $X$ corresponding to the state $\omega$.

The following theorem is (a slightly improved) version of the result of S.Wang from \cite{wang98}.

\begin{theorem} \label{perm}
Let $n\in \mathbb{N}$. Quantum group $S_n^+$ is the  final object in the category $\mathfrak{C} (\mathbb{C}^n)$.
\end{theorem}

In fact Wang proves in \cite{wang98} only that $S_n^+$ is the  final object in the category $\mathfrak{C} (\mathbb{C}^n, \omega)$, wheer $\omega$ is the `counting measure' on $\bc^n$. The generalisation above follows from Lemma 3.4 in \cite{BSS} (see \cite{Metabief} for details).

\medskip

The concept of a free permutation group (i.e.\ the quantum symmetry group of $n$ points) can be extended to the quantum symmetry group of a finite graph, as it was done in \cite{Juliengraph} and in \cite{BanJFA}.

Given a compact quantum group $\QG$ one can always construct its `classical version'. The procedure is described in the following lemma, which is very easy to prove.

\begin{lemma}
Let $\QG$ be a compact quantum group. Let $\mathfrak{I}$ be the commutator ideal of $C(\QG)$ (the closed ideal generated by all elements of the form $ab-ba$ for $a,b \in C(\QG)$). The quotient $C^*$-algebra $C(\QG)/\mathfrak{I}$ possesses a natural structure of the (commutative!) algebra of continuous functions on a compact quantum  group. In other words, there exists a compact group $G$ such that  $C(G) \approx C(\QG)/\mathfrak{I}$ (the isomorphism preserves the respective coproducts). We call $G$ the \emph{classical version} of $\QG$ and denote it by $\QG_{\textup{clas}}$.
\end{lemma}

The procedure of passing to the classical version often leads to a loss of  information -- for example $(SU_q(2))_{\textup{clas}} = \mathbb{T}$. For quantum symmetry groups, however, the situation is very satisfactory. The classical version of a quantum symmetry group of a given finite graph is the classical symmetry group of the same graph. In particular, we have $(S_n^+)_{\textup{clas}} = S_n$.

These observations lead to a natural question: given a compact group $G$, what is the `largest' (in some suitable sense) quantum group $\QG$ such that $G \approx\QG_{\textup{clas}}$? Some answers can be found in the paper \cite{TeoRoland}, where so-called liberated quantum groups (see the previous subsection) are considered.

The quantum symmetry groups discussed above all correspond to quantum groups acting on `finite quantum spaces', so all are quantum permutation groups. In recent years Goswami and his collaborators developed a theory of \emph{quantum isometry groups}, which is a quantum counterpart of the theory of isometry groups of compact Riemannian manifolds. In the quantum setting, in accordance with general principles of noncommutative geometry, the role of the `differential manifold structure' is played by a so-called `spectral triple'. Goswami's original construction from \cite{gos} can be in particular applied to the group $C^*$-algebra of a discrete, finitely generated group $\Gamma$ (see \cite{bhs}). Recall that the group algebra $C^*_r(\Gamma)$ can be viewed as the algebra of continuous functions on the quantum group $\hat{\Gamma}$, the `dual' of $\Gamma$. Hence the quantum isometry group of $C^*(\Gamma)$ can be interpreted as the quantum isometry group of $\hat{\Gamma}$.

An interesting phenomenon was observed in \cite{bsk}: it turns out that the classical version of the quantum isometry group of the `dual' of $\mathbb{F}^n$, the free group on $n$-generators is the isometry group of the $n$-dimensional torus $\mathbb{T}^n$. This can be understood in the following way: the `classical' (i.e.\ commutative) version of the group $\mathbb{F}^n$ is the free abelian group $\mathbb{Z}^n$, and of course $\widehat{\mathbb{Z}^n}=\mathbb{T}^n$. This means that this time both the notion of the isometric group action and the space on which the group acts become quantized!

\section{Classification problem for L\'{e}vy processes on $S_n^+$}
\label{sec-class}
We explained in previous sections that L\'{e}vy processes on compact quantum
groups can be studied through several of the corresponding objects: the
semigroup of states, the Markov semigroup, the generating functional or the Sch\"{u}rmann triple. The latter offers the most convenient approach to classify all L\'{e}vy processes on a fixed quantum group, when the representation theory of the underlying $C^*$-algebra is known (see the result for $SU_q(2)$ in \cite{schurmann+skeide98}), as then the problem can be reduced to certain cohomological computations, as will be seen in what follows. If the $C^*$-algebra of the quantum group in question is not type I, as is the case for $S_n^+$ for $n\geq 5$, the description of all its representations becomes very difficult.  In this section we shall describe the classification of L\'{e}vy processes on $S_n^+$ for a given, fixed representation of ${\rm Pol}(S_n^+) $. The main result is the necessary and sufficient condition for a tuple of vectors in the representation space to define a cocycle and a generating functional on $S_n^+$.

\medskip
For that we first state some general observation about Sch\"{u}rmann triples on $S_n^+$.
Let us recall (see Section \ref{sec-q-perm}) that the generators of $S_n^+$ satisfy the relations of magic unitary:
\begin{eqnarray}
 & p_{ij}=p_{ij}^*=p_{ij}^2, & i,j=1, \ldots, n, \label{sn-projections}\\
 & p_{ij}p_{ik}=0= p_{ji}p_{ki}, & j\neq k,\, i,j,k=1, \ldots, n, \label{sn-orthogonal}\\
 & \sum_{j=1}^n p_{ij} = \sum_{j=1}^n p_{ji} =1,& i=1,\ldots,n. \label{sn-columns}
\end{eqnarray}

Let $(j_t)_{t\geq 0}$ be a L\'{e}vy process on the free permutation group $S_n^+$ and let $(\rho, \eta, L)$ be the related Sch\"{u}rmann triple. We consider $\rho$ as a fixed representation of $\Pol(S_n^+)$ (equivalently, $C^u(S_n^+)$), denote by $P_{ij}=\rho(p_{ij})$, the image of $p_{ij}$ under the representation $\rho$, and write $\xi_{ij}:=\eta(p_{ij})$ for $i,j=1,\ldots,n$. As mentioned in Section \ref{sec-q-perm}, $P=[P_{ij}]_{i,j=1}^n$ is a magic unitary.

We start with two simple but useful observations.

\begin{lemma} \label{lem-diagonal}
Once the representation $\rho$ is fixed, the cocycle $\eta$ is uniquely determined by its values on $p_{ii}$, $i=1,\ldots,n$.
\end{lemma}

\begin{proof}
By the cocycle property, the value of $\eta$ on any polynomial in $p_{ij}$ ($i,j=1,\ldots,n$) can be expressed as a sum of elements of the form \linebreak $P_{i_1,j_1} \ldots P_{i_k,j_k} \xi_{i_{k+1},j_{k+1}}$, so is determined by the vectors $\xi_{ij}$ ($i,j=1,\ldots,n$). Moreover, the first relation in \eqref{sn-orthogonal} implies that if $j\neq k$ then
$$0= \eta(p_{ij}p_{ik})= P_{ij} \xi_{ik}+\delta_{ik} \xi_{ij}$$
If we choose $i=k$ then we see that $\xi_{ij}=-P_{ij} \xi_{ii}$.
\end{proof}

\begin{lemma} \label{lem-L_values}
A generating functional $L$ on $S_n^+$ in a Sch\"urmann triple $(\rho, \eta, L)$ is uniquely determined by the representation $\rho$ and the cocycle
$\eta$.

On the generators of $S_n^+$ the generating functional is given by
$L(p_{ij})= (-1)^{\delta_{ij}}\|\xi_{ij}\|^2$ for $i,j=1,\ldots,n$.
\end{lemma}

\begin{proof}
The relation $L(ab)=\langle \eta(a^*),\eta(b) \rangle +\e(a)L(b)+L(a)\e(b)$ induces the values on $1$ and on all products of at least two generators. Hence the only freedom we would have is to assign the values of $L$ on the individual generators.  We can then however apply the formula above to $a=b=p_{ij}$, which gives
$$L(p_{ij})=L(p_{ij}^2)
=\|\xi_{ij}\|^2+\delta_{ij} \cdot 2L(p_{ij}).$$
We conclude that
$$
L(p_{ij})=  \left \{ \begin{array}{ll}
\|\xi_{ij}\|^2, & i\neq j, \\ -\|\xi_{ii}\|^2, & i=j.
\end{array} \right .
$$
\end{proof}

We now show a general result for the existence and classification of
Sch\"{u}rmann triples on ${\rm Pol}(S_n^+)$ for a given representation $\rho$.

\begin{proposition} \label{prop_relations}
Let $\rho$ be a representation of $S_n^+$ on a Hilbert space $\hil$ and let the $n$-tuple of vectors $\xi_1, \xi_2,\ldots, \xi_n\in \hil$ satisfy the relations
\begin{equation} \label{conditions_on_xi}
\rho(p_{ii})\xi_i=0 \quad \mbox{and} \quad \rho(p_{ij})\xi_i=\rho(p_{ij})\xi_j \; \mbox{ for any } \; i,j=1,\ldots,n.
\end{equation}
Then we have the following properties.
\begin{enumerate}
\item
The mapping $\eta: \{p_{ij}:i,j=1,\ldots, n\}  \to \hil$ defined by
$$\eta(p_{ii}):=\xi_i, \;\;\; \eta(1):=0$$
extends uniquely to a cocycle on $S_n^+$ (which we denote by the same symbol).
\item
There exists a unique functional $L:\pol (S_n^+) \to \C$ such that $(\rho, \eta,L)$ is a Sch\"{u}rmann triple.
\item
If $\xi_1=\cdots=\xi_n$, then the unique cocycle with
$\eta(p_{ii}):=\xi_i$, $i=1,\ldots,n$ is a 1-coboundary.
\end{enumerate}
\end{proposition}

\begin{proof}
We write again $P_{ij}:=\rho(p_{ij})$ for $i,j=1,\ldots,n$.
(1) Following Lemma \ref{lem-diagonal} we define $\eta(1)=0$,
\begin{equation*}
\xi_{ij}= \eta(p_{ij})= \left \{ \begin{array}{ll}
\xi_i &\; \mbox{if}\;i=j,\\
-P_{ij} \xi_{i} &\; \mbox{if}\; i\neq j, \end{array} \right.
\end{equation*}
and extend $\eta$ to a linear operator satisfying the cocycle property on
the algebra $\mathcal{A}$ generated by $p_{11},p_{12},\ldots,p_{nn}$.
By Lemma \ref{quotient}, in order to prove that $\eta$ is a cocycle on $S_n^+$,
we need to check that it vanishes on all the relations \eqref{sn-projections}-\eqref{sn-columns}.

We havec(for $i,j,k=1,\ldots,n$, $j \neq k$)
\begin{eqnarray*}
 \eta( p_{ij}-p_{ij}^2)&=& \xi_{ij} - P_{ij}\xi_{ij} - \xi_{ij}\delta_{ij} =
\left \{ \begin{array}{ll}
- P_{ii}\xi_{i}=0 &\; \mbox{if}\;i=j\\
(P_{ij}-I)P_{ij} \xi_{i}=0 &\; \mbox{if}\; i\neq j \end{array} \right.,   \\
 \eta( p_{ij}p_{ik})&=& P_{ij}\xi_{ik} + \xi_{ij}\delta_{ik} \\
&=& \left \{ \begin{array}{ll}
P_{ij}\xi_{i}+ \xi_{ij}=0 &\; \mbox{if}\;i=k,j\neq k\\
-P_{ij} P_{ik}\xi_{i}=-\rho(p_{ij}p_{ik})\xi_{i}=0 &\; \mbox{if}\; i\neq k,j\neq k \end{array} \right.,   \\
 \eta(p_{ji}p_{ki}) &=& P_{ji}\xi_{ki} + \xi_{ji}\delta_{ki} \\
&=& \left \{ \begin{array}{ll}
P_{ji}\xi_{i}+ \xi_{ji} \stackrel{(\star)}{=} P_{ji}\xi_j + \xi_{ji}=0 &\; \mbox{if}\;i=k,j\neq k\\
-P_{ji} P_{ki}\xi_{k}=-\rho(p_{ji}p_{ki})\xi_{k}=0 &\; \mbox{if}\; i\neq k,j\neq k \end{array} \right.,   \\
 \eta(\sum_{l=1}^n p_{il} - 1) &=& \sum_{l=1}^n \xi_{il} = \xi_i - \sum_{l\neq i}  P_{il} \xi_i =\rho(1-\sum_{l\neq i} p_{il})\xi_i = P_{ii}\xi_i \stackrel{(\star)}{=} 0, \\
\eta(\sum_{l=1}^n p_{li} - 1) &=& \sum_{l=1}^n \xi_{li}  = \xi_i - \sum_{l\neq i}  P_{li} \xi_l \stackrel{(\star)}{=} (I-\sum_{l\neq i} P_{li})\xi_i = P_{ii}\xi_i \stackrel{(\star)}{=} 0.
\end{eqnarray*}
Above, we marked by $(\star)$ places where we use the assumptions $P_{ii}\xi_i=0$ and $P_{ij}\xi_i=P_{ij}\xi_j$.

(2) To show that the generating functional exists, let us define $L(1)=0$ and
$L(p_{ij})= (-1)^{\delta_{ij}}\|\xi_{ij}\|^2$
for $i,j=1,\ldots,n$. More explicitely,
\begin{equation}
L(p_{ii})= -\|\xi_{i}\|^2, \quad L(p_{ij})= \|\xi_{ij}\|^2 = \langle\xi_{i}, P_{ij}\xi_{i} \rangle \quad \mbox{for} \; i\neq j.
\end{equation}
Then $L$ extends uniquely by the coboundary condition and hence is conditionally positive and hermitian. It remains to check, again, that $L$ vanishes on all the relations \eqref{sn-projections}-\eqref{sn-columns}. This is the case since we have
\begin{eqnarray*}
L( p_{ij}-p_{ij}^2)&=& L( p_{ij}) - \langle\xi_{ij},\xi_{ij}\rangle
-2\delta_{ij} L( p_{ij}) = 0,
\end{eqnarray*}
and, for $j\not=k$,
\begin{eqnarray*}
L( p_{ij}p_{ik})&=& \langle\xi_{ij},\xi_{ik}\rangle + \delta_{ij} L( p_{ik})+\delta_{ik} L( p_{ij})
\\ &=&
\left \{ \begin{array}{ll}
-\langle\xi_{i},P_{ik}\xi_{i}\rangle + \|\xi_{ik}\|^2=0 &\; \mbox{if}\;i=j, i\neq k\\
\langle P_{ij}\xi_{ii},P_{ik}\xi_{ii}\rangle=0 &\; \mbox{if}\; i\neq j, i\neq k \end{array} \right.,   \\
\end{eqnarray*}
On the other hand, using the assumption on $\xi_i$'s, we get
\begin{align*}
L (p_{ji}p_{ki}) &= \langle\xi_{ji},\xi_{ki}\rangle + \delta_{ji} L( p_{ki})+\delta_{ki} L( p_{ji})
\\&= \langle P_{ji}\xi_{jj},P_{ki}\xi_{kk}\rangle + \delta_{ji} L( p_{ki})+\delta_{ki} L( p_{ji})\\
&=
\langle P_{ji}\xi_{ii},P_{ki}\xi_{ii}\rangle + \delta_{ji} L( p_{ki})+\delta_{ki} L( p_{ji}),
\end{align*}
which vanishes too. Finally, we check that
\begin{align*}
L(\sum_{j=1}^n p_{ij} - 1) &= \sum_{j\neq i} \langle\xi_{i}, P_{ij}\xi_{i} \rangle  -\langle\xi_{i}, \xi_{i} \rangle
= -\langle\xi_{i}, (I -\sum_{j\neq i} P_{ij}) \xi_{i} \rangle  \\&= -\langle\xi_{i}, P_{ii} \xi_{i} \rangle =0,
\end{align*}
\begin{align*}
L(\sum_{j=1}^n p_{ji} - 1) &= \sum_{j\neq i} \langle P_{ji}\xi_{j}, P_{ji}\xi_{j} \rangle  -\langle\xi_{i}, \xi_{i} \rangle
\\&=\sum_{j\neq i} \langle P_{ji}\xi_{i}, P_{ji} \xi_{i} \rangle  -\langle\xi_{i}, \xi_{i} \rangle =0.
\end{align*}
We showed that $L$ vanishes on \eqref{sn-projections}-\eqref{sn-columns}, hence $L$ is a generating functional.

(3)
Set $v=-\xi_1$. Since $P_{ii}v=-P_{ii}\xi_{i}=0$, we see that
$$\eta(p_{ii})=\xi_i = \big( P_{ii}-I \big) v=\big( \rho(p_{ii})-\e(p_{ii})I \big) v, \quad i=1,\ldots, n.$$
Then for any off-diagonal element $p_{ij}$ (with $i\neq j$) we have
$$\eta(p_{ij}) = -P_{ij}\eta(p_{ii}) = -P_{ij} \big( P_{ii}-I \big) v = P_{ij} v = \big( \rho(p_{ij}) -\e(p_{ij})I \big) v,$$
because $ p_{ij} p_{ii}=0$ and $\e(p_{ij})=0$.
This proves that $\eta$ is the coboundary of $v$.
\end{proof}

\begin{remark}\label{allasabove}
Note that once a representation $\rho$ is fixed, all the cocycles $\eta$ and generating functionals $L$ such that $(\rho,\eta,L)$ is a Sch\"urmann triple are of the form described in the above proposition.
\end{remark}

\begin{remark}\label{rem-diag-equal}
The converse of Property 3 is not true, i.e.\ there exist coboundaries for which
$\xi_i=\eta(p_{ii})$ is not independent of $i$. On the other hand, if for a given representation $\rho$ we have
$P_{11}=\cdots=P_{nn}$, then all coboundaries are of the form appearing in Property 3 above. This observation is trivial, but will be useful in the
next section.
\end{remark}

Let us recall that two L\'{e}vy processes are equivalent if they have the same joint moments. On the other hand, we shall say that two tuples \linebreak $(\rho; \xi_1, \xi_2,\ldots, \xi_n)$ and $(\rho'; \xi_1', \xi_2',\ldots, \xi_n')$, with $\rho$ and $\rho'$ being representations of $\pol (S_n^+)$ on $\hil$ and $\hil'$, respectively, and $\xi_i\in \hil$, $\xi_i'\in \hil'$ for $i=1,\ldots,n$, are \emph{equivalent} if there exists a unitary $V:\hil\to \hil'$ for which $V\rho(a)=\rho'(a)V$ and $\xi_i'=V\xi_i$ for $i=1,\ldots,n$.

Below, we summarize the results of this section in a classification theorem.
\begin{corollary} \label{cor_class}
There is a one-to-one correspondence between L\'{e}vy processes on $S_n^+$ (up to equivalence) and
the equivalence classes of tuples \linebreak $(\rho; \xi_1, \xi_2,\ldots, \xi_n)$ (up to equivalence), consisting of
a representation $\rho$ of $S_n^+$ on some Hilbert space $\hil$ and vectors $\xi_1, \xi_2,\ldots, \xi_n\in \hil$
satisfying the relations $\rho(p_{ii})\xi_i=0$ and $\rho(p_{ij})\xi_i=\rho(p_{ij})\xi_j$ for any $i,j=1,\ldots,n$.
\end{corollary}

The first step to answer the question about the L\'evy-Khinchin decomposition into Gaussian and non-Gaussian part (cf.\ the discussion below Definition \ref{def-gaussian}) usually starts with taking a closer look on Gaussian processes on the quantum group in question. In case of $S_n^+$, it is easy to see that no Gaussian process exists and hence the decomposition result holds trivially.

\medskip

Note that Lemma \ref{lem-L_values} states that there is no drift part in L\'evy processes on $S_n^+$.

\begin{proposition}\label{prop-no-Gauss}
There are no Gaussian processes on $S_n^+$ (and on other quantum permutation groups).
\end{proposition}

\begin{proof}
Let $\eta: {\rm Pol}(S_n^+) \to \hil$ be a Gaussian cocycle. Then $\eta(p_{jj})=\eta(p_{jj}^2) = 2\eta(p_{jj})$, so $\eta(p_{jj})=0$ for any $j=1,\ldots,n$. According to Lemma \ref{lem-diagonal}, this implies that $\eta=0$, so $L$ is a drift.
But, by Lemma \ref{lem-L_values}, a L\'{e}vy process on $S_n^+$ has no drift part, so the generating functional must vanish too.
\end{proof}

\section{Cohomology groups for some representations on $S_n^+$}
\label{sec-cohomology}

Two main examples of classical L\'{e}vy processes are Brownian motion
(Gaussian process) and Poisson process. Their quantum analogues (in the
bialgebra framework) have been introduced in \cite{schurmann93}. We saw that
quantum permutation groups admit no Gaussian L\'evy processes, cf.\ Proposition
\ref{prop-no-Gauss}, so we can focus here on compound Poisson processes and
some of their limits.

\medskip
For a fixed representation $\rho$ of the algebra of continuous functions on
the compact quantum group $\QG$ on a Hilbert space $\hil$, we denote by $Z_1(\rho)$ the set of all $\rho$-$\varepsilon$-cocycles on $\hil$ (see Definition \ref{def-triple}) and by $B_1(\rho)$ the set of all cocycles in $Z_1(\rho)$ which are 1-coboundaries  (cf.\ below Definition \ref{def-poisson}). Then the \emph{cohomology group} $\HA_1(\rho)$ is defined by
$$\HA_1(\rho):=Z_1(\rho)/B_1(\rho).$$

\medskip
It would be interesting to understand the cohomology groups related to all representations of ${\rm Pol}(S_n^+)$. However, since the quantum permutation group tend to be quite ``wild'' (not coamenable, hence with the associated $C^*$-algebras not of type I), a complete characterization of the irreducible representations of $C^u(S_n^+)$
is out of reach. That is why we are going to describe some interesting (classes of) examples, which show that
quantum permutation group has a much richer structure than its classical counterpart.

\subsection{Permutation matrices}

An $n \times n$ matrix $A$ is called a \emph{permutation matrix} if it has exactly one entry 1 in each row and each column and 0s elsewhere. Such matrices are in one-to-one correspondence with permutations of $n$ elements: given a permutation $\sigma$, the related matrix $A_\sigma$ has all entries 0 except that in row i, the entry $\sigma(i)$ equals 1. Conversely, if $A$ is a permutation matrix, then the corresponding permutation is $\sigma(i)=j$ when $A_{ij}=1$.

One checks directly that each permutation matrix $A_{\sigma}$ satisfies the
relations of magic square and thus it defines a one-dimensional representation of $S_n^+$ on a Hilbert space $\mathbb{C}$ by the formula $\rho_\sigma(p_{ij})=(A_{\sigma})_{ij}$. Now we want to study the cocycles related to a matrix $A_\sigma$ and describe their cohomology group.

Using Proposition \ref{prop_relations}, it is straightforward to determine all
cocycles of the representation $\rho_\sigma$ associated to a permutation $\sigma$.

\begin{proposition}\label{prop-coh-perm}
The cohomology group of a representation $\rho_\sigma$ associated to a
permutation $\sigma\not={\rm id}$ of $S_n^+$ is
\[
\HA_1(\rho_\sigma)\cong \mathbb{C}^{{\rm cyc}(\sigma)-{\rm fix}(\sigma)-1},
\]
where ${\rm cyc}(\pi)$ and ${\rm fix}(\pi)$ denote the number of cycles and the
number of fixed points of $\sigma$, respectively.

For $\sigma={\rm id}$ we have $\rho_{\rm id}=\varepsilon$ and
$B_1(\varepsilon)\cong Z_1(\varepsilon)\cong \HA_1(\varepsilon)\cong\{0\}$.
\end{proposition}

\begin{proof}
If $\sigma={\rm id}$ is the identity permutation, then
$\rho_\sigma=\varepsilon$ is equal to the counit and we have
$B_1(\varepsilon)\cong\{0\}$. By Proposition \ref{prop-no-Gauss}, there are no
Gaussian cocycles, so we also have $Z_1(\varepsilon)\cong
\HA_1(\varepsilon)\cong\{0\}$.

Assume now $\sigma\not={\rm id}$. It is then easy to see that  $B_1(\rho_\sigma)\cong\mathbb{C}$.

As shown in Lemma \ref{lem-diagonal}, a $\rho$-$\e$-cocycle $\eta$ on $S_n^+$ is uniquely determined by the diagonal values $\xi_1,\xi_2,\ldots, \xi_n\in \mathbb{C}$ of the matrix $\eta(p)$ (that is, $\xi_j=\eta(p_{jj})$). Moreover, by Proposition \ref{prop_relations}, these vectors should satisfy the relations
\begin{equation} \tag{\ref{conditions_on_xi}}
 \rho(p_{ii})\xi_i=0, \quad \mbox{and}\quad \rho(p_{ij})\xi_i=\rho(p_{ij})\xi_j \; (i\neq j).
\end{equation}

If $i$ is a fixed point of the permutation, then $\rho_\sigma(p_{ii})=1$ and the
first relation implies $\xi_i=0$.

If $(i_1,\ldots,i_k)$ is a cycle of $\sigma$  of length $k\ge 2$, then we have $\rho_\sigma(p_{i_\ell i_{\ell+1}})=1$ for $\ell=1,\ldots,k-1$, and the second relation implies $\xi_{i_1}=\cdots = \xi_{i_k}$, i.e., $\xi_i$ is constant along indices that belong to a cycle.

Therefore we have
\[
B_1(\rho_\sigma)\cong \mathbb{C}, \quad Z_1(\rho_\sigma)\cong \mathbb{C}^{{\rm
    cyc}(\sigma)-{\rm fix}(\sigma)},\quad \HA_1(\rho_\sigma)\cong \mathbb{C}^{{\rm
    cyc}(\sigma)-{\rm fix}(\sigma)-1}.
\]
\end{proof}

\subsection{Fourier-Hadamard matrices}
An $n \times n$ matrix $H$ is called a \emph{complex Hadamard matrix} if all its entries are of modulus 1 and if its rows are pairwise orthogonal. It follows from the definition that the columns of $H$ are pairwise orthogonal as well and that $HH^* = n \; I$.

For an Hadamard matrix $H\in M_n(\C)$ with rows $h_1, \ldots, h_n$, the rank one projections
$$P_{jk} = {\rm Proj}(h_j/h_k)$$
form a magic unitary matrix $P\in M_n(A)$ with $A = M_n(\C)$, \cite{Hadamard}, where above we use the notation $h_j/h_k=[\frac{h_{jm}}{h_{km}}]_{m=1}^n$. This implies that we can define the homomorphism
$\rho_H: S_n^+ \ni u_{jk} \mapsto P_{jk} \in M_n(A)$.
Our aim now is to study the cohomology group of the representations related to Hadamard matrices.

\medskip
We first focus on a special case of complex Hadamard matrices, the \emph{Fourier matrices}. These are matrices given by
$$    (F_n)_{lm}:= {\rm e}^{2\pi i\frac{(l - 1)(m - 1)}{n}} {\quad \rm for \quad} l,m=1,2,\dots,n. $$

For a fixed $n\in \mathbb{N}$, the representation associated to the Fourier matrix $\rho_{F_n}(p_{jk})= P_{jk}$ can be described as follows. We observe first that for any $j=1,\ldots, n$, the diagonal element $P_{jj}= {\rm Proj}(h_j/h_j)$ is a projection on the vector $[1,\ldots,1]=h_1$. Similarily, the element $P_{j1}$ is the projection on the vector $h_j$, which we denote for convenience by $P_{j}$. Thus the family $P_1, \ldots, P_{n}$ is a partition of identity and induces a Hilbert space decomposition
$$\C^n=P_1H\oplus P_2H\oplus \ldots \oplus P_n H=\C h_1 \oplus \C h_2 \oplus  \ldots \oplus \C h_n.$$

Next, we note that
$$\frac{h_{jm}}{h_{km}} = {\rm e}^{2\pi i\frac{(j - 1)(m - 1)}{n}}{\rm e}^{ -2\pi i\frac{(k - 1)(m - 1)}{n}} = {\rm e}^{2\pi i\frac{(j - k)(m - 1)}{n}}=h_{(j-k+1)_{{\rm mod}\, n}}, $$
which means that
\begin{equation} \label{obs_pjk_mod_n}
P_{jk}=P_{(j-k+1)_{{\rm mod}\, n}}.
\end{equation}
These observations allow us to write down explicitely the matrix $\rho_{F_n}(p)=[\rho_{F_n}(p_{jk})]_{j,k=1}^n$:
$$\rho_{F_n}(p)=\left [ \begin{array}{cccccc}
P_1 & P_{n} & P_{n-1} & \cdots & P_3 & P_2 \\
P_2 & P_1     & P_n     & \cdots & P_4 & P_3 \\
P_3 & P_2     & P_1 & \cdots & P_5 & P_4 \\
\ldots & \ldots & \ldots & \ldots & \ldots & \ldots \\
P_{n} & P_{n-1} & P_{n-2} & \cdots & P_2 & P_1\\
\end{array} \right].$$

This means that $\rho_{F_n}$ is a direct sum over representations coming from
permutations. Indeed, denote by $s_k$, $0\le k\le n-1$ the cyclic shift on
$\{1,\ldots,n\}$ that acts as
\[
s_k(i) = \left\{\begin{array}{lcl}
i+k & \mbox{ if} & i+k\le n, \\
i+k-n & \mbox{ if} & i+k> n.
\end{array}\right.
\]
Then we have
\begin{equation}\label{eq-FH-decomp}
\rho_{F_n}(p) = \sum_{k=0}^{n-1} A_{s_k}P_{k+1}
\end{equation}
i.e., the magic unitary associated to $F_n$ is a direct sum of permutation matrices.

\begin{example} \label{exampleF4}
The cohomology group of $\rho_{F_n}$ can be non-trivial. To see that, let us have a look on the following example. If $n=4$, the representation $\rho_{F_4}$ given by
$$\rho_{F_4}(p)= \left [ \begin{array}{cccc}
P_1 & P_4 & P_3 & P_2 \\
P_2 & P_1 & P_4 & P_3 \\
P_3 & P_2 & P_1 & P_4 \\
P_4 & P_3 & P_2 & P_1\\
\end{array} \right],$$
where $P_k$ is the projection on the vector $h_{k} = [e^{\frac{i\pi (k-1)(j-1)}{2}}]_{j=1}^4$. Now, let us define the vectors $\xi_1=\xi_3=0$ and $\xi_2=\xi_4=h_3\not=0$. We can check directly that such vectors satisfy the relations from Proposition \ref{prop_relations} and thus define a cocycle $\eta: S_4^+ \to \C^4$ by the values $\eta(p_{ii})=\xi_i$, $i=1,\ldots, 4$. On the other hand, if $\eta$ was a coboundary, this would mean that there exists a vector $v\in \C^4$ such that $0=(P_1-I)v$ and $h_3=(P_1-I)v$, which leads to a contradiction. In fact, we can chose $\xi_2=\xi_4$ to be an arbitrary (non-zero) element from the space generated by $h_3$, which shows that $H_1(\rho(F_4))=\C h_3$.
\end{example}

We shall now generalize this observation.

\begin{proposition}
Let $F=F_n$ be the Fourier matrix of size $n$.
\begin{enumerate}
\item[(i)] For any $n\in \mathbb{N}$
$$H_1(\rho_F)= \bigoplus_{k=1}^{n-1} \mathbb{C}^{{\rm gcd}(n,k)-1},$$
where $gcd(n,k)$ denotes the greatest common divisor of $n$ and $k$. 
\item[(ii)] If $n$ is a prime number, then $H_1(\rho_F)=\{0\}$ and so any cocycle corresponding to the representation $\rho_F$ 
is a coboundary.
\end{enumerate}
\end{proposition}

\begin{proof}
From the decomposition \eqref{eq-FH-decomp} we get
\[
H_1(\rho_F)\cong \bigoplus_{k=0}^{n-1}H_1(\rho_{s_k}).
\]
For $k=0$ we have $s_0={\rm id}$, and
\[
H_1(\rho_F)=\{0\}
\]
by Proposition \ref{prop-coh-perm}.

Since a shift $s_k$, with $n>k\ge 1$, on $\{1,\ldots,n\}$ decomposes into ${\rm
  gcd}(n,k)$ disjoint cycles and has no fixed points, we get
\[
H_1(\rho_{s_k})\cong \mathbb{C}^{{\rm gcd}(n,k)-1},
\]
by Proposition \ref{prop-coh-perm}.

Part (ii) follows from the previous one when we observe that for prime $n$ all $gcd(n,k)=1$ and thus ${\rm dim}\, H_1(\rho_F)=0$.
\end{proof}

\medskip
In general, the description of all $n$-dimensional complex Hadamard matrices is not known for dimensions higher than $n=5$ (see \cite{hadamard_matrices_catalogue} for the current state of art). To deal with the classification problem, it suffices to treat only inequivalent matrices. Two complex Hadamard matrices $H_1$ and $H_2$ are called \emph{equivalent} if there exist diagonal unitary matrices $D_1$, $D_2$ and permutation matrices $P_1$, $P_2$ such that $H_1=D_1P_1H_2P_2D_2$. Then any complex Hadamard matrix is equivalent to a \emph{dephased} Hadamard matrix, i.e.\ an Hadamard matrix in  which all elements in the first row and first column are equal to unity.

We do not need to distinguish between equivalent Hadamard matrices.

\begin{proposition}\label{prop-equiv-H}
If two complex Hadamard matrices $H_1$ and $H_2$ are equivalent, then there exists a Hopf $*$-algebra automorphism $\theta:{\rm Pol}(S_n^+)\to{\rm Pol}(S_n^+)$ such that representations $\rho_{H_1}$ and $\rho_{H_2}\circ \theta$ are unitarily equivalent.
\end{proposition}

We split the proof into two lemmas. Let us start with the multiplication by permutation matrices.

\begin{lemma}
Let $H$ be a $n\times n$ complex Hadamard matrix and $P=(\delta_{j,\pi(k)})_{1\le j,k\le n}$ the permutation matrix associated to a permutation $\pi\in S_n$.
\begin{description}
\item[(a)]
Multiplying $H$ from the right by $P$ corresponds to a permutation of the canonical basis vectors in $\mathbb{C}^n$, i.e.
\[
\rho_{HP} = P^{-1}\rho_H P.
\]
\item[(b)]
Multiplying $H$ from the left by $P$ correspond to a renumbering of the generators of ${\rm Pol}(S_n^+)$. More precisely, there exists a unique Hopf $*$-algebra automorphism $\theta_\pi:{\rm Pol}(S_n^+)\to{\rm Pol}(S_n^+)$ such that
\[
\theta(p_{jk})=p_{\pi(j)\pi(k)}
\]
for $1\le j,k\le n$, and furthermore we have
\[
\rho_{PH} = \rho\circ\theta_\pi.
\]
\end{description}
\end{lemma}
\begin{proof}
\begin{description}
\item[(a)]
If we multiply $H$ from the right by a permutation $P$, we permute the columns of $H$, which corresponds to a permutation of the canonical basis vectors in the space where the projections \linebreak $P^H_{jk} = {\rm Proj}(h_j/h_k)$ act. If $H=(h_{jk})_{1\le j,k\le n}$, then 
$$HP=(h_{j\pi^{-1}(k)})_{1\le j,k\le n},$$ 
and therefore we get
\[
P^{HP}_{jk} = {\rm Proj}\left( \left(\frac{h_{j\pi^{-1}(m)}}{h_{k\pi^{-1}(m)}}\right)_{m=1}^n\right) = P^{-1} P^H_{jk} P.
\]
\item[(b)]
It is straightforward to check that the entries of the matrix \linebreak $(p_{\pi(j)\pi(k)})_{1\le j,k\le n}$ satisfy the same relations as those of $(p_{jk})_{1\le j,k\le n}$ and this guarantees the existence of $\theta_\pi$. Uniqueness is clear because the coefficients $(p_{jk})_{1\le j,k\le n}$ generate ${\rm Pol}(S_n^+)$.

We have $PH=(h_{\pi(j),k})_{1\le j,k\le n}$ and therefore
\[
\rho_{PH}(p_{jk}) = P^{PH}_{jk} = {\rm Proj}(h_{\pi(j)}/h_{\pi(k)})= P^H_{\pi(j),\pi(k)} = \rho_H(p_{\pi(j),\pi(k)})
\]
for $1\le j,k\le n$.
\end{description}
\end{proof}

In the following lemma we consider the multiplication by diagonal matrices.

\begin{lemma}
Let $H$ be a $n\times n$ complex Hadamard matrix and $D=(c_j\delta_{jk})_{1\le j,k\le n}$ a diagonal matrix with coefficients of modulus one, i.e., $|c_j|=1$ for $j=1,\ldots,n$.

Multiplying $H$ from the right or the left by $D$ does not change the associated representation $\rho_H$, i.e., $\rho_H=\rho_{HD}=\rho_{DH}$.
\end{lemma}
\begin{proof}
\begin{description}
\item[1. Case] Multiplication from the right.

If $H=(h_{jk})_{1\le j,k\le n}$, then $HD=(h_{jk}c_k)_{1\le j,k\le n}$ and therefore
\[
\rho_{HD}(p_{jk}) = P^{HD}_{jk} = {\rm Proj}\left( \left(\frac{h_{jm}c_m}{h_{km}c_m}\right)_{m=1}^n\right) = P^H_{jk}=\rho_H(p_{jk})
\]
for $1\le j,k\le n$.
\item[2. Case] Multiplication from the left.

If $H=(h_{jk})_{1\le j,k\le n}$, then $DH=(c_jh_{jk})_{1\le j,k\le n}$ and the vectors onto which we will project become
\[
\left(\frac{c_j h_{jm}}{c_k h_{km}}\right)_{m=1}^n, \qquad 1\le j,k\le n.
\]
Since these vectors differ from the vectors
\[
\left(\frac{h_{jm}}{h_{km}}\right)_{m=1}^n, \qquad 1\le j,k\le n.
\]
only by a phase, they generate the same one-dimensional subspaces and give the same projections, i.e., $\rho_{DH}(p_{jk}) = P^{DH}_{jk} = P^H_{jk} = \rho_H(p_{jk})$ for $1\le j,k\le n$.
\end{description}
\end{proof}

\begin{proof} (of Proposition \ref{prop-equiv-H})
The proposition now follows from the two lemmas above.
\end{proof}

The representation $\rho$ of $S_n^+$ related to the dephased Hadamard matrix has diagonal entries equal to the projection on $h_1=(1,\ldots, 1)$. Then the $\rho$-$\e$-cocycle is a coboundary iff the vectors $\eta(p_{jj})$ are all equal (see Remark \ref{rem-diag-equal}).

As an example of what happens for non-Fourier matrices, we can have a look at the case $n=4$. Every $4\times 4$ complex Hadamard matrix is equivalent to a matrix of the form
$$ F_{4}(\phi):=
\begin{bmatrix}
1 & 1 & 1 & 1 \\ 1 & ie^{i\phi} & -1 & -ie^{i\phi} \\
1 & -1 & 1 &-1 \\ 1 & -ie^{i\phi}& -1 & i e^{i\phi}
\end{bmatrix} {\quad \rm with \quad } \phi\in [0,\pi) .$$
For $\phi=0$ we recover the Fourier matrix $F_4$.

\begin{proposition}
Let $\rho_\phi$ denotes the representation of $S_n^+$ associated to the Hadamard matrix $F_{4}(\phi)$. The cohomology group of $\rho_\phi$, $H_1(\rho_\phi)$, is isomorphic to $ \C $ if $\phi\in [0,\frac{\pi}{2})\cup(\frac{\pi/2},\pi)$ and to $\C^3$ if $\phi=\frac{\pi}{2}$.
\end{proposition}

\begin{proof}
Let $k_1=(1, ie^{-i\phi}, -1, -ie^{-i\phi})$ and $k_2 = (1, -ie^{-i\phi}, -1, ie^{-i\phi})$. It is easy to see that $k_1= \alpha h_2 + \beta h_4$ and $k_2= \beta h_2 + \alpha h_4$ with $\alpha = \alpha(\phi)=\frac12(1+e^{-2i\phi})$ and $\beta= \beta(\phi)=\frac12(1-e^{-2i\phi})$. For $P_j$ denoting the projection onto $\C h_j$ ($j=1,2,3,4$) and $Q_j$ being the projection onto $\C k_j$ ($j=1,2$), we recover the magic unitary
$$\rho_\phi = \left [ \begin{array}{cccc}
P_1 & Q_2 & P_3 & Q_1 \\
P_2 & P_1 & P_4 & P_3 \\
P_3 & Q_1 & P_1 & Q_2 \\
P_4 & P_3 & P_2 & P_1
\end{array} \right],$$
where $P_iP_j=0$ for $i\neq j$.

Let now $\xi_1,\ldots,\xi_4\in \C^4$ define a cocycle on $S_n^+$ for the representation $\rho_{\phi}$.  The conditions \eqref{conditions_on_xi} imply that
$P_1\xi_j=0$ for any $j$,
\begin{eqnarray*}
 P_2 \xi_1=P_2\xi_2, & P_2 \xi_3=P_2\xi_4, & P_3\xi_1=P_3\xi_3, \\
  P_3\xi_2=P_3\xi_4, & P_4\xi_2=P_4\xi_3,  & P_4\xi_1=P_4\xi_4,
\end{eqnarray*}
and
$$ Q_1\xi_1=Q_1\xi_4, \quad Q_1\xi_2=Q_1\xi_3, \quad Q_2 \xi_1=Q_2\xi_2, \quad Q_2 \xi_3=Q_2\xi_4.$$
It is convenient to express the coefficients of $\xi_i$'s in the basis $\{h_1,h_2,h_3,h_4\}$. Write then $\xi_i =\sum_{k=1}^4 x_{ik}h_k$ for $i=1,2,3,4$. We get then
$$\begin{array}{c|cccc}
x_{ik}& h_1 & h_2 & h_3 & h_4 \\ \hline
\xi_1 & 0   & a   &  c  &  f \\
\xi_2 & 0   & a   &  d  &  g \\
\xi_3 & 0   & b   &  c  &  g \\
\xi_4 & 0   & b   &  d  &  f
\end{array}$$
Moreover, from the relations involving the projections $Q_1$ and $Q_2$, we deduce that
\begin{eqnarray*}
 0 &=& Q_1(\xi_1-\xi_4)= \langle \xi_1- \xi_4,k_1 \rangle
\\ &=& \alpha \langle \xi_1-\xi_4, h_2 \rangle + \beta \langle \xi_1-\xi_4, h_4  \rangle
= \alpha (a-b) + \beta (f-f) = \alpha (a-b), \\
 0 &=& Q_1(\xi_2-\xi_3) = \langle \xi_2-\xi_3, \alpha h_2+ \beta h_4  \rangle
= \alpha (a-b),\\
 0 &=& Q_2 (\xi_1-\xi_2) = \langle \xi_1-\xi_2,\beta h_2 + \alpha h_4  \rangle
= \alpha (f-g), \\
 0 &=& Q_2 (\xi_3-\xi_4) = \langle \xi_3-\xi_4,\beta h_2 + \alpha h_4  \rangle
 = \alpha (g-f).
\end{eqnarray*}

Thus we fix uniquely a cocycle on $S_n^+$ by choosing the constants $a,b,c,$ $d,f,g$ which satisfy the following two conditions:
(a) if $\beta(\phi) \neq 0$ then $a=b$, and
(b) if $\alpha(\phi) \neq 0$ then $a=b$ and $f=g$.
Such a cocyle is a coboundary iff $\xi_1=\xi_2=\xi_3 =\xi_4$.

Consider first the case $\phi\neq 0$ and $\phi\neq \frac{\pi}{2}$, then $\alpha\neq 0$ and $\beta \neq 0$, hence we have $a=b$ and $f=g$. This leaves 4 degrees of freedom to define an arbitrary cocyle related to $\rho_\phi$ (the choice of $a,c,d,f$). The coboundaries will form 3-dimensional subspace with the condition $c=d$ and $a$, $f$ arbitrary. Therefore $H_1(\rho_\phi) = \C $. Similarly, for $\phi=0$ (and $\alpha=1$, $\beta=0$) we choose $a,c,d,f$ for an arbitrary cocycle, while only $c\neq d$ guarantees that the cocycle is not a coboundary. Hence also here $H_1(\rho_0) = \C $.

On the other hand, if we choose $\phi =\frac{\pi}{2}$, then $\alpha=0$ and so the relations involving the projections $Q_1$ and $Q_2$ are satisfied for all constants $a,b,\ldots,g$. So the dimension of $Z_1(\rho_\phi)$ equals 6, while $\eta\in B_1(\rho_\phi)$ whenever $a=b$, $c=d$, $f=g$. So  $H_1(\rho_\frac{\pi}{2}) = \C^3 $.

This can also be seen from the fact that for $\phi=\frac{\pi}{2}$ we have $Q_1=P_4$ and $Q_2=P_2$, so the representation $\rho_\phi$ decomposes into a direct sum of four irreducible representations. The first is the trivial representation (= counit), and it has no non-trivial cocycles. The other three representations come from permutations which consist of two disjoint transpositions, so each of them has a one-dimensional first cohomology group.
\end{proof}

\subsection{Two-blocks example}
\label{ssec_two_blocks}

In this subsection we consider the two-block representation $\rho(P,Q)$ of $\mathbb{G}=S_4^+$. It is defined by
\begin{eqnarray} \label{pq-rep}
\rho(P, Q)(p)&=&\left [ \begin{array}{cccc}
P & 1-P & 0 & 0 \\
1-P & P & 0 & 0 \\
0 & 0 & Q & 1-Q \\
0 & 0 & 1-Q & Q
\end{array}\right ],
\end{eqnarray}
where $P$ and $Q$ are two projections on a Hilbert space $\hil$.

According to Lemma \ref{lem-diagonal}
and Proposition \ref{prop_relations},
the related cocycle in the Sch\"{u}rmann triple is given by
\begin{eqnarray} \label{pq-cocycle}
\eta(p) = [\eta(p_{ij})]_{i,j=1}^n &=&\left [ \begin{array}{cccc}
\xi & -\xi & 0 & 0 \\
-\xi & \xi & 0 & 0 \\
0 & 0 & \zeta & -\zeta \\
0 & 0 & -\zeta & \zeta
\end{array}\right ],
\end{eqnarray}
where $\xi$ and $\eta$ are arbitrary two vectors in the representation space $\hil$ satisfying $P\xi=0$ and $Q\zeta=0$.

\medskip
We first study the cohomology group for a special case of projections $P$ and $Q$ in a so-called  \emph{general position}. This means that the Hilbert space $\hil$ in which acts $P$ and $Q$ decomposes into $\hil_1\oplus \hil_2$ and that there exist two commuting positive contractions $c,s \in B(\hil)$ such that each of them has trivial kernel, $c^2+s^2=I$ and (with respect to the corresponding block-matrix decomposition)
\begin{equation}  \label{gen_pos_decomp}
P=\begin{bmatrix} 1 & 0 \\0 & 0 \end{bmatrix}, \;\;\;
Q = \begin{bmatrix} c^2 & cs \\cs & s^2 \end{bmatrix}
\end{equation}
(see \cite{takesaki}).

\begin{proposition} \label{prop_gen_pos}
Let $P$ and $Q$ be two projections in some finite-dimensional Hilbert space
$\hil$ which are in a general position. Suppose that $\xi, \zeta\in \hil$ are
such that $P{\xi}=0,$ $Q\zeta=0$. Then there exists a vector $\eta\in \hil$ such
that
\begin{equation}
P^{\perp}\eta= \xi, \;\;\; Q^{\perp}\eta= \zeta.
\end{equation}
\end{proposition}

\begin{proof}
If $P$ and $Q$ are in general position then we can write them in the form \eqref{gen_pos_decomp}, where the matrix notation is with respect to the Hilbert space decomposition
\begin{equation*} \label{h5_decomposition}
\hil=P \hil\oplus P^\perp \hil
\end{equation*}

Reformulate first the problem in terms of vectors in $\hil$. Suppose that
$$ \xi = \binom{\tilde{w}}{w} , \;\;\; \zeta = \binom{u}{v}\;\;\; $$
for some $\tilde{w},w,u,v \in \hil$. Then $P{\xi}=0,$ $Q\zeta=0$ rewrites as
\begin{equation} \tilde{w}=0, \;\; c^2 u + csv=0, \;\; cs u + s^2 v =0.\label{uv}\end{equation}
We want to show that there exist vectors $x,x'\in \hil$ such that
\begin{equation*}
P^{\perp} \binom{x}{x'} = \binom{0}{w} , \;\;\; Q^{\perp} \binom{x}{x'} = \binom{u}{v}.
\end{equation*}
It is easy to see that we must have $x'=w$. The remaining condition on $x$ can be written as follows:
\[ (1-c^2)x - cs w = u, \;\;\; - cs x + (1-s^2)w = v,\]
or equivalently, putting $y=sx$,
\begin{equation} sy = u+ csw, \;\; cy = c^2 w - v.\label{scy}\end{equation}
Acting on the last two equations respectively by $s$ and $c$, and adding the resulting vectors we see that we must have
\begin{equation} y = su+ cs^2 w + c^3w - cv = su + cw - cv. \label{candy}\end{equation}
We can check that due to \eqref{uv} the vector $y$ defined by \eqref{candy} satisfies \eqref{scy}. Indeed,
\[ s(su+cw - cv) = s^2 u +scw - csv = s^2u +c^2u +scw = u+csw, \]
\[ c(su +cw - cv) = csu +c^2w -c^2v = -s^2v +c^2w -c^2v = csw -v\]
From \eqref{candy} we deduce that $x$ must satisfy the equation
\begin{equation*} \label{eq_for_x}
s(x-u)= c(w - v).
\end{equation*}

Note that both $s\hil$ and $c\hil$ are dense in $\hil$, e.g.
$$ \overline{{\rm Im } \ s } = ({\rm Ker} \ s^*)^\perp =({\rm Ker} \ s)^\perp = \{0\}^\perp=\hil,$$
and so due to the assumption that $\hil$ is finite dimensional, we get ${\rm Im } \ s = \hil = {\rm Im } \ c$. This means that for $w-v$ there exists a vector $\tilde{x}\in \hil$ such that $s\tilde{x}=c(w-v)$. To recover $x$ we put $x=\tilde{x}+u$.
\end{proof}

\begin{corollary} \label{cor_gen_pos}
If $P$ and $Q$ are projections in general position on a finite dimensional Hilbert space, then $Z_1(\rho(P, Q))=B_1(\rho(P, Q))$ and $\HA_1(\rho(P, Q))=\{0\}$.
\end{corollary}

\begin{proof}
Let $\eta$ be a $\rho(P, Q)$-$\e$-cocycle determined by the vectors $\xi=\eta(p_{11})$ and $\zeta=\eta(p_{33})$, i.e.\ defined by \eqref{pq-cocycle}. By Proposition \ref{prop_gen_pos}, there exist a vector $v\in \hil$ such that
$$
\eta(p_{11})=\xi=(P-I)v, \;\;\; \eta(p_{33})= \zeta=(Q-I)v \;\;\; .
$$
We check directly that then $\eta(p_{ij}) = (P_{ij}-\delta_{ij}I) v$ for any $i,j=1,2,3,4$, which means that $\eta$ is a coboundary.
\end{proof}

Now we shall study the general case of cohomology group of $\rho(P, Q)$. For arbitrary $P$ and $Q$ let us define the following projections:
$$ R_1=P\wedge Q, \; R_2=P\wedge Q^\perp, \; R_3=P^\perp\wedge Q, \; R_4=P^\perp\wedge Q^\perp$$
and denote the respective images by $\hil_1$, $\hil_2$, $\hil_3$ and $\hil_4$. Moreover, we set $\hil_5 = \hil \ominus \bigoplus_{i=1}^4 \hil_i$.

\begin{proposition} \label{prop_PQ_cohomology}
Let $\rho(P,Q)$ be the representation \eqref{pq-rep} of $S_4^+$. If the subspace $\hil_5$ is finite-dimensional, then  we have
$$\HA_1(\rho(P,Q)) = \hil_4=(P^\perp \wedge Q^\perp)(\hil).$$
\end{proposition}

\begin{proof}
Let us decompose both vectors appearing in the form of $\eta$ as $\xi=\xi_1+\ldots+ \xi_5$ and $\zeta=\zeta_1+\ldots+ \zeta_5$, where $\xi_i, \zeta_i \in \hil_i$ for $i=1,\ldots,5$. We shall study each part of this decomposition separately and show first that the parts coresponding to $i=1,2,3,5$ do not influence the cohomology group.

For that note that since $P\xi=0$ and $Q\zeta=0$ we must have $\xi_1=\xi_2=0$ and $\zeta_1=\zeta_3=0$. So there is no problem with $\hil_1$. If $\zeta_2\neq0$, then putting $v_2=-\zeta_2\in (P\wedge Q^\perp)(\hil)$ we have $\zeta_2=(Q-I)v_2$ and $\xi_2=0=(P-I)v_2$. This means that cocycles on $\hil_2$ are always coboundaries. Similarily, we check that for $v_3=-\xi_3 \in (P^\perp\wedge Q)(\hil)$ we have $\xi_3=(P-I)v_3$ and $\zeta_3=0=(Q-I)v_3$.

As for the last part, we observe that the projections $P$ and $Q$ on $\hil_5$ are in the general position and, by Corollary \ref{cor_gen_pos}, we deduce that $\eta$ is a coboundary on $\hil_5$.

This way we showed that $\HA_1(\rho(P,Q)) \subset \hil_4 = (P^\perp \wedge Q^\perp)(\hil)$.
To end the proof we need to check that for any non-zero vector $v_4\in \hil_4$ there exists a cocycle which is not a coboundary. So let us take $v_4\in \hil_4$, $v_4\neq 0$ and set $\xi_4=v_4$ and $\zeta_4=0$. If we assume that $v_4=\xi_4=(P-I)\tilde{v}$ and $0=\zeta_4=(Q-I)\tilde{v}$ for some $\tilde{v}\in \hil$, then, by repeating the reasoning as above, we see that we can always take $\tilde{v}\in \hil_4$, from which we deduce that $\tilde{v}=(Q-I)\tilde{v}=0$. But this contradicts to the assumption that $v_4=(P-I)\tilde{v}\neq 0$.

\end{proof}

Both Propositions \ref{prop_gen_pos} and \ref{prop_PQ_cohomology} are  no longer true if we omit the assumption that $\hil$ (resp. $\hil_5$) is finite dimensional. In general, the group of all cocycles can be identified with a certain closed subspace, say $\mathsf{k}$, of $\hil \oplus \hil$ (in fact it is always just $\ker P \oplus \ker Q$), viewed as the additive group, and the 1-coboundaries are given by a subspace of $\mathsf{k}$, say $\mathsf{k}_0$. One can compute, see the proof of Propositions \ref{prop_gen_pos}, that precisely
\begin{eqnarray*}
\mathsf{k}&=&  \hil \oplus \{(u,v)\in \hil \oplus \hil: c^2u+csv=0, csu +s^2v=0\} \\
\mathsf{k}_0 &=& \{(w,u,v)\in \mathsf{k}: cw-cv \in {\rm Im}\ (s)\}.
\end{eqnarray*}
Now the quotient would be the orthogonal complement of $\mathsf{k}_0$ in $\mathsf{k}$, \emph{if} $\mathsf{k}_0$ to be closed (we provide an explicit example below). This need not be the case if we have ${\rm dim}\ \hil=\infty$ and the general position situation. This means that we cannot hope for such a nice explicit  description (we will only get a description of $\HA^1(P,Q)$ as an `abstract' quotient).

\begin{example}[Two projections for which the space $\mathsf{k}_0$ not closed]
Let $\mathsf{h} = \mathsf{H} \oplus \mathsf{H}$, where $\mathsf{H}$ is infinite dimensional with the orthonormal basis $\{e_n\}_{n\in \mathbb{N}}$, and let
\begin{equation*}
P=\begin{bmatrix} 1 & 0 \\0 & 0 \end{bmatrix}, \;\;\;
Q = \begin{bmatrix} c^2 & cs \\cs & s^2 \end{bmatrix},
\end{equation*}
where
$$ se_n = \sqrt{1-\frac1{n}} \ e_n, \quad ce_n = \sqrt{\frac1{n}} \ e_n.$$

Then
$$
u_N = -\sum_{n=0}^N n^{-\frac34} (n-1)^{-\frac12} e_n, \quad
v_N = \sum_{n=0}^N n^{-\frac34} e_n, \quad
w_N = \frac12\sum_{n=0}^N n^{-\frac34} e_n$$
are such that $(w_N,u_N,v_N)\in \mathsf{k}_0$ for all $N \in \N$, but the limit of the  sequence $(w_N,u_N,v_N)_{N \in \N}$ does not belong to $\mathsf{k}_0$.

\end{example}

\section{Symmetric and ad-invariant generating functionals on $S_n^+$}
\label{sec-sym-ad-inv}

A L\'{e}vy process is symmetric if and only if it's generating
functional is invariant under the antipode, i.e.\ $L=L\circ S$, cf.\ Theorem \ref{thm-sym}. In this section we shall first explore conditions for the symmetry of L\'{e}vy processes on $S_n^+$, related to the representations studied in previous sections. Then we shall describe all ad-invariant generating functionals.

It is a matter of direct calculation to check the following equivalent condition for the
generating functional on $S_n^+$ to be symmetric. By a 'word' of length $n$ we will understand a product of $n$ generators (that is an element of the form
$w=p_{i_1,j_1}p_{i_2,j_2}\ldots p_{i_n,j_n}$).

\begin{lemma} \label{lem_handy}
A generating functional $L$ on $S_n^+$ is symmetric if and only if the
following two conditions are satisfied:
\begin{enumerate}
 \item $L(p_{ij})=L(p_{ji})$ for any $i,j=1,\ldots,n$;
 \item for any $w = p_{i_1j_1}\ldots p_{i_rj_r}$ with $(i_{k_1},j_{k_1})\neq(i_{k_2},j_{k_2})$ for
$k_1\neq k_2$,  and any generator $p=p_{ij}$, we have
 \begin{equation} \label{eq_eta_gns}
 \langle \eta (p),\eta(w) \rangle = \langle \eta\big((Sw)^*\big) , \eta (Sp)\rangle .
 \end{equation}
\end{enumerate}
\end{lemma}

\begin{proof}

By linearity, it is enough to check when $L(w)=L\circ S(w)$ for all words $w$.
We shall moreover restrict to \emph {concatenated} words, i.e.\ to the words
to which the rules $p_{ij}=p_{ij}^2=p_{ij}^*$ have already been applied. This
means that in the concatenated word there are neither higher powers of generators, nor of
their adjoints. That is why we can restrict the reasoning to words of the form
$w=p_{i_1,j_1}p_{i_2,j_2}\ldots p_{i_n,j_n}$ with
$(i_{k_1},j_{k_1})\neq(i_{k_2},j_{k_2})$ for $k_1=k_2+1$.

In $S_n^+$, due to the equalities $p_{ij}=p_{ij}^*$ and $\e(Sw)=\e(w)$, for any word $w$ we have
 \begin{eqnarray*}
L(p_{ij} w) &=& \langle \eta(p_{ij}), \eta (w) \rangle + L(p_{ij})\e(w) +
  \e(p_{ij}) L( w)\\
L\circ S (p_{ij} w) &=& L\big( (Sw) (Sp_{ij} )\big)
\\ &=&
  \langle \eta \big( (Sw)^*\big), \eta (Sp_{ij}) \rangle + L(Sw)\e(p_{ji}) +
\e(w) L(Sp_{ij}).
\end{eqnarray*}

If $L(p_{ij})=L(p_{ji})$ holds, then $L=L\circ S$ for words of length $n=1$. This is because
$ L\circ S(p_{ij}) =L(p_{ji}^*) = L(p_{ji}) = L(p_{ij})$. For words of arbitrary length $n\geq 2$
we use \eqref{eq_eta_gns} to deduce by induction that
$$L(p_{ij} w) - L\circ S (p_{ij} w) =  \big[ L( w)- L(Sw)\big]\delta_{ij}=0.$$

The converse implication follows directly.
\end{proof}

\begin{remark}
 The lemma above can be seen as a version of the fact that a generating
fuctional $L$ is invariant under $S$ on $({\rm ker}\, \e)^2$ if and only if the
related cocyle $\eta$ is {\it real}, i.e.
$$ \langle \eta(a), \eta (b) \rangle =
  \langle \eta(S(b^*)), \eta (S(a)^*) \rangle,\quad a,,b\in Pol (\mathbb{G}),$$
see Theorem 4.6 in \cite{kyed11} (attributed to R. Vergnioux),
also Remark 6.20 in \cite{daws+fuma+skalski+white}.
\end{remark}

\subsection{Permutation matrix}
It is easy to see that there are no symmetric processes related to the
maximal cycle representation
\begin{equation}
\rho_\sigma (p) = A_\sigma I_\hil =\left[ \begin{array}{cccccc}
0 & I_\hil & 0 & \ldots & 0 & 0 \\
0 & 0 & I_\hil & \ldots & 0 & 0 \\
\vdots & \vdots & \vdots & \ddots & \vdots & \vdots \\
0 & 0 & 0 & \ldots & I_\hil & 0 \\
0 & 0 & 0 & \ldots & 0 & I_\hil \\
I_\hil & 0 & 0 & \ldots & 0 & 0 \\
\end{array}\right]
\end{equation}
Indeed, using results from Section \ref{sec-class} (Lemma \ref{lem-diagonal}, Lemma \ref{lem-L_values}
and Proposition \ref{prop_relations}), we see that the cocycle and the generating functional related to $A_\sigma$ are given by
\begin{eqnarray} \label{perm_symetry}
\eta(p)&=&
\left [ \begin{array}{ccccc}
\xi & -\xi & \ldots & 0 & 0 \\
0 & \xi & \ldots & 0 & 0 \\
 \vdots & \vdots & \ddots & \vdots & \vdots \\
0 & 0 & \ldots & \xi & -\xi  \\
-\xi & 0 &\ldots & 0 & \xi
\end{array}\right ], \\
L(p) &=&\left [ \begin{array}{ccccc}
-\|\xi\|^2 & \|\xi\|^2 & \ldots & 0 & 0 \\
0 & -\|\xi\|^2 & \ldots & 0 & 0 \\
 \vdots & \vdots & \ddots & \vdots & \vdots \\
0 & 0 & \ldots & -\|\xi\|^2 & \|\xi\|^2  \\
\|\xi\|^2 & 0 &\ldots & 0 & -\|\xi\|^2
\end{array}\right ]
\end{eqnarray}
for some $\xi\in \hil$.
Now it is evident that, if $L$ was symmetric, then $\|\xi\|^2 =
L(Su_{j,j+1})=L(u_{j+1,j})=0$.

\subsection{Fourier matrix}
Let us consider the representation of $S_n^+$ related to the Fourier matrix
$F_n$:
$$\rho(F_n)(p)=\left [ \begin{array}{cccccc}
P_1 & P_{n} & P_{n-1} & \cdots & P_3 & P_2 \\
P_2 & P_1     & P_n     & \cdots & P_4 & P_3 \\
P_3 & P_2     & P_1 & \cdots & P_5 & P_4 \\
\ldots & \ldots & \ldots & \ldots & \ldots & \ldots \\
P_{n} & P_{n-1} & P_{n-2} & \cdots & P_2 & P_1\\
\end{array} \right],$$
where $P_{j}$ is the projection on the vector $h_j=({\rm e}^{2\pi i(k - 1)(j - 1) / n})_{k=1,2,\dots,n}$. Recall that $P_1, \ldots, P_{n}$ is a decomposition of identity and we have here $P_{jk}=P_{(j-k+1)_{{\rm mod}\, n}}$, cf.\ \eqref{obs_pjk_mod_n}. Moreover, note that if $P_{ij}=P_m$ with $m=m(i,j)=_n i-j+1$, then
$$ \rho\circ S(p_{ij}) = P_{ji} = P_{(j-i+1)_{{\rm mod}\ n}} = P_{(2-m)_{{\rm mod}\ n}}.$$
For the sake of clearity, we shall omit the notation ``mod $n$'' in the subscript of $P$. Any time a projection appears with the index $n$ outside of $\{1,\ldots,n\}$ it should be understood as taken modulo $n$.

Let $L$ be a generating functional on $S_n^+$ with the related representation $\rho(F_n)$ and the cocycle $\eta$. The cocycle is uniquely determined by $\xi_j=\eta(p_{jj})$, $j=1,\ldots,n$, and we have $\eta(p_{ij})=-P_{ij}\xi_i=-P_{ij}\xi_j$.

\begin{proposition}
A generating functional $L$ on $S_n^+$ in a Sch\"urmann triple $(\rho(F_n), \eta,L)$ is symmetric if and only if the vectors associated to the cocycle $\eta$ via the formulas above satisfy the equalities
\begin{equation} \label{eq_gns_2}
 \langle \xi_i, P_{m}\xi_k \rangle =\langle P_{2-m}\xi_k , \xi_i \rangle
\end{equation}
for any $i,m,k=1,2,\ldots,n$.
\end{proposition}

\begin{proof}
Assume first that $L$ is symmetric.
Then, by Lemma \ref{lem_handy}, $\eta$ satisfies \eqref{eq_eta_gns}, which applied to $w=p_{kl}$ yields
\begin{eqnarray*}
 \langle \eta (p_{ij}),\eta(p_{kl}) \rangle &=& \langle \eta\big((Sp_{kl})^*\big) , \eta (Sp_{ij})\rangle
\\
 \langle P_{ij}\xi_i, P_{kl}\xi_k \rangle &=& \langle P_{lk}\xi_k , P_{ji}\xi_i \rangle
\\
 \langle \xi_i, P_{m(i,j)}P_{m(k,l)}\xi_k \rangle &=& \langle P_{m(j,i)}P_{m(l,k)}\xi_k , \xi_i \rangle.
\end{eqnarray*}
Note that $m(i,j)=m(k,l)$ iff $m(j,i)=m(l,k)$, and if these indices are different the product of two projections vanishes.
If $m:=m(i,j)=m(k,l)$, then we get $ \langle \xi_i, P_{m}\xi_k \rangle =\langle P_{2-m}\xi_k , \xi_i \rangle. $
But for any fixed $i,j,k$, we can choose $l$ such that $m(i,j)=m(k,l)$, so the condition \eqref{eq_gns_2} holds for any choice of $i,k,m=1,\ldots,n$.

Now let us now assume \eqref{eq_gns_2} holds. Then, for any $i,j$ and $m=i-j+1$, we have
\begin{eqnarray*}
\langle\xi_i, P_{m}\xi_i \rangle = \langle \xi_i, P_{2-m}\xi_i \rangle
&\Rightarrow & \langle \xi_i, P_{ij}\xi_i \rangle = \langle \xi_i, P_{ji}\xi_i \rangle \\
\quad \Rightarrow \quad  \|P_{ij}\xi_i\|^2 = \|P_{ji}\xi_i\|^2
& \Rightarrow & \|\eta(p_{ij})\|^2 = \|\eta(p_{ji})\|^2 .
\end{eqnarray*}
By Lemma \ref{lem-L_values}, the last equation is equivalent to $L(p_{ij})=L(p_{ji})$. Hence, the condition (1) from Lemma \ref{lem_handy} holds. To see that (2) is true as well, let us consider an arbitrary word $w$ of length $s$. Then
$$\rho(w) = P_{m_1} P_{m_2}\ldots P_{m_s} = P_{m}^r P_{m_{r+1}}\cdots P_{m_s} = P_{m} P_{m_{r+1}}\cdots P_{m_s},$$
where $m_p=i_p-j_p+1$ (for $p=1,2,\ldots,s$) and $m_1=m_2=\ldots=m_r \neq m_{r+1}$. This means that the first $r$ (and no more) generators are transformed by $\rho$ into the same projection $P_m$, and $P_mP_{m_{r+1}}=0$. Hence
\begin{eqnarray*}
\eta(w) &=& \eta (p_{i_1j_1}\ldots p_{i_rj_r}\cdot p_{i_{r+1}j_{r+1}}\ldots p_{i_sj_s})
\\ &=&  P_m\eta( p_{i_{r+1}j_{r+1}}\ldots p_{i_sj_s}) + \eta(p_{i_1j_1}\ldots p_{i_rj_r})\e( p_{i_{r+1}j_{r+1}}\ldots p_{i_sj_s})
\\ &=&  P_m\eta( p_{i_{r+1}j_{r+1}}) + \eta(p_{i_1j_1}\ldots p_{i_rj_r})\e( p_{i_{r+1}j_{r+1}}\ldots p_{i_sj_s}),
\end{eqnarray*}
and
\begin{eqnarray*}
\eta\big((Sw)^*\big) &=& \eta (p_{j_1i_1}\ldots p_{j_ri_r}\cdot p_{j_{r+1}i_{r+1}}\ldots p_{j_si_s})
\\ &=&  P_{2-m}\eta( p_{j_{r+1}i_{r+1}}) + \eta\big(S(p_{i_1j_1}\ldots p_{i_rj_r})^*\big)\e( p_{i_{r+1}j_{r+1}}\ldots p_{i_sj_s}),
\end{eqnarray*}

Therefore
\begin{eqnarray*}
\lefteqn{ \langle \eta (p_{i_0j_0}),\eta(w) \rangle}\\
& = &
 \langle \eta (p_{i_0j_0}), P_m\eta( p_{i_{r+1}j_{r+1}})
 +   \eta(p_{i_1j_1}\ldots p_{i_rj_r})\e( p_{i_{r+1}j_{r+1}}\ldots p_{i_sj_s}) \rangle
\\
&=&  \langle P_{m_0} \xi_{i_0}, P_m P_{m_{r+1}}\xi_{i_{r+1}}\rangle
 \\
& &  + \  \e( p_{i_{r+1}j_{r+1}}\ldots p_{i_sj_s}) \langle \eta (p_{i_0j_0}),\eta(p_{i_1j_1}\ldots p_{i_rj_r})\rangle ,
\end{eqnarray*}
while
\begin{eqnarray*}
\lefteqn{ \langle \eta\big((Sw)^*\big) , \eta (Sp_{ij})\rangle}\\
&=& \langle P_{2-m}\eta( p_{j_{r+1}i_{r+1}}) , \eta (p_{j_0i_0})\rangle
\\
&  & + \ \e( (p_{i_{r+1}j_{r+1}}\ldots p_{i_sj_s})^*)\langle \eta\big(S(p_{i_1j_1}\ldots p_{i_rj_r})^*\big) , \eta (Sp_{i_0j_0})\rangle
\\
& \stackrel{{\rm ind}}{=} & \langle P_{2-m}\eta( p_{j_{r+1}i_{r+1}}) , \eta (p_{j_0i_0})\rangle
\\
& &  + \ \e( p_{i_{r+1}j_{r+1}}\ldots p_{i_sj_s})\langle  \eta (p_{i_0j_0}), \eta(p_{i_1j_1}\ldots p_{i_rj_r})\rangle
\\
& = & \langle P_{m_0}P_{2-m}P_{m_{r+1}}\xi_{i_{r+1}} , \xi_{i_0}\rangle
\\
& &  + \  \e( p_{i_{r+1}j_{r+1}}\ldots p_{i_sj_s})\langle  \eta (p_{i_0j_0}), \eta(p_{i_1j_1}\ldots p_{i_rj_r})\rangle
\end{eqnarray*}

To show that \eqref{eq_eta_gns} holds it remain to show that
$$ \langle P_{m_0}P_{2-m}P_{m_{r+1}}\xi_{i_{r+1}} , \xi_{i_0}\rangle  = \langle \xi_{i_0}, P_{m_0} P_m P_{m_{r+1}}\xi_{i_{r+1}}\rangle. $$
But this follows from \eqref{eq_gns_2} provided $m_0=m=m_{r+1}$. Otherwise, both terms vanish. We conclude by Lemma \ref{lem_handy} that $L$ is symmetric.
\end{proof}

\subsection{Two-blocks example}

Let $L$ be a generating functional on $S_4^+$ related to the representation
\begin{eqnarray*}
\rho(P,Q)(p)&=&\left [ \begin{array}{cccc}
P & 1-P & 0 & 0 \\
1-P & P & 0 & 0 \\
0 & 0 & Q & 1-Q \\
0 & 0 & 1-Q & Q
\end{array}\right ],
\end{eqnarray*}
where $P$ and $Q$ are two projections on a Hilbert space $\hil$, see Eq. \eqref{pq-rep}.
Then, recall Eq. \eqref{pq-cocycle}, the related cocycle is given by
\begin{eqnarray*}
\eta(p) &=&\left [ \begin{array}{cccc}
\xi & -\xi & 0 & 0 \\
-\xi & \xi & 0 & 0 \\
0 & 0 & \zeta & -\zeta \\
0 & 0 & -\zeta & \zeta
\end{array}\right ],
\end{eqnarray*}
where $\xi$ and $\eta$ are arbitrary two vectors in the representation space $\hil$ satisfying $P\xi=0$ and $Q\zeta=0$.
Moreover, by Lemma \ref{lem-L_values},
\begin{eqnarray*} 
L(p)&=&\left [ \begin{array}{cccc}
-\|\xi\|^2 & \|\xi\|^2 & 0 & 0 \\
\|\xi\|^2 & -\|\xi\|^2 & 0 & 0 \\
0 & 0 & -\|\zeta\|^2 & \|\zeta\|^2 \\
0 & 0 & \|\zeta\|^2 & -\|\zeta\|^2
\end{array}\right ],
\end{eqnarray*}

The result below answers the question for which $\xi$ and $\zeta$ the
generating functional $L$ is symmetric.
\begin{theorem} \label{thm_charact}
Let $P,Q \in B(\hil)$ be projections and $L$ be a generating functional in a Sch\"urmann triple $(\rho(P,Q),\eta,L)$.  Then $L$ is symmetric if and only if
\begin{equation} \label{eq_first}
 \big\langle \zeta, (PQ)^k \xi \big\rangle = \big\langle (PQ)^k
\xi, \zeta
\big\rangle \quad \mbox{for any}\quad  k \geq 0,
\end{equation}
where vectors $\xi, \zeta \in \hil$ are associated to $\eta$ as in \eqref{pq-cocycle}.
\end{theorem}

Before proving this result, we need some preparations.
\begin{remark}
Let the representation $\rho$ be \emph{symmetric}, i.e.
$\rho(p_{ij})=\rho(p_{ji})$ for all $i,j=1,\ldots, n$.
Then it follows from the relations $p_{ik}p_{ij}=0$ and
 $p_{ij}^2=p_{ij}$ that $\eta(p_{ij})=\eta(p_{ji})$ and $L(p_{ij})=L(p_{ji})$.
If we use the notation $p_{ij} \sim p_{kl}$ for elements on which
$\rho$, $\eta$, $L$ and $\e$ take equal values, this means that $\rho$ being
symmetric implies $p_{ij} \sim p_{ji}$.
\end{remark}

We start with showing a version of Lemma \ref{lem_handy} for symmetric representations of $S_n^+$.
\begin{lemma} \label{lem_gns_handy}
Let $L$ be the generating functional on $S_n^+$ associated to a Schurmann triple
$(\rho, \eta, L)$ for a symmetric representation $\rho$. Then $L$ is symmetric if and only if
\begin{equation} \label{GNS_eta}
 \langle \eta(p_{ij}), \eta (w) \rangle = \langle \eta (w), \eta(p_{ij})
\rangle
\end{equation}
for $i,j=1,\ldots,n$, $n\in \mathbb{N}$ $w=p_{i_1,j_1}p_{i_2,j_2}
\ldots p_{i_n,j_n}$ and $(i_{k_1},j_{k_1})\neq(i_{k_2},j_{k_2})$ for
$k_1\neq k_2$.
\end{lemma}

\begin{proof}
We have already remarked that if $\rho$ is symmetric then $L(p_{ij})=L(p_{ji})$.
Moreover, it follows from the formula
\begin{eqnarray*}
\lefteqn{\eta(p_{i_1,j_1}\ldots p_{i_n,j_n}) }\\ &=&
\sum_{k=1}^n \rho(p_{i_1,j_1})\rho(p_{i_2,j_2})\ldots
\rho(p_{i_{k-1},j_{k-1}}) \eta(p_{i_k,j_k})\e(p_{i_{k+1},j_{k+1}}) \ldots
\e(p_{i_n,j_n}),
\end{eqnarray*}
that the value of $\eta$ on a fixed word depends only on the
values of $\eta$, $\rho$ and $\e$ on the individual  components $p_{i_k,j_k}$. Consequently,
it does not change when $p_{ij}$ is replaced by $p_{kl}$ such that $p_{kl}\sim
p_{ij}$. Thus $\eta(w) = \eta(p_{i_1,j_1}p_{i_2,j_2}\ldots p_{i_n,j_n})
 = \eta\big(p_{j_1,i_1}p_{j_2,i_2}\ldots p_{j_n,i_n}\big)=\eta\big( (Sw)^*\big)$,
and \eqref{GNS_eta} is equivalent to \eqref{eq_eta_gns}.

\end{proof}

\begin{proof}[Proof of Theorem \ref{thm_charact}]
Using Lemma \ref{lem_gns_handy}, we shall first extract the necessary
conditions for $L$ to be symmetric and next show that this condition is
enough to
ensure \eqref{GNS_eta} for every word $w$.

\medskip \noindent {\bf Step 1.} Assume that $L=L\circ S$ and consider the word
$w_1=(p_{11}p_{33})^kp_{11}$. Then
$$ \eta(w_1)= \sum_{l=0}^{k} (PQ)^l \xi + \sum_{l=0}^{k-1} (PQ)^l
P\zeta$$
and the condition \eqref{GNS_eta} translates into
$$ \langle \eta(p_{ij}), \sum_{l=0}^{k} (PQ)^l \xi + \sum_{l=0}^{k-1} (PQ)^l
P\zeta\rangle = \langle \sum_{l=0}^{k} (PQ)^l \xi + \sum_{l=0}^{k-1} (PQ)^l
P\zeta, \eta(p_{{ij}})\rangle.$$

For $i=j=1$ this is always satisfied (since $P\xi=0$), whereas for
$i=j=3$ this gives
$$ \sum_{l=0}^{k} \left( \langle \zeta, (PQ)^l \xi \rangle -
\langle (PQ)^l \xi, \zeta\rangle \right)
= \sum_{l=0}^{k-1} \left( \langle \zeta, (PQ)^l P\zeta \rangle -
\langle (PQ)^l P\zeta, \zeta\rangle \right) =0.$$
Comparing the given equation for $k$ and $k-1$ we conclude that
\begin{equation*}
\langle \zeta, (PQ)^k \xi \rangle = \langle (PQ)^k \xi, \zeta\rangle, \quad
k\geq 0,
\end{equation*}
which is just  \eqref{eq_first}. Note that, in particular, $\langle \xi, \zeta \rangle = \langle \zeta, \xi
\rangle$.

\medskip \noindent {\bf Step 2.} We are going to show that the condition
\eqref{eq_first} is also sufficient for a generator $L$ to
be symmetric. Assume that \eqref{eq_first} is satisfied. We are going to show that it implies \eqref{GNS_eta}:
$$\langle \eta(p_{ij}), \eta (w) \rangle = \langle \eta (w), \eta(p_{ij})
\rangle.$$
Note that it is enough to check the latter relation for
$\eta(p_{ij})\in \{\xi, \zeta\}$.

\medskip Observe that $p_{11} \sim 1-p_{12}$. For example
$L(p_{11})=L(1-p_{12})$ since $L(1)=0$ and $L(p_{11})=-\|\xi\|^2 =-L(p_{12})$.
Similarly,
$$p_{33} \sim 1-p_{34}$$
and
$$p_{11} \sim p_{22}, \quad p_{12} \sim p_{21}, \quad p_{33} \sim p_{44},
\quad p_{34} \sim p_{43}.$$
Moreover, if $\rho(p_{ij})=0$ then $p_{ij} \sim 0.$

As mentioned in the proof of Lemma \eqref{lem_gns_handy}, the value of $\eta $
on equivalent elements is equal: $p_{ij}\sim p_{kl}$ $\Rightarrow$ $\eta(p_{ij})=
\eta(p_{kl})$. Thus it is enough to consider the condition \eqref{GNS_eta} for
these words which are build up from $p_{11}$ and $p_{33}$ (all the other
elements can be replaced by a linear combination of these two or by 0).
Moreover, after such a replacement and the concatenation what we get is a
linear combination of words of the form $$ w=p_{33}^m (p_{11}p_{33})^k
p_{11}^n$$ for some $k\in \mathbb{N}$ and $m,n\in \{0,1\}$.

\medskip 
In fact we have already checked in Step 1 that \eqref{GNS_eta} holds for
$w_1=(p_{11}p_{33})^kp_{11}$. For the word $w_2=p_{33}(p_{11}p_{33})^k $  we
have
$$ \eta(w_2)= \sum_{l=0}^{k} (QP)^l \zeta + \sum_{l=0}^{k-1} (QP)^l Q\xi$$
and thus
\begin{eqnarray*}
 \lefteqn{ \langle \eta(p_{ij}), \eta (w_2) \rangle - \langle \eta (w_2),
\eta(p_{ij}) \rangle }
\\ &=& \sum_{l=0}^{k} \left[ \langle \eta(p_{ij}), (QP)^l \zeta \rangle -
\langle (QP)^l \zeta,\eta(p_{ij}) \rangle \right]
\\ & & +\
\sum_{l=0}^{k-1} \left[\langle \eta(p_{ij}), (QP)^l Q\xi \rangle
-\langle (QP)^l Q\xi,\eta(p_{ij}) \rangle \right].
\end{eqnarray*}
If $\eta(p_{ij})=\xi$, then the first summand becomes zero because of
\eqref{eq_first}, and so does the second. So $\langle \xi, \eta (w_2) \rangle = \langle
\eta (w_2), \xi \rangle$. If $\eta(p_{ij})=\zeta$, then $\langle \zeta, \eta
(w_2) \rangle = \langle \eta (w_2), \zeta \rangle$, because $Q\zeta=0$.

Now let us consider $w_3=(p_{11}p_{33})^k$.
Then
$$ \eta(w_3)= \sum_{l=0}^{k-1} (PQ)^l (\xi + P\zeta)$$
and Eq. \eqref{eq_first} is equivalent to
\begin{eqnarray*}
 \lefteqn{
 \langle \eta(p_{ij}), \eta (w_3) \rangle - \langle \eta (w_3),
\eta(p_{ij}) \rangle}\\
&=&
\sum_{l=0}^{k-1} \left[ \big\langle \eta(p_{ij}), (PQ)^l \xi\big\rangle
-\big\langle (PQ)^l \xi, \eta(p_{ij})\big\rangle \right] \\
& & - \
\sum_{l=0}^{k-1} \left[ \big\langle (PQ)^l P\zeta, \eta(p_{ij})\big\rangle
- \big\langle \eta(p_{ij}), (PQ)^l P\zeta \big\rangle \right] .
\end{eqnarray*}
When $\eta(p_{ij})=\xi$, then, as $P\xi=0$,  the first term becomes $\langle
\xi,\xi\rangle -\langle \xi,\xi\rangle =0$ and each factor in the second
sum is $\langle (PQ)^l P\zeta, \xi\big\rangle = \langle (QP)^l\zeta,
P\xi\big\rangle=0$. When $\eta(p_{ij})=\zeta$, then the first sum is equal to 0
because of \eqref{eq_first} and the second sum is equal to 0, as
$Q\zeta=0$. So $\langle \eta(p_{ij}), \eta (w_3) \rangle = \langle \eta (w_3),
\eta(p_{ij}) \rangle $.

Finally, by similar calculations we show that $\langle
\eta(p_{ij}), \eta (w_4) \rangle=$ \linebreak$ \langle \eta (w_4), \eta(p_{ij}) \rangle $
for $w_4=(p_{33}p_{11})^k$. This finishes the proof.
\end{proof}

Theorem \ref{thm_charact} translates the symmetry of the generating functional related to the two-block representation \eqref{pq-rep} into a scalar-product condition on the vectors $\xi$ and $\zeta$ in the Hilbert space. It would be interesting to find a geometric interpretation of the condition \eqref{eq_first}.

We finish this subsection by some remarks which follow easily from Theorem \ref{thm_charact}.

\begin{remark}
If $PQ=QP$, then the related generating functional is symmetric provided
$$ \langle \xi, \zeta \rangle = \langle \zeta, \xi \rangle.$$
\end{remark}

\begin{remark}
The generating functional $L_{\xi,\zeta}$ is always symmetric,
independently on the choice of projections $P$ and $Q$, in each of the following
two cases:
(1) if $\xi=\zeta$
or
(2) if either $\xi$ or $\zeta$ is 0.
\end{remark}

\begin{remark}
There exists an generator which is a coboundary, but is not symmetric.
To see that, consider $\hil=\C^3$, $P$ the orthogonal projection onto the span of
$e_1=(1,0,0)$ and $e_2=(0,1,0)$, and $Q$ the orthogonal projection onto the
span of $f=\frac1{\sqrt3}(1,1,1)$. Let us consider the vector $v=(1,0,i)\in \mathsf{h}$ and the cocycle given by the conditions $\xi=\eta(p_{11})=(I-P)v=(0,0,i)$ and
$\zeta=\eta(p_{22})=(I-Q)v=(\frac23-\frac13i,
-\frac13-\frac13i,\frac13-\frac23i)$. Then obviously $\langle \xi,\zeta\rangle$
is complex and thus the related generating functional cannot be
symmetric.
\end{remark}

\subsection{Description of ad-invariant functionals}

In \cite{CFK14} the notion of ad-invariant generating functionals has been studied and the classification of all such objects has been given in the case of the free orthogonal quantum group $O_n^+$. In this section we establish a  similar classification result for $S_n^+$.

For the quantum permutation group $S_n^+$ and $n\ge 4$, the algebra of central functions ${\rm Pol}_0(S_n^+)$, see Section \ref{sec-markov}, is isomorphic to the $*$-algebra of polynomials on the interval $[0,n]$, cf.\ \cite[Proposition 4.8]{brannan13},
via the mapping $\chi_s \mapsto ([0,n]\ni x\mapsto U_{2s}(\sqrt{x}))$, where $(U_s)_{s\in\N}$ is the
family of Chebyshev polynomials of the second kind, defined inductively by $U_0 = 1$, $U_1 (x)= x$
and $U_{s+1}(x) =xU_s(x) - U_{s-1}(x)$ for $s\ge 1$. Furthermore, the restriction
of the counit to this subalgebra is the evaluation of a function in the right boundary point $x=n$.

In \cite[Proposition 10.1]{CFK14} we showed that linear functionals on the algebra of polynomials on a closed interval that are conditionally positive w.r.t.\ evaluation in one of the boundary points can be expressed as a sum of the derivative evaluated in the boundary point and an integral over the interval against a positive measure. In our situation this means that for $n\ge 4$ any conditionally positive functional on ${\rm Pol}(S_n^+)_0$ is of the form
\[
Lf = -a f'(n) + \int_0^n \frac{f(x)-f(n)}{n-x} {\rm d}\nu
\]
with a uniquely determined real number $a\ge 0$ and a uniquely determined finite positive measure $\nu$ supported on $[0,N)$, if we identify ${\rm Pol}_0(S_n^+)$ with $C([0,n])$ in the way described above.

The conditional expectation ${\rm ad}_h$ defined in Theorem \ref{thm-ad-inv} acts on the coefficients of the irreducible corepresentations of ${\rm Pol}(S_n^+)$ as
\[
{\rm ad}_h (u_{jk}^{(s)}) = \frac{\delta_{jk}}{n_s} \chi_s
\]
for $s \in\mathcal{I}$, $1\le j,k\le n_s$, see \cite[Formula (6.2)]{CFK14}. The dimension of the representation $u^{(s)}$ is given by $n_s = \varepsilon(\chi_s)=U_{2s}(\sqrt{n})$. This allows to describe all {\rm ad}-invariant generating functionals on ${\rm Pol}(S_n^+)$.

\begin{theorem}\label{thm-hunt-sn+}
\begin{description}
\item[(a)]
The ad-invariant generating functionals on ${\rm Pol}(S_n^+)$ are of the form
\[
\hat{L}=L\circ {\rm ad}_h
\]
with $L$ defined on ${\rm Pol}(S_n^+)_0\cong {\rm Pol}([0,n])$ by
\[
Lf = -a f'(n) + \int_{0}^n \frac{f(x) - f(n)}{n-x} {\rm d}\nu(x)
\]
where $a>0$ is a real number and $\nu$ is a finite measure on $[0,n]$. Furthermore, $a$ and $\nu$ are uniquely determined by $L$.
\item[(b)]
On the coefficients of the irreducible corepresentations of ${\rm Pol}(S_n^+)$ we get
\[
\hat{L}(u^{(s)}_{jk}) = \frac{\delta_{jk}}{U_{2s}(\sqrt{n})} \left(-a \frac{U'_{2s}(\sqrt{n})}{2\sqrt{n}} + \int_0^n \frac{U_{2s}(\sqrt{x})-U_{2s}(\sqrt{n})}{n-x}{\rm d}\nu(x)\right)
\]
for $s\in\mathcal{I}$, $1\le j,k\le n_s$, where $U_s$ denotes the $s$-th Chebyshev polynomial of the second kind.
\end{description}
\end{theorem}

\section{L\'evy processes on $S_n^+$ coming from classical permutations} \label{classicalstoch}

In this section we will construct examples of noncommutative L\'evy processes on the quantum permutation group $S_n^+$. This examples have some similarity with L\'evy processes in the classical permutation group, furthermore they can be constructed explicitly from the conservation, creation, and annihilation operators on the symmetric Fock space. It will turn out that the L\'evy processes on $S_n^+$ described in this Section are exactly those arising from the representation $\rho_\sigma$ of ${\rm Pol}(S_n^+)$ associated to a classical permutation $\sigma\in S_n$ and its cocycles, which we classified in Proposition \ref{prop-coh-perm}.

\subsection{L\'evy processes on finite groups}\label{subsec-fin-groups}
Let us first translate a few well-known results about L\'evy processes on finite groups into the language of L\'evy processes on $*$-bialgebras (see for example \cite{liggett10} for a good description of the classical theory).

A stationary Markov chain $(X_t)_{t\ge 0}$ in continuous time on a finite group $G$ is uniquely determined by its initial distribution $\mu_0=\mathcal{L}(X_t)$ (i.e., $\mu_0(B)=P(X_0\in B)$ for $B\subseteq G$) and its Q-matrix, defined by
\[
Q_{g,h} = \lim_{t\searrow 0} \frac{1}{t} P(X_{s+t}=h|X_s=g)
\]
for $g,h\in G$. Since necessarily $\sum_{h\in G}  P(X_{s+t}=h|X_s=g)=1$, we have $Q_{g,g} = - \sum_{h\in G; h\not= g} Q_{g,h}$.

If $(X_t)_{t\ge 0}$ is a L\'evy process, then we have $\mu_0=\delta_e$ and the infinitesimal transition probabilities are translation invariant, i.e., we have $Q_{h,k}=Q_{gh,gk}$ for all $g,h,k\in G$, and $Q$ is uniquely determined by the row of the identity element $e$. The formula $\nu(B) = \sum_{g\in B} Q_{e,g}$  (for $B \subset G$) defines a measure $\nu$ on the group $G$, to which we can associate a functional $L:C(G)\to\mathbb{C}$. It is easy to see that $L$ is a generating functional, i.e., it satisfies the properties stated in Definition \ref{generating_functional_def}. Since $G$ is finite, all functions on $G$ are continuous and are coefficients of a finite-dimensional representation, so that we have $\mathbb{C}^G=C(G)={\rm Pol}(G)$, and this is naturally a Hopf$^*$-algebra with the coproduct induced from the group operation. The counit on $C(G)$ is the evaluation at the identity, i.e., $\varepsilon(f)=f(e)$ for $f\in C(G)$.

If $(\Omega,\mathcal{F},P)$ denotes the classical probability space over which the L\'evy process $(X_t)_{t\ge 0}$ is defined, then we can define a L\'evy process on $C(G)$ in the sense of Definition \ref{Levy_process_on_A} by
\[
j_{st}: C(G)\to L^\infty(\Omega), \qquad j_{st}(f) = f(X_s^{-1}X_t)
\]
for $f\in C(G)$, $0\le s\le t$. The functional $L$ is then the generating functional of the L\'evy process $(j_{st})_{0\le s\le t}$. Indeed, we have
\begin{eqnarray*}
\left.\frac{{\rm d}}{{\rm d}t}\right|_{t=0}\omega_t(f) &=& \left.\frac{{\rm d}}{{\rm d}t}\right|_{t=0}E(j_{0,t}(f)) =
\left.\frac{{\rm d}}{{\rm d}t}\right|_{t=0}E\big(f(X_0^{-1}X_{t})\big) \\
&=& \left.\frac{{\rm d}}{{\rm d}t}\right|_{t=0} \sum_{g,h} f(g^{-1}h) P(X_0=g) P(X_t=h|X_0=g) \\
&=&\left.\frac{{\rm d}}{{\rm d}t}\right|_{t=0} \sum_{g,h} f(h) P(X_t=h|X_0=e) \\
&=& \sum_{h\in G} f(h)  \underbrace{\left.\frac{{\rm d}}{{\rm d}t}\right|_{t=0}P(X_t=h|X_0=e)}_{= \nu(\{h\})}  = \int_G f(h){\rm d}\nu(h).
\end{eqnarray*}

Let us compute the Sch\"urmann triple for $L$. Let
\[
S= \{g\in G; \nu(g)>0, g \neq e\}
\]
be the support of $\nu$ (without the identity $e$), then we take
\[
D=L^2(S,\nu|_{S})
\]
for the pre-Hilbert space. The representation $\rho:C(G)\to \mathcal{L}(D)$ sends a function $f\in C(G)$ to the operator that multiplies pointwise by $f$, i.e.,
\[
\big(\rho(f)\psi)(g) = f(g)\psi(g)
\]
for $\psi\in D$, $g\in G$. The cocycle $\eta:C(G)\to D$ is given by
\[
\eta(f) = f|_S - f(e)\mathbf{1}_S.
\]
Indeed, with this definition we have
\begin{eqnarray*}
\langle \eta(f_1),\eta(f_2)\rangle &=& \sum_{g\in S} \overline{\big(f_1(g)-f_1(e)\big)} \big(f_2(g)-f_2(e)\big)\nu(\{g\}) \\
&=& \sum_{g\in G} \overline{\big(f_1(g)-f_1(e)\big)} \big(f_2(g)-f_2(e)\big)\nu(\{g\}) \\
&=&
\sum_{g\in G} \overline{f_1(g)} f_2(g)\nu(\{g\}) - \overline{f_1(e)} \sum_{g\in G} f_2(g)\nu(\{g\}) \\
&& - \sum_{g\in G} \overline{f_1(g)}\nu(\{g\}) f_2(e) + \overline{f_1(e)}f_2(e) \sum_{g\in G} \nu(\{g\}) \\
&=& L(f_1^* f_2) - \overline{\varepsilon(f_1)} L(f_2) - \overline{L(f_1)} \varepsilon(f_2),
\end{eqnarray*}
since $\sum_{g\in G}\nu(\{g\}) = \nu(G) =0$.

We leave the rest of the verification that $(\rho,\eta,L)$ is indeed a surjective Sch\"urmann triple for $L$ as an exercise for the motivated reader.

\subsection{Examples of L\'evy processes on classical permutation groups}
\label{subsec-class-exam}

Let us now take $G=S_n$ and consider the case where $\nu=\alpha(\delta_\sigma-\delta_e)$ with $\alpha>0$, i.e., the ``support'' of $\nu$ (without the identity $e$) consists only of one point $\sigma\in S_n$. In this case there is only one possible jump of the associated process from a permutation $\tau$, namely to $\tau\circ\sigma$, and therefore the process will stay in the subgroup $\langle \sigma\rangle$ generated by $\sigma$.  Then the L\'evy process $(X_t)_{t\le 0}$ can easily be described. It starts at $X_0=e$ and stays there for a random time that is exponentially distributed with parameter $\alpha$. When the first jump occurs, the process moves necessarily to $\sigma$. The next jump occurs again after an exponentially distributed random time, then the value becomes $\sigma^2$, and so on. We can write $X_t = \sigma^{N_t}$, where $(N_t)_{t\ge 0}$ is a Poisson process with intensity $\alpha$.

Let $\sigma=c_1,\cdots, c_r$ be the decomposition of $\sigma$ into a product of disjoint cycles, and denote by $\ell_1,\cdots,\ell_r$ their lengths. Then we have $c_j^{\ell_j}=e$. Since disjoint cycles commute, we get
\[
\sigma^k = c_1^k \cdots c_r^k = c_1^{k_1} \cdots c_r^{k_r}
\]
if $k\equiv_{\textup{mod}\ell_1} k_1$, $\ldots$, $k\equiv_{\textup{mod}\ell_r} k_r$. Therefore we can identify the elements of $\langle \sigma \rangle = \{\sigma^k;k\in \mathbb{Z}\}$ with elements in a product of cyclic groups of orders $\ell_1,\ldots,\ell_r$, via the map
\[
\langle \sigma\rangle \ni \sigma^k =  c_1^{k_1} \cdots c_r^{k_1} \mapsto (k_1,\ldots,k_r)\in \mathbb{Z}_{\ell_1}\times \cdots \times \mathbb{Z}_{\ell_r}.
\]
In general, the image of $\langle \sigma \rangle$ under this map might be a strict subgroup of $\mathbb{Z}_{\ell_1}\times \cdots \times \mathbb{Z}_{\ell_r}$, but if $\ell_1,\cdots,\ell_r$ are mutually prime, then this map is a bijection.

Let us now assume for simplicity that $\sigma$ consists of only one cycle. We can assume that this cycle is $\sigma=(123\cdots n1)$ (if not, we can renumber the generators, similarly as was done in Proposition \ref{prop-equiv-H}). Then we have $\sigma^k(j) \equiv_{\textup{mod}n} j+k$.

Denote by $\mathbf{1}_{i\to j}\in C(S_n)$ the function
\[
\mathbf{1}_{i\to j}(\pi) = \left\{\begin{array}{cl} 1 & \mbox{ if }\pi(i)=j, \\ 0 & \mbox{ else},\end{array}\right.
\]
then we have
\[
j_{st}(\mathbf{1}_{i\to j}) = \mathbf{1}_{\{N_t-N_s\equiv_{\textup{mod}n} j-i\}}.
\]
This defines the L\'evy process $(j_{st})_{0\le s\le t}$ in $C(S_n)$ that is associated to the $S_n$-valued L\'evy process $(X_t)_{t\ge 0}$, since the functions $\mathbf{1}_{i\to j}$ with $1\le i,j\le n$ generate $C(G)$ as an algebra.

\subsection{Realization of Poisson processes on the symmetric Fock space}
\label{subsec-real-Fock-space}

We will briefly recall some facts about the realisation of Poisson processes by operators on the symmetric Fock space, which we will use in the next subsection to construct examples of noncommutative L\'evy processes on the quantum permutation group. See, e.g., \cite{meyer93,parthasarathy92} for a more detailed exposition. We will also use some formulas from \cite{quasi-inv} for our computations. For a Hilbert space $\mathsf{h}$ we write
\[
\Gamma_\hil=\Gamma\big(L^2(\mathbb{R}_+;\hil)\big) = \mathbb{C}\Omega\oplus \bigoplus_{n=1}^\infty L^2(\mathbb{R}_+;\hil)^{\circ n}
\]
for the symmetric Hilbert space over $L^2(\mathbb{R}_+;\hil)\cong L^2(\mathbb{R}_+)\otimes \hil$.

Note that the direct sum decomposition
\[
L^2(\mathbb{R}_+;\hil)\cong L^2([0,s[;\hil)\oplus L^2([s,t[;\hil)\oplus L^2([t,+\infty[;\hil)
\]
yields the following tensor product decomposition for the symmetric Fock space:
\[
\Gamma_{\mathbb{R}_+}(\hil)=\Gamma\big(L^2([0,s[;\hil)\big)\otimes\underbrace{\Gamma\big(L^2([s,t[;\hil)\big)}_{=:\Gamma_{[s,t[}(\hil)}\otimes\Gamma\big(L^2([t,+\infty[;\hil)\big)
\]
The functional $\Phi:B(\Gamma_{\mathbb{R}_+}(\hil))\to\mathbb{C}$, $\Phi(X)=\langle\Omega,X\Omega\rangle$ is called the \emph{vacuum expectation}, we can also apply it to unbounded operators on $\Gamma_{\mathbb{R}_+}(\hil)$, if their domain contains the vector $\Omega$.

For $0\le s\le t$, $K\in B(\hil)$, $v\in \hil$, we denote by $\Lambda_{st}(K)$, $A_{st}(v)$, $A^*_{st}(v)$ the conservation (or differential second quantization), annihilation, and creation operators of $\mathbf{1}_{[s,t[}\otimes K\in B\big(L^2(\mathbb{R}_+;\hil)\big)$ and $\mathbf{1}_{[s,t[}\otimes v\in L^2(\mathbb{R}_+;\hil)$. Denote by $I_\hil$ and $I_\Gamma$ the identity operator on $\hil$ and $\Gamma_{\mathbb{R}_+}(\hil)$, respectively, and let $v\in \hil$. One can show that the operator
\[
N_{st}^v = \Lambda_{st}(I_\hil) + A^*_{st}(v) + A_{st}(v) + \|v\|^2 I_\Gamma
\]
is essentially self-adjoint and distributed (with respect to $\Phi$) like the increment of a Poisson process with intensity $\|v\|^2$, i.e., we have
\[
\Phi\big(\exp( i u N_{st}^v)\big) = \exp\big(-(t-s)\|v\|^2 (e^{iu}-1)\big)
\]
for its characteristic function.

If we further introduce the Weyl operators $W_{st}(u)=\exp\big(A^*_{st}(u)-A_{st}(u)\big)$ for $u\in \hil$, then we can show
\begin{equation}\label{eq-weyl-conj}
W_{st}(u)^*N^v_{st}W_{st}(u) = N^{v+u}_{st}
\end{equation}
as in \cite[Equation (4.1)]{quasi-inv}.

The conservation operator $\Lambda_{st}(I_\hil)$ of the identity operator $I_\hil$, called also the \emph{number operator}, has spectrum equal to $\mathbb{N}_0=\{0,1,\cdots\}$. Its eigenspaces are the finite symmetric tensor products, i.e., the eigenspace for an eigenvalue $n\in\mathbb{N}_0$ of $\Lambda_{st}(I_\hil)$ is given by
\[
E_n\big(\Lambda_{st}(I_\hil)\big) = \Gamma_{[0,s[}(\hil) \otimes L^2(\mathbb{R}_+;\hil)^{\circ n} \otimes \Gamma_{[t,+\infty[}(\hil).
\]

Since $N^0_{st}=\Lambda_{st}(I_\hil)$, we can use Equation \eqref{eq-weyl-conj} to determine the eigenspaces of all $N^v_{st}$. They are given by
\begin{equation}\label{eq-eigensp}
E^v_{st}(n) := E_n\big(\Lambda_{st}(I_\hil)\big) = \Gamma_{[0,s[}(\hil) \otimes W_{st}(v)^*\big( L^2([s,t[;\hil)^{\circ n}\big) \otimes \Gamma_{[t,+\infty[}(\hil)
\end{equation}
We will denote the projections onto these spaces by $P_{st}^v(n)$. For fixed $v$ we get commutative families of projections, as we can obtain $P_{st}^v(n)$ from $N_t^v-N_s^v$ via Borel functional calculus, namely as $P_{st}^v(n)=\mathbf{1}_{\{n\}}(N_t^v-N_s^v)$, and because $(N_t^v)_{t\ge 0}$ is a commuting family of operators (remember, it is a realization of a Poisson process with intensity $\|v\|^2$).

Let us collect here a few useful formulas about the Weyl operators. The exponential vectors $\mathcal{E}(f)$ for $f\in L^2(\mathbb{R}_+;\hil)$ are defined by
\[
\mathcal{E}(f) = \sum_{n=0}^\infty \frac{f^{\circ n}}{\sqrt{n!}},
\]
their inner products are given by $\langle \mathcal{E}(f_1),\mathcal{E}(f_2)\rangle = \exp(\langle f_1,f_2\rangle)$. Weyl operators act on exponential vectors as
\[
W_{st}(u)\mathcal{E}(f) = \exp\big(-\langle u\mathbf{1}_{[s,t[},f\rangle - \frac{t-s}{2}\|u\|\big)\mathcal{E}(f+u\mathbf{1}_{[s,t[})
\]
for $u\in \hil$, $f\in L^2(\mathbb{R}_+;\hil)$, see, e.g., \cite[Equation (20.2)]{parthasarathy92}. The product of two Weyl operators is given by
\[
W_{st}(u) W_{st}(v) = e^{-i(t-s){\rm Im}(\langle u,v\rangle)} W(u+v),
\]
cf.\ \cite[Proposition 20.1]{parthasarathy92}.

\subsection{Examples of L\'evy processes on quantum permutation groups}

We will now put all the ingredients from the previous subsection together to define examples of noncommutative L\'evy processes on the quantum permutation group $S_n^+$.

The classical permutation group is a quantum subgroup of the quantum permutation group, the morphism $\pi:C(S_n^+) \to C(S_n)$ is determined by its action
\[
\pi(p_{ij}) = \mathbf{1}_{i\to j}
\]
on the generators. Any L\'evy process $(j_{st})_{0\le s\le t}$ on $C(S_n)$ gives a L\'evy process $(\tilde{\jmath}_{st})_{0\le s\le t}$ on $C(S_n^+)$ simply by setting $\tilde{\jmath}_{st} = j_{st}\circ \pi$. If we choose a Hilbert space $\hil$ and a unit vector $v\in \hil$, then we can realise the L\'evy process associated to $X_t=\sigma^{N_t}$ with $\sigma=(12\cdots n1)$ also as
\[
\tilde{\jmath}_{st} (p_{ij}) = \sum_{k\in\mathbb{N},~k\equiv_{\textup{mod}n}j-i} P_{st}^v(k).
\]
It is clear from the remarks at the end of the last subsection that this is still a commutative L\'evy process, i.e., the images $\big(j_{st}(C(S_n^+))\big)_{0\le s\le t}$ generate a commutative algebra.

But if we do this construction with a permutation $\sigma$ that is not a cycle, then we can actually choose different realizations of the Poisson process for different cycles. For simplicity we consider the case of two cycles $\sigma = (12\cdots \ell1)((\ell+1)\cdots n(\ell+1))$, with $1<\ell<n$. Then, for any Hilbert space $\hil$ and two vector $v,w\in \hil$, the following defines a L\'evy process on $C(S_n^+)$:
\begin{equation}\label{eq-def-example}
j_{st} (p_{ij}) = \left\{\begin{array}{cl}
\displaystyle \sum_{k\in\mathbb{N},~k\equiv_{\textup{mod}\ell}j-i} P_{st}^v(k) & \mbox{ if } 1\le i,j\le \ell, \\
\displaystyle \sum_{k\in\mathbb{N},~k\equiv_{\textup{mod}n-\ell}j-i} P_{st}^w(k) & \mbox{ if } \ell+1 \le i,j\le n, \\
0 & \mbox{ else}\end{array}\right.
\end{equation}
Let us describe the Sch\"urmann triple of this L\'evy process.

For $1\le i,j\le \ell$, we have
\begin{eqnarray*}
\Phi\big(j_{st}(p_{ij})\big) &=& \langle \Omega, j_{st}(p_{ij}) \Omega\rangle = \left\langle \Omega,\sum_{k\equiv_{\textup{mod}\ell}j-i} P_{st}^v(k) \Omega\right\rangle \\
&=& \left\langle W_{st}(v)\Omega, \sum_{k\equiv_{\textup{mod}\ell}j-i} P_{st}^0(k)W_{st}(v)\Omega\right\rangle \\
&=& e^{-(t-s)\|v\|^2} \langle\mathcal{E}(v\mathbf{1}_{[s,t[}), \sum_{k\equiv_{\textup{mod}\ell}j-i} P_{st}^0(k) \mathcal{E}(v\mathbf{1}_{[s,t[})\rangle \\
&=& e^{-(t-s)\|v\|^2} \sum_{k\equiv_{\textup{mod}\ell}j-i} \frac{(t-s)\|v\|^2}{k!}.
\end{eqnarray*}
This is the probability with which the number of jumps between time $s$ and time $t$ of Poisson process with intensity $\|v\|^2$ is equal to $j-i$ modulo $\ell$. I.e., we have shown that
\[
\Phi\big(j_{st}(p_{ij})\big) = E(\mathbf{1}_{\{N_t-N_s\equiv_{\textup{mod}\ell} j-i\}}),
\]
for $1\le i,j\le \ell$, if $(N_t)_{t\ge 0}$ is a classical Poisson process with intensity $\|v\|^2$. In particular
\[
L(p_{ij}) = \left\{\begin{array}{cl}
-\|v\|^2 & \mbox{ if } j=i, \\
\|v\|^2 & \mbox{ if } j\equiv_{\textup{mod}\ell}i+1, \\
0 & \mbox{ else }.
\end{array}\right.
\]
Similar, we find
\[
\Phi\big(j_{st}(p_{ij})\big) = E(\mathbf{1}_{\{N'_t-N'_s\equiv_{\textup{mod}n-\ell} j-i\}}),
\]
with $(N'_t)_{t\ge 0}$ a classical Poisson process with intensity $\|w\|^2$ and
\[
L(p_{ij}) = \left\{\begin{array}{cl}
-\|w\|^2 & \mbox{ if } j=i, \\
\|w\|^2 & \mbox{ if } j\equiv_{\textup{mod}\ell}i+1, \\
0 & \mbox{ else },
\end{array}\right.
\]
in the case $\ell+1\le i,j\le n$, and
\[
\Phi\big(j_{st}(p_{ij})\big) = 0, \qquad L(p_{ij}) = 0
\]
of $1\le i \le \ell < j \le n$ or $1\le j \le \ell < i \le n$.

It is now natural to expect that the restrictions to the subalgebras generated by $(p_{ij})_{1\le i,j\le \ell}$ and $(p_{ij})_{\ell+1\le i,j\le n}$ of the Sch\"urmann triple for this process are of the form of the Sch\"urmann triple of the L\'evy processes discussed in Subsections \ref{subsec-fin-groups} and \ref{subsec-class-exam}.

To be able to guess the Sch\"urmann triple for our process, we now want to describe $\Phi(j_{st}(p_{ij}p_{i'j'}))$ for small $t-s$ in the case $1\le i,j\le \ell < i',j'\le n$. We get
\begin{gather*}
\Phi(j_{st}(p_{ij}p_{i'j'})) = \langle j_{st}(p_{ij})\Omega, j_{st}(p_{i'j'})\Omega\rangle = \\
\left\langle W_{st}(v)^* \sum_{k\equiv_{\textup{mod}\ell}j-i} P_{st}^0(k) \mathcal{E}(v\mathbf{1}_{[s,t[}), W_{st}(w)^* \sum_{k'\equiv_{\textup{mod}n-\ell}j'-i'} P_{st}^0(k') \mathcal{E}(w\mathbf{1}_{[s,t[})\right\rangle \\
\end{gather*}
Assuming furthermore $i=j$ and $i'=j'$, and neglecting terms in higher order tensor products (since they will lead higher powers of $t-s$), we get
\begin{eqnarray*}
\Phi(j_{st}(p_{ij}p_{i'j'})) &\approx& e^{-(t-s)(\|v\|^2+\|w\|)^2/2} \langle \underbrace{W_{st}(v)^*}_{=W_{st}(-v)}\Omega, \underbrace{W_{st}(w)^*}_{=W_{st}(-w)} \Omega\rangle \\
&=& e^{-(t-s)(\|v\|^2+\|w\|)^2}\langle \mathcal{E}(-v\mathbf{1}_{[s,t[}), \mathcal{E}(-w\mathbf{1}_{[s,t[})\rangle \\
&=& e^{-(t-s)(\|v\|^2+\|w\|)^2 }\exp \big((t-s)\langle v,w\rangle\big),
\end{eqnarray*}
from which we can deduce
\[
L(p_{ii}p_{i'i'}) = \langle v,w\rangle- \|v\|^2 - \|w\|^2.
\]
This implies
\begin{eqnarray*}
\langle \eta(p_{ii}),\eta(p_{i'i'})\rangle &=& - \varepsilon(p_{ii}) L(p_{i'i'}) + L(p_{ii}p_{i'i'}) - L(p_{ii})\varepsilon(p_{i'i'}) \\
&=& \langle v,w\rangle.
\end{eqnarray*}
This allows us to claim the form of a Sch\"urmann triple for the L\'evy process defined in Equation \eqref{eq-def-example}: the representation is given by $\rho=\rho_\sigma I_{\hil}$ and the cocycle is the unique cocycle with $\eta(p_{ii})=v$ for $1\le i\le \ell$ and $\eta(p_{ii})=w$ for $\ell+1\le i\le n$. The verification of this claim is again an exercise for the motivated reader.

This construction can be generalised to an arbitrary permutation $\sigma$. We choose a Hilbert space $\hil$, take again $\rho=\rho_\sigma  I_{\hil}$, and choose one vector $v_c$ for each cycle $c$ of $\sigma$ which is not a fixed point. We set $\xi_i=v_c$ if $i$ belongs to the cycle $c$, and $\xi_i=0$ if $i$ is a fixed point of $\sigma$, then these vectors define a unique Sch\"urmann triple by Corollary \ref{cor_class}. The freedom we have for choosing these vectors corresponds exactly to our computation of the cohomology of the representation of $C(S_n^+)$ associated to a classical permutation, see Proposition \ref{prop-coh-perm}. If these vectors are not all identical, then they define a cocycle that is not a coboundary, and so they define an unbounded generating functional, by Theorem \ref{thm-unbounded}. The associated L\'evy process can be constructed using Weyl operators and noncommutating realisations of the Poisson process on the symmetric Fock space over $L^2(\mathbb{R}_+,\hil)$ in a fashion analogous to Equation \eqref{eq-def-example}.

\subsection*{Acknowledgement}
We thank Irene Sabadini, Fabio Cipriani and Fabrizio Colombo for their
hospitatility and for organizing the wonderful conference on ``Noncommutative Analysis, Operator Theory and Applications.'' We also thank IMPAN  for hosting a visit of UF and AK to Warsaw, which started this collaboration.


\end{document}